\documentclass[12pt]{amsart}
\usepackage{amsmath, amssymb, amsfonts}
\usepackage[all]{xypic}
\usepackage[pdftex]{graphicx}
\usepackage{stmaryrd}

\theoremstyle{plain}
\newtheorem{theorem}{Theorem}[section]
\newtheorem{lemma}[theorem]{Lemma}
\newtheorem{corollary}[theorem]{Corollary}
\newtheorem{prop}[theorem]{Proposition}

\theoremstyle{remark}

\newtheorem{remark}[theorem]{Remark}

\newtheorem{example}[theorem]{Example}

\newtheorem{conj}[theorem]{Conjecture}
\newtheorem*{note*}{Note}
\newtheorem*{remark*}{Remark}
\newtheorem*{example*}{Example}

\theoremstyle{definition}
\newtheorem*{definition*}{Definition}
\newtheorem{definition}[theorem]{Definition}

\newtheorem*{hypothesis*}{Hypothesis}

\newtheorem{theoremA}{Theorem}

\newcommand{\Z}{\mathbb{Z}}
\newcommand{\R}{\mathbb{R}}
\newcommand{\Q}{\mathbb{Q}}
\newcommand{\C}{\mathbb{C}}
\newcommand{\N}{\mathbb{N}}
\newcommand{\F}{\mathbb{F}}
\newcommand{\Irr}{\mathrm{Irr}}
\newcommand{\Ann}{\mathrm{Ann}}

\newcommand{\Gal}{\mathrm{Gal}}
\newcommand{\Tr}{\mathrm{Tr}}
\newcommand{\GL}{\mathrm{GL}}

\newcommand{\cl}{\mathrm{cl}}
\newcommand{\Hyp}{\mathrm{Hyp}}

\newcommand{\ind}{\mathrm{ind}}
\newcommand{\res}{\mathrm{res}}
\newcommand{\infl}{\mathrm{infl}}

\newcommand{\nr}{\mathrm{nr}}

\newcommand{\Hom}{\mathrm{Hom}}
\newcommand{\Maps}{\mathrm{Maps}}
\newcommand{\ram}{\mathrm{ram}}
\newcommand{\et}{\mathrm{\acute{e}t}}
\newcommand{\Spec}{\mathrm{Spec}}

\newcommand{\Fitt}{\mathrm{Fitt}}

\newcommand{\cone}{\mathrm{cone}}
\newcommand{\half}{{\textstyle \frac{1}{2}}}
\newcommand{\PMod}{\mathrm{PMod}}
\newcommand{\perf}{\mathrm{perf}}

\newcommand{\quot}{\mathrm{quot}}
\newcommand{\rad}{\mathrm{rad}}
 
\newcommand{\squarematzwei}[4]{
 	\begin{pmatrix}
 		{#1} & {#2} \\
 		{#3} & {#4} 
 	\end{pmatrix}
}

\numberwithin{equation}{section}

\title[Integrality of Stickelberger elements]{Integrality of Stickelberger elements\\ 
and annihilation of natural Galois modules}

\author{Nils Ellerbrock}
\address[Ellerbrock]{Universit\"{a}t Duisburg--Essen\\
	Fakult\"{a}t f\"{u}r Mathematik\\
	Thea-Leymann-Str.\ 9\\
	45127 Essen\\
	Germany}
\email{nils.ellerbrock@uni-due.de}

\author{Andreas Nickel}
\address[Nickel]{Universit\"{a}t der Bundeswehr M\"{u}nchen\\
	Fakult\"{a}t f\"{u}r Informatik\\
	Wer\-ner-Hei\-sen\-berg-Weg 39\\
	85579 Neubiberg\\
	Germany}
\email{andreas.nickel@unibw.de}
\urladdr{https://www.unibw.de/timor/mitarbeiter/univ-prof-dr-andreas-nickel}

\date{Version of 28th November 2024}

\begin{document}

\begin{abstract}
	To each Galois extension $L/K$ of number fields with Galois group $G$
	and each integer $r \leq 0$
	one can associate Stickelberger elements in the centre of the 
	rational group ring $\mathbb{Q}[G]$ in terms of values of
	Artin $L$-series at $r$. We show that the denominators of
	their coefficients are bounded by the cardinality of the
	commutator subgroup $G'$ of $G$ whenever $G$ is nilpotent.
	Moreover, we show that, after multiplication by $|G'|$ and away
	from $2$-primary parts,
	they annihilate the class group of $L$ if $r=0$ and higher
	Quillen $K$-groups of the ring of integers in $L$ if $r<0$.
	This generalizes recent progress on conjectures of Brumer
	and of Coates and Sinnott from abelian to nilpotent extensions.
	
	For arbitrary $G$ we show that the denominators remain bounded along
	the cyclotomic $\mathbb{Z}_p$-tower of $L$ for every odd
	prime $p$. This allows us to give an affirmative answer to
	a question of Greenberg and of Gross on the behaviour of 
	$p$-adic Artin $L$-series at zero.
\end{abstract}

\maketitle

\section*{Introduction}

Let $K$ be a number field and let $L$ be a finite Galois extension of $K$
with Galois group $G = \Gal(L/K)$. 
Then $G$ acts naturally
upon the class group $\cl_L$ of $L$
so that we may view the class group as a module over
the integral group ring $\Z[G]$. Its structure is very hard to
determine in general and even constructing unconditional non-trivial
annihilators is very challenging. 

The class number obviously annihilates the class group and is related to
the leading term of the Laurent series expansion at $s=1$ 
(or equivalently at $s=0$ via the functional
equation) of the Dedekind zeta function attached to $L$
by means of the analytic class number formula.
It is therefore natural to consider (special) values of 
Artin $L$-series $L(s,\chi)$
at $s=0$ in order to construct interesting annihilators, where $\chi$
runs over the irreducible complex-valued characters of $G$.

More precisely, let $S$ and $T$ be two finite disjoint sets of places of $K$.
We always assume that $S$ contains all archimedean places and all
places that ramify in $L$. The set $T$ is chosen such that
no non-trivial root of unity in $L$ is congruent to $1$ modulo
every place of $L$ lying above one in $T$. 
The $S$-truncated equivariant Artin $L$-function $\Theta_S(s)$ is
a meromorphic function on the whole complex plane with values in
$\zeta(\C[G])$, where $\zeta(\Lambda)$ denotes the 
centre of a ring $\Lambda$. 
There is a natural isomorphism $\zeta(\C[G]) \simeq \prod_{\chi} \C$,
where the product runs over all complex valued irreducible characters of 
$G$. Then $\Theta_S(s)$ corresponds to the tuple $(L_S(s, \check\chi))_{\chi}$,
where $\check\chi$ denotes the character contragredient to $\chi$,
and $L_S(s,\chi)$ denotes the $S$-truncated Artin $L$-series
associated to $\chi$. The latter is obtained from $L(s,\chi)$
by removing the Euler factors corresponding to the primes in $S$.
The $T$-modified
$S$-truncated equivariant Artin $L$-function $\Theta_{S,T}(s)$ is
then obtained from $\Theta_S(s)$ 
by inserting certain 
modified Euler factors for the primes in $T$.
Let $r \leq 0$ be an integer. Then the equivariant $L$-values
$\theta_S^T(r) := \Theta_{S,T}(r)$ are called Stickelberger elements
and it follows from a result of 
Siegel \cite{MR0285488} that they indeed have rational coefficients.

If $G$ is abelian, then it was independently shown by
Pi.\ Cassou-Nogu\`es \cite{MR524276}, Deligne and Ribet \cite{MR579702}, and by Barsky \cite{MR525346} that one has
\begin{equation} \label{eqn:abelian-int}
	\theta_S^T(r) \in \Z[G]
\end{equation}
and Brumer's conjecture as discussed by Tate \cite{MR782485} simply
asserts that $\theta_{S}^T(0)$ annihilates the class group $\cl_L$.
In the case $K = \Q$ this is equivalent to Stickelberger's theorem from the
late 19th century \cite{MR1510649}. 

In the last decades, Brumer's conjecture has been attacked via Iwasawa theory,
starting with the seminal work of Wiles \cite{MR1053490}. The method
has been refined over the years by several authors 
\cite{MR1750935, MR3383600, MR2805422, MR3980291}, but in almost all cases
the results were dependent on the vanishing of the $\mu$-invariant of 
a certain Iwasawa module. The latter is conjectured to be true, but
this question is still wide open and so this only has led to conditional 
results. A different approach, which assumes the validity of the
relevant special case of the equivariant Tamagawa number conjecture,
is due to Greither \cite{MR2371374}.

Rather recently, ground-breaking work of Dasgupta and Kakde
\cite{dasgupta-kakde} has overcome this problem. They have shown
that Brumer's conjecture holds unconditionally away from its $2$-primary
part. They indeed proved refinements thereof, one of which is known as 
the Brumer--Stark conjecture. Roughly speaking, the latter replaces the
class group by a ray class group $\cl_L^T$ whose definition also involves the
set $T$. Their method is inspired by Wiles' proof of the main conjecture
for totally real fields \cite{MR1053488} and is a very clever 
variant of `Ribet's method' using group ring valued Hilbert modular forms.
The interest in the Brumer--Stark conjecture comes from its relation 
to Hilbert's 12th problem, and in \cite{Hilbert12} 
Dasgupta and Kakde indeed provide an effective method
to construct the maximal abelian extension of a totally real field.

Finally, we also want to mention some unconditional and computational results
on the $2$-primary part of the conjecture for extensions of exponent $2$
and $4$ due to Sands \cite{MR743968} 
and due to Roblot and Tangedal \cite{MR1850628}, respectively.
Extensions of degree $2p$ for odd primes $p$
have been considered in \cite{MR2034123}.

Coates and Sinnott \cite{MR0369322} have formulated an
analogue of Brumer's conjecture when $r<0$. It asserts that $\theta_S^T(r)$
annihilates the higher Quillen $K$-group $K_{-2r}(\mathcal{O}_{L,S})$
of the ring of $S(L)$-integers in $L$. Here $S(L)$ denotes the set of places
of $L$ lying above a place in $S$. Based on the work of Dasgupta
and Kakde, Johnston and the second named author \cite{abelian-MC} 
have shown that
the conjecture of Coates and Sinnott holds away from $2$-primary parts.
Assuming the vanishing of the relevant $\mu$-invariant, this was already
known by work of Burns and Greither \cite{MR2046598}
(see also \cite{MR3383600} for a different approach). 
The $2$-primary part has been considered by Kolster and Taleb
\cite{MR4120473}.\\

In this article we go a step further and no longer assume that
the Galois group $G$ is abelian. Generalizations of Brumer's conjecture
have been formulated by the second named author \cite{MR2976321}
and, independently and in even greater generality, by
Burns \cite{MR2845620}. A further, slightly different approach has been 
developed by Dejou and Roblot \cite{MR3208394}. Here we mainly
follow the treatment in the survey article \cite{nickel-conjectures}.

A main obstacle is the fact that Stickelberger elements do no longer
have integral coefficients. So our first task is to bound their
denominators. The theory of noncommutative Fitting invariants
\cite{MR2609173, MR3092262} has led to the definition of the so-called 
integrality ring $\mathcal{I}(G)$ and the denominator ideal
$\mathcal{H}(G)$ of the integral group ring $\Z[G]$. The former is
defined to be the smallest subring of $\zeta(\Q[G])$ that
contains $\zeta(\Z[G])$ and the reduced norms $\nr(H)$
of all square matrices $H$ with entries in $\Z[G]$.
The definition of the denominator ideal is even less explicit.
It is an ideal in $\mathcal{I}(G)$ such that
\[
	\mathcal{H}(G)\mathcal{I}(G) \subseteq \zeta(\Z[G])
\]
and it measures the failure of `generalized adjoint matrices'
to have coefficients in $\Z[G]$.
The Stickelberger elements are now conjectured to lie in the
integrality ring so that each integer in $\mathcal{H}(G)$
is conjecturally a bound for the denominators.

It is not hard to see that $|G| \in \mathcal{H}(G)$
and indeed $|G| \theta_S^T(r) \in \zeta(\Z[\half][G])$ as follows from
work of Burns and the second named author (we refer the reader to
Theorem \ref{thm:ETNC-cases} for details). It is also known
that $|G| \theta_S^T(0)$ annihilates the class group away from
its $2$-primary part as has been shown by Burns and Johnston
\cite{MR2771125} under certain hypotheses and by the second named author
\cite{StrongStark} in general (note, however, that the hypotheses in
\cite{MR2771125} have been tailored to guarantee the validity of the 
strong Stark
conjecture of Chinburg \cite{MR724009} for totally odd characters
to the state of knowledge at that time, and it is the latter
that has been verified in \cite{StrongStark}). 
For a slightly weaker variant
of these results for the non-abelian Brumer--Stark conjecture
of Dejou and Roblot see \cite{MR3830817}.

Already the abelian case makes it clear that this is not the best
possible bound. For this reason, we first deal with the purely 
algebraic question of which integers lie in the denominator ideal.
The following result gives a complete answer.

\begin{theoremA} \label{thm:denominator-ideal-capZ}
	Let $G$ be a finite group with commutator subgroup $G'$. Then we have an equality
	\[
		\mathcal{H}(G) \cap \Z = |G'| \Z.
	\]
\end{theoremA}

The proof is by induction on the cardinality of $G$ and makes heavy use
of Clifford theory. The method also allows us to show that
the integrality ring $\mathcal{I}(G)$ is strictly larger than
$\zeta(\Z[G])$ unless $G$ is abelian (in which case both rings
are obviously equal). This observation is not needed in later sections,
but we feel that it might be of independent interest.

Given Theorem \ref{thm:denominator-ideal-capZ}, the integrality
conjecture predicts that $|G'|\theta_S^T(r)$ has integral
coefficients. We refine the method of proof of Theorem 
\ref{thm:denominator-ideal-capZ} using \eqref{eqn:abelian-int}
and the results of Dasgupta and Kakde \cite{dasgupta-kakde}
and of Johnston and the second named author \cite{abelian-MC}
as our base case. The main result of this article is the following.

\begin{theoremA} \label{thm:intro-main-result}
	Let $L/K$ be a Galois extension of number fields such that $G = \Gal(L/K)$
	is nilpotent. Let $r \leq 0$ be an integer. Then we have
	\[
	|G'| \theta_S^T(r) \in \zeta(\Z[G])
	\]
	and moreover the following holds.
	\begin{enumerate}
		\item 
		If $r=0$, then $|G'| \theta_S^T(0)$ annihilates $\Z[\half] \otimes_{\Z} \cl_{L}^{T}$.
		\item
		If $r<0$, then $|G'| \theta_S^T(r)$ annihilates $\Z[\half] \otimes_{\Z} 
		K_{-2r}(\mathcal{O}_{L,S})$.
	\end{enumerate}
\end{theoremA}

We refer the reader to Theorem \ref{thm:main-result}
for a more precise statement. We stress that
replacing the factor $|G|$ by $|G'|$ produces a 
considerably stronger result. This is similar in spirit to the fact that
the equivariant Tamagawa number conjecture
(as formulated by Burns and Flach \cite{MR1884523})
for the motive $h^0(\Spec(L)(r))$
with coefficients in $\Z[G]$ is stronger than with coefficients
in a maximal order containing $\Z[G]$.

We do not know yet how to remove the condition on $G$ to be nilpotent.
Roughly speaking, the reason is that Artin $L$-series 
behave well under induction
and inflation of characters, but not under Morita equivalence --
in contrast to denominator ideals. That is why the condition appears
in Theorem \ref{thm:intro-main-result}, but not in
Theorem \ref{thm:denominator-ideal-capZ}.\\

We finally give an application to Iwasawa theory.
Let $K$ be a totally real number field and let $p$ be an odd prime.
Let $\psi$ be a totally even Artin character of the absolute 
Galois group of $K$.
Recall that the $S$-truncated $p$-adic Artin $L$-series
$L_{p,S}(s, \psi) : \Z_{p} \rightarrow \C_{p}$ attached to $\psi$
is the unique $p$-adic meromorphic
function with the property that for each strictly negative 
integer $r$ we have
\begin{equation}\label{eqn:intro-interpolation-property}
	L_{p,S}(r, \psi) = L_{S}(r, \psi \omega^{r-1}),
\end{equation}
where $\omega$ denotes the Teichmüller character.
If $\psi$ is linear, then \eqref{eqn:intro-interpolation-property} is also valid when $r=0$ and the construction of $p$-adic Artin $L$-series is due to
Pi.\ Cassou-Nogu\`es \cite{MR524276}, Deligne and Ribet \cite{MR579702}, and  Barsky \cite{MR525346}. Greenberg \cite{MR692344} 
then generalized the construction 
to arbitrary $\psi$ by means of Brauer induction. As the
$L$-values at $r=0$ may vanish, we cannot deduce that
\eqref{eqn:intro-interpolation-property} still holds at $r=0$, as we would
possibly divide by zero. Greenberg \cite[\S 4]{MR692344} writes that
`it would be interesting to prove that $L_p(s, \psi)$ is in fact analytic
at $s=0$ and has the correct value', which is also predicted by a
conjecture of Gross \cite[Conjecture 2.12b]{MR656068} on leading terms
of $p$-adic $L$-functions.

It is clear that \eqref{eqn:intro-interpolation-property} holds
with $r=0$ if the character is monomial, i.e.\ induced from a linear
character. In general, it is known by work of Burns \cite{MR4195656} that
the left hand side of \eqref{eqn:intro-interpolation-property} vanishes whenever the right hand side does. This is ultimately a consequence
of the main conjecture proved by Wiles \cite{MR1053488}. If Gross's `order of vanishing
conjecture' \cite[Conjecture 2.12a]{MR656068} holds 
for all irreducible characters,
then \eqref{eqn:intro-interpolation-property} with $r=0$ follows from work
of Dasgupta, Kakde and Ventullo \cite{MR3866887} as has been observed
by Burns \cite[Theorem 2.6]{MR4195656} (see also 
\cite[Corollary 16.4.18]{nickel-conjectures}).

Since every irreducible character of a nilpotent group is monomial
by \cite[Theorem 11.3]{MR632548},
we do not want to restrict ourselves to nilpotent extensions. 
So we cannot apply Theorem
\ref{thm:intro-main-result} directly. However, we can still prove
that the denominators of
the Stickelberger elements remain bounded
along the cyclotomic $\Z_p$-tower as predicted by the integrality
conjecture and our Theorem \ref{thm:denominator-ideal-capZ}
(we stress that the easy bound $|G| \in \mathcal{H}(G)$ would not be
sufficient for this purpose).
As a consequence of Proposition \ref{prop:bounding-denominators} (ii)
we have the following result, which is
Theorem \ref{thm:bounded-denominators} below.
We write $\theta_S^T(L/K,r)$ if the Stickelberger
element $\theta_S^T(r)$ is associated with the extension $L/K$.

\begin{theoremA}
	Let $p$ be an odd prime and let $r \leq 0$ be an integer.
	Let $L/K$ be a Galois CM extension of number fields
	and denote the $n$-th layer of the cyclotomic
	$\Z_p$-extension of $L$ by $L_n$. 
	Then the denominators
	of the Stickelberger elements 
	$\theta_{S}^T(L_n/K,r) \in \Q_p[\mathrm{Gal}(L_n/K)]$
	are bounded by a constant that does not depend on either
	$n$ or $r$.
\end{theoremA}

In this way, we obtain a new direct construction of $p$-adic Artin
$L$-series with the additional property that the value at zero
is as predicted by Greenberg and by Gross. 
This essentially follows by construction so that
we obtain the following result, which is Corollary
\ref{cor:Gross-conjecture} below.

\begin{theoremA} \label{thm:Gross-conj}
	For each totally even character $\psi$ and each odd prime $p$
	one has an equality
	\[
	L_{p,S}(0,\psi) = L_S(0, \psi \omega^{-1}).
	\]
\end{theoremA}

As a further application, we give a new proof of the $p$-adic Artin conjecture
\cite[p.\ 82]{MR692344}, which does not rely on the validity of the
main conjecture.

\subsection*{Notation and conventions}
All rings are assumed to have an identity element and all modules are assumed
to be left modules unless otherwise  stated. We denote the set of all $m \times n$
matrices with entries in a ring $R$ by $M_{m \times n} (R)$ and in the case $m=n$
the group of all invertible elements of $M_n(R) := M_{n \times n} (R)$ by $\GL_{n}(R)$.
We let $\mathbf{1}_{n \times n}$ be the $n \times n$ identity matrix.
We write $\zeta(R)$ for the centre of a ring $R$ and $\Ann_{R}(M)$
for the annihilator ideal of an $R$-module $M$.
We shall sometimes abuse notation by using the symbol $\oplus$ to 
denote the direct product of rings or orders.

If $q$ is a prime power, we denote the finite field with $q$ elements by $\F_q$.
If $K$ is a field, we choose a separable closure $K^c$ of $K$
and set $G_K := \Gal(K^c/K)$.

\subsection*{Acknowledgments}
The authors wish to thank Samit Dasgupta and Mahesh Kakde for helpful
correspondence concerning their work on the Brumer--Stark conjecture 
\cite{dasgupta-kakde} and 
Xavier-Fran\c{c}ois Roblot for answering our questions on
his work with Dejou \cite{MR3208394}. We are very grateful to
Henri Johnston for a thorough reading
of the first author's PhD thesis \cite{Nils_thesis} and his useful comments.
Moreover, we thank the anonymous referee for her or his
suggestions to further improve the article.
The first named author thanks his office mates 
Antonio Mejías Gil and Gürkan Dogan for several fruitful discussions.

The authors acknowledge financial support provided by the 
Deutsche Forschungsgemeinschaft (DFG) through the research grant
`Integrality of Stickelberger elements' (project no.\ 399131371).
The second named author was also supported through
the DFG Heisenberg programme (project no.\ 437113953).

\section{The integrality ring and the denominator ideal}

\subsection{Reduced norms and the integrality ring} \label{subsec:int-ring}
Let $\mathcal{O}$ be a noetherian integrally closed domain 
with field of quotients $F$ and
let $A$ be a finite dimensional separable $F$-algebra. 
If $e_{1}, \ldots, e_{t}$ are the central primitive
idempotents of $A$ then
\[
A = A_{1} \oplus \cdots \oplus A_{t}
\]
where $A_{i}:=Ae_{i}=e_{i}A$.
Each $A_{i}$ is isomorphic to an algebra of $n_{i} \times n_{i}$ matrices over a skewfield $D_{i}$,
and $F_{i}:=\zeta(A_{i})=\zeta(D_{i})$ is a finite separable 
field extension of $F$; in particular, each $A_{i}$ is a central
simple $F_{i}$-algebra. We denote the Schur index of $D_{i}$ by $s_{i}$ so that $[D_{i}:F_{i}]=s_{i}^{2}$.
The reduced norm map
\[
\nr = \nr_{A}: A \longrightarrow \zeta(A)=F_{1} \oplus \cdots \oplus F_{t}
\]
is defined componentwise (see \cite[\S 9]{MR1972204})
and extends to matrix rings over $A$ in the obvious way.

Let $\Lambda$ be an $\mathcal{O}$-order in $A$. Then $\Lambda$
is noetherian and by \cite[Corollary 10.4]{MR1972204} 
we may choose a maximal $\mathcal{O}$-order
$\Lambda'$ in $A$ that contains $\Lambda$. Moreover, $\Lambda'$
decomposes into a finite product
\[
	\Lambda' = \Lambda_1' \oplus \cdots \oplus \Lambda_t'
\]
where $\Lambda_i' := \Lambda' e_i$. Denote the integral closure
of $\mathcal{O}$ in $F_i$ by $\mathcal{O}_i$. 
By \cite[Theorem 10.5]{MR1972204} each 
$\Lambda_i'$ is a maximal $\mathcal{O}_i$-order
in the $F_i$-algebra $A_i$ and one has $\zeta(\Lambda_i') = \mathcal{O}_i$.
We point out that the reduced norm maps an element in $\Lambda$
or, more generally, in $M_{n}(\Lambda)$ into
$\zeta(\Lambda') = \mathcal{O}_1 \oplus \cdots \oplus \mathcal{O}_t$
by \cite[Theorem 10.1]{MR1972204},
but not necessarily into $\zeta(\Lambda)$.
Following \cite[\S 3.4]{MR3092262}, 
we define a $\zeta(\Lambda)$-submodule of $\zeta(A)$ by
\[
	\mathcal{I}(\Lambda) := \langle \nr(H) \mid H \in M_{n}(\Lambda), 
	n \in \N  \rangle_{\zeta(\Lambda)}.
\]
As the reduced norm is multiplicative, this is in fact an $\mathcal{O}$-order in $\zeta(A)$ 
contained in $\zeta(\Lambda')$. We call $\mathcal{I}(\Lambda)$ the
\emph{integrality ring} of $\Lambda$.
It is clear that $\mathcal{I}(\Lambda)$ always contains the centre
$\zeta(\Lambda)$, and that $\mathcal{I}(\Lambda) = \zeta(\Lambda)$
whenever $\Lambda$ is commutative. In general, however,
the inclusion $\zeta(\Lambda) \subseteq \mathcal{I}(\Lambda)$
might be strict. We consider some examples below.\\

We are mainly interested in the case $\mathcal{O} = \Z$
and $\Lambda = \Z[G]$, where $G$ is a finite group. Then $\Lambda$
is a $\Z$-order in the $\Q$-algebra $A = \Q[G]$ which is not maximal
unless $G = 1$. We will also consider certain variations thereof.
In particular, if $p$ is a prime, we shall be interested in the
case $\mathcal{O} = \Z_p$, $\Lambda = \Z_p[G]$ and $A = \Q_p[G]$.
For brevity, we set $\mathcal{I}(G) := \mathcal{I}(\Z[G])$
and $\mathcal{I}_p(G) := \mathcal{I}(\Z_p[G])$.

For any commutative ring $\mathcal{O}$ the group ring $\mathcal{O}[G]$
is free of rank $|G|$ as an $\mathcal{O}$-module. Its centre
$\zeta(\mathcal{O}[G])$ is also a free $\mathcal{O}$-module
of rank $c(G)$, where $c(G)$ denotes the number of conjugacy classes of $G$.
More precisely, if $\mathcal C_i$, $1 \leq i \leq c(G)$,
are the distinct conjugacy classes,
then $C_i := \sum_{g \in \mathcal{C}_i} g$ constitute
an $\mathcal{O}$-basis of $\zeta(\mathcal{O}[G])$.
We denote the commutator subgroup of $G$ by $G'$.

\begin{example}
	Let $G = D_{8} = \langle a, x \mid a^{8} = x^{2} = 1 , x a x = a^{-1} \rangle$ be the dihedral group of order $16$. Then the commutator subgroup 
	$D_{8}' = \langle a^2 \rangle$ is cyclic of order $4$. 
	Furthermore, the conjugacy classes of $D_{8} $ are given by 
	\[	
		\mathcal C_{1} = \{ 1 \}, \quad \mathcal C_{2}  = \{ a^{4} \}, 
		\quad \mathcal C_{3} = \{a^{2}, a^{6} \}, 
		\quad \mathcal C_{4} = \{a, a^{7} \}, 
	\]
	\[
	 	\mathcal C_{5} = \{ a^{3}, a^{5} \}, 
	 	\quad \mathcal C_{6} = \{ x, a^{2}x, a^{4}x,a^{6}x\}, 
	 	\quad \mathcal C_{7} = \{ ax, a^{3}x,a^{5}x,a^{7}x\}. 
	\]
	An easy calculation gives that
	\[ 
		\nr(a) = \frac{1}{4} (3 C_{1} - C_{2} - C_{3} + C_{4} + C_{5})
		\in \mathcal{I}(D_8).
	\]
	On the one hand, we see that $\zeta(\Z[D_8])$ is strictly contained in
	$\mathcal{I}(D_8)$. On the other hand, we have that
	$|D_{8}'| \nr(a) = 4 \nr(a) \in \zeta(\Z[D_8])$. 
\end{example}

\begin{example} \label{ex:special-linear-group}
	Let $G = \mathrm{SL}_2(\F_3)$ be the special linear group of degree $2$
	over $\F_3$. This group is generated by the matrices
	\[
		\alpha = \squarematzwei{0}{-1}{1}{1}, \quad  \beta = \squarematzwei{1}{1}{1}{-1}, \quad \gamma = \squarematzwei{1}{1}{0}{1}
	\]
	and indeed has the presentation
	\[
	 	\mathrm{SL}_2(\F_3) = \langle \alpha, \beta, \gamma \mid \alpha^{4} = \beta^{4} = \gamma^{3} = 1 , \beta \gamma = \gamma \alpha, \gamma \beta = \alpha \beta \gamma \rangle. 
	\]
	There are seven conjugacy classes. We let 
	$\mathcal{C}_1 = \left\{\mathbf{1}_{2 \times 2}\right\}$ and
	$\mathcal{C}_2 = \left\{-\mathbf{1}_{2 \times 2}\right\}$. Moreover,
	we denote the conjugacy classes containing $\gamma$ and $\gamma^2$
	by $\mathcal{C}_3$ and $\mathcal{C}_4$, and those containing
	$\alpha^2 \gamma$ and $\alpha^2 \gamma^2$ by $\mathcal{C}_5$ and
	$\mathcal{C}_6$, respectively. Finally, $\mathcal{C}_7$ denotes
	the class that contains $\alpha$ and hence $\beta = \gamma \alpha \gamma^{-1}$. 
	The commutator subgroup equals $\mathcal{C}_1 \cup \mathcal{C}_2
	\cup \mathcal{C}_7$ and has cardinality $8$. Then clearly
	$\nr(\alpha) = \nr(\beta) = 1$. However, a lengthy computation
	shows that
	\[
		\nr(\gamma) =  \frac{1}{8} (3C_{1} + 3C_{2} + C_{3} - 2C_{4} + 2C_{5} + C_{6} - C_{7}) \in \mathcal{I}(\mathrm{SL}_2(\F_3)).
	\]
	We see that the integrality ring $\mathcal{I}(\mathrm{SL}_2(\F_3))$ 
	is again strictly larger
	than $\zeta(\Z[\mathrm{SL}_2(\F_3)])$ and that 
	$|\mathrm{SL}_2(\F_3)'| \nr(\gamma) = 8 \nr(\gamma) \in
	\zeta(\Z[\mathrm{SL}_2(\F_3)])$.	
\end{example}

\begin{lemma} \label{lem:int-ring-under-projections}
	Let $G$ be a finite group and let $N$ be a normal subgroup of $G$.
	Let $\mathcal{O}$ be a noetherian integrally closed domain 
	with field of quotients $F$ whose characteristic does not divide 
	the cardinality of $G$. Then the canonical projection $\pi: G \rightarrow
	\overline{G} := G/N$ induces a map $\pi: F[G] \rightarrow F[\overline{G}]$
	such that the following hold.
	\begin{enumerate}
		\item 
		We have inclusions $\pi(\zeta(\mathcal{O}[G])) \subseteq 
		\zeta(\mathcal{O}[\overline{G}])$ and 
		$\pi(\mathcal{I}(\mathcal{O}[G])) \subseteq 
		\mathcal{I}(\mathcal{O}[\overline{G}])$.
		\item 
		If $\pi: \zeta(\mathcal{O}[G]) \rightarrow
		\zeta(\mathcal{O}[\overline{G}])$ is surjective, then  
		$\pi: \mathcal{I}(\mathcal{O}[G]) \rightarrow
		\mathcal{I}(\mathcal{O}[\overline{G}])$ is also surjective.
		\item 
		The map $\pi: \mathcal{I}(\mathcal{O}[G]) \rightarrow
		\mathcal{I}(\mathcal{O}[\overline{G}])$ is surjective whenever
		$\overline{G}$ is abelian.
	\end{enumerate}
\end{lemma}

\begin{proof}
	We first observe that $F[G]$ is a separable $F$-algebra
	by Maschke's theorem \cite[Theorem 3.14]{MR632548}.
	If $x$ is central in $\mathcal{O}[G]$, then clearly $\pi(x)$ is central
	in $\mathcal{O}[\overline{G}]$, which is the first inclusion in (i). 
	Let us write $\nr_G$ for the reduced
	norm on (matrix rings over) $F[G]$ and similarly for $\overline{G}$.
	Then $\pi$ induces a surjective map 
	$\pi_n: M_{n}(\mathcal{O}[G]) \rightarrow
	M_{n}(\mathcal{O}[\overline{G}])$ for each $n \in \N$ and one has
	\[
		\pi(\nr_G(H)) = \nr_{\overline{G}}(\pi_n(H))
	\]
	for each $H \in M_{n}(\mathcal{O}[G])$. Since every 
	$x \in \mathcal{I}(\mathcal{O}[G])$ can be written as a finite sum
	$x = \sum_j x_j \nr_G(H_j)$ with $x_j$ central in $\mathcal{O}[G]$
	and square matrices $H_j$ with entries in $\mathcal{O}[G]$,
	this proves (i). For (ii) we let $y \in \mathcal{I}(\mathcal{O}[\overline{G}])$
	be arbitrary. Then likewise $y$ may be written as a finite sum
	$y = \sum_j y_j \nr_{\overline{G}}(Z_j)$ with $y_j$ central in $\mathcal{O}[\overline{G}]$
	and matrices $Z_j \in M_{n_j}(\mathcal{O}[\overline{G}])$
	with $n_j \in \N$.
	Since $\pi_n$ is surjective for every $n$, there are matrices
	$H_j \in M_{n_j}(\mathcal{O}[G])$ such that 
	$\pi_{n_j}(H_j) = Z_j$ for all $j$. If $\pi: \zeta(\mathcal{O}[G]) \rightarrow
	\zeta(\mathcal{O}[\overline{G}])$ is surjective, then there likewise are
	$x_j \in \zeta(\mathcal{O}[G])$ such that $\pi(x_j) = y_j$
	and hence $x := \sum_j x_j \nr_G(H_j)$ is a preimage of $y$ such that
	$x \in \mathcal{I}(\mathcal{O}[G])$. For (iii) we assume that
	$\overline{G}$ is abelian. Then we have that 
	$\mathcal{I}(\mathcal{O}[\overline{G}]) = \mathcal{O}[\overline{G}]$
	and we can simply write $y = \nr_{\overline{G}}(y) = \pi(\nr_G(x))$
	for any $x \in \mathcal{O}[G]$ with $\pi(x) = y$.
\end{proof}

\begin{example}
	Let $G = S_3$ be the symmetric group on three letters and let
	$N = A_3$ be the unique non-trivial normal subgroup of $S_3$.
	Then $S_3/A_3$ is cyclic of order $2$ so that the canonical map
	$\mathcal{I}(\mathcal{O}[S_3]) \rightarrow \mathcal{I}(\mathcal{O}[S_3/A_3])
	= \mathcal{O}[S_3/A_3]$ is surjective by 
	Lemma \ref{lem:int-ring-under-projections} (iii). However, we show that
	$\pi: \zeta(\mathcal{O}[S_3]) \rightarrow \zeta(\mathcal{O}[S_3/A_3])
	= \mathcal{O}[S_3/A_3]$ is in general not surjective. There are
	three conjugacy classes $\mathcal{C}_i$ of $S_3$, where each $\mathcal{C}_i$
	consists of the $i$-cycles, $1 \leq i \leq 3$. Let us denote the non-trivial
	element of $S_3/A_3$ by $\tau$. Then $\pi(C_1) = 1$, $\pi(C_2) = 3 \tau$
	and $\pi(C_3) = 2$ so that the image of $\pi$ equals $\mathcal{O} + 
	3 \mathcal{O} \tau$. Hence $\pi$ is surjective if and only if $3$ is
	a unit in $\mathcal{O}$.
\end{example}

\begin{remark} \label{rem:canonical-central-order}
	In \cite[\S 3.1]{non-abelian-zeta} Burns and Sano introduce 
	a variant of the integrality ring when $\mathcal{O}$
	is a Dedekind domain. Assume for simplicity that $\mathcal{O}$
	is a discrete valuation ring. Then they
	let $\xi(\Lambda)$ be the $\mathcal{O}$-submodule 
	of $\zeta(\Lambda')$ generated by the
	elements $\nr(H)$, where $H$ runs over all square matrices with entries
	in $\Lambda$, and call it the Whitehead order of $\Lambda$. 
	Then $\xi(\Lambda)$ is clearly contained in $\mathcal{I}(\Lambda)$,
	and this inclusion can be strict in general.
	In the setting of Lemma \ref{lem:int-ring-under-projections}
	the Whitehead order indeed behaves better under projections.
	The proof of Lemma \ref{lem:int-ring-under-projections} (ii) shows that
	one always has $\pi(\xi(\mathcal{O}[G])) = \xi(\mathcal{O}[\overline{G}])$.
	See also \cite[Lemma 3.2(v)]{non-abelian-zeta}. 
\end{remark}

\subsection{Generalised adjoints}
We keep the notation of \S \ref{subsec:int-ring}.
Let $H \in M_{n}(A)$ and decompose 
$H = (H_1, \dots, H_t)$,
where $H_i = H e_i \in M_{n}(A_i)$. Following \cite[\S 9]{MR1972204},
we let $f_{H_i}(X) \in F_i[X]$ be the reduced characteristic polynomial of 
$H_i$ and set $f_H(X) := (f_{H_i}(X))_i \in \bigoplus_{i=1}^{t} F_i[X]$.
The latter may itself be written as reduced norm of the separable
$F(X)$-algebra $A \otimes_F F(X)$. Indeed a straightforward exercise
(see also \cite[Lemma 1.6.6]{watson_thesis}) shows that
\begin{equation} \label{eqn:redchar-as-nr}
	f_H(X) = \nr_{A \otimes_F F(X)} (1 \otimes X - H \otimes 1) \in
	\zeta(A)[X] = \bigoplus_{i=1}^{t} F_i[X].
\end{equation}
Set $m_i := n_i \cdot s_i \cdot n$.
Then we may write $f_{H_i}(X) = \sum_{j=0}^{m_i} \alpha_{ij} X^j$
with $\alpha_{ij} \in F_i$ for all $1 \leq j \leq m_i$.
Note that the constant term $\alpha_{i0}$ equals $(-1)^{m_i} \cdot \nr(H_i)$.
We define
\[
	H_i^{\ast} := (-1)^{m_i+1} \sum_{j=1}^{m_i} \alpha_{ij} H_i^{j-1}
	\in M_{n}(A_i)
\]
and call  $H^{\ast} := (H_1^{\ast}, \dots, H_t^{\ast}) \in M_{n}(A)$
the \emph{generalised adjoint} of $H$. Since $f_{H_i}(H_i)$ vanishes
for each $i$, one has
\[
	H H^{\ast} = H^{\ast} H = \nr(H) \cdot \mathbf{1}_{n \times n}.
\]
Note that we follow the conventions in \cite[\S 3.6]{MR3092262} 
which slightly differ from those given in \cite[\S 4]{MR2609173}.

Now suppose that $H \in M_{n}(\Lambda)$. Then by another application
of \cite[Theorem 10.1]{MR1972204} the generalised adjoint $H^{\ast}$ has
entries in $\Lambda'$, but not necessarily in $\Lambda$.
This motivates the following definition. We set
\[
	\mathcal{H}(\Lambda) := \left\{ x \in \zeta(\Lambda) \mid 
	xH^{\ast} \in M_{n}(\Lambda) \forall H \in M_{n}(\Lambda) 
	\forall n \in \N \right\}
\]
and call $\mathcal{H}(\Lambda)$ the \emph{denominator ideal}
of $\Lambda$. It is easy to see that $\mathcal{H}(\Lambda)$
is an ideal in $\zeta(\Lambda)$. With a little more work 
\cite[Lemma 1.10.9]{watson_thesis} one can 
actually show that it is an ideal in $\mathcal{I}(\Lambda)$ such that
\begin{equation} \label{eqn:HI-in-centre}
	\mathcal H(\Lambda) \cdot \mathcal I(\Lambda) \subseteq \zeta(\Lambda).
\end{equation}
If $G$ is a finite group, we set
$\mathcal{H}(G) := \mathcal{H}(\Z[G])$ and $\mathcal{H}_p(G) :=
\mathcal{H}(\Z_p[G])$ for each prime $p$.

General results on denominator ideals are rather sparse. Watson \cite{watson_thesis}
has computed $\mathcal{H}(\mathcal{O}[G])$ whenever $\mathcal{O}$ has characteristic
$0$ and $G$ is a finite $p$-group such that $|G'| = p$.
Moreover, we have the following result.

\begin{prop} \label{prop:HpG-equals-centre}
	Let $G$ be a finite group and let $p$ be a prime. Then we have
	$\mathcal{H}_p(G) = \zeta(\Z_p[G])$ if and only if $p$ does not
	divide the 
	order of the commutator subgroup $G'$ of $G$. In this case, we also have that
	$\mathcal{I}_p(G) = \zeta(\Z_p[G])$.
\end{prop}

\begin{proof}
	The first claim is \cite[Proposition 4.4]{MR3092262}, but
	note that this can also be easily deduced
	from Theorem \ref{thm:denominator-ideal-capZ} (whose proof does not
	depend on Proposition \ref{prop:HpG-equals-centre}).
	The second claim then follows immediately from \eqref{eqn:HI-in-centre}.
\end{proof}

We are mainly interested in the ideal $\Z \cap \mathcal{H}(G)$.
The following important example, which is essentially taken from 
the proof of Proposition \ref{prop:HpG-equals-centre} given in \cite{MR3092262},
provides an upper bound for that ideal.

\begin{example} \label{ex:0-star}
	Assume that $\mathcal{O}$ has characteristic $0$.
	We consider $H = 0 \in M_{1}(\mathcal{O}[G])$.
	Decompose $H = (H_1, \dots, H_t)$ as above. Then the
	reduced characteristic polynomial of $H_i$ is given by $f_{H_i}(X) = X^{m_i}$,
	where $m_i = n_i \cdot s_i$, so that $H_i^{\ast} = h_i(0)$
	with $h_i(X) = X^{m_i - 1}$. In other words, we have $H_i^{\ast} = 1$
	if $m_i=1$, and $H_i^{\ast} = 0$ otherwise. This implies that
	\[
		0^{\ast} = \frac{1}{|G'|} \Tr_{G'},
	\]
	where $\Tr_{G'} := \sum_{g \in G'} g$.
	In particular, every integer in $\Z \cap \mathcal{H}(G)$ must be divisible
	by $|G'|$.
\end{example}

\section{Clifford theory} \label{sec:Clifford-theory}

\subsection{General notation and Clifford's Theorem} \label{subsec:Clifford-general}
Let $K$ be a field of characteristic $0$ and let $G$ be a finite group.
We denote the set of $K$-irreducible characters of $G$ by $\Irr_K(G)$.
If $K$ is clear from context, we simply write $\Irr(G)$ for $\Irr_{K^c}(G)$.

We will always assume in this section that $K$ is sufficiently large 
in the sense that it is a
splitting field for $G$ and all of its subgroups. In particular,
we have $\Irr_K(G) = \Irr(G)$ and
the Wedderburn decomposition of $K[G]$ is given by
\[
	K[G] = \bigoplus_{\chi \in \Irr(G)} K[G] e_{\chi}
	\simeq \bigoplus_{\chi \in \Irr(G)} M_{\chi(1)}(K),
\]
where $e_{\chi} := \chi(1)|G|^{-1} \sum_{g \in G} \chi(g^{-1}) g$
are the primitive central idempotents of $K[G]$.
We write $R^+(G)$ for the set of characters associated to finite-dimensional
$K^c$-valued representations of $G$, and $R(G)$ for the ring of
virtual characters generated by $R^+(G)$.
For each character $\chi \in R^+(G)$ we choose a $K[G]$-module $V_{\chi}$
with character $\chi$. Conversely, if $V$ is a finitely generated $K[G]$-module,
we denote the associated character by $\chi_{V}$.
If $U$ is a subgroup of $G$, we write $\res^G_U \chi$ for the character
in $R^+(U)$ obtained by restriction. Conversely, for $\psi \in R^+(U)$
we let $\ind^G_U \psi \in R^+(G)$ be the induced character.
The associated $K[G]$-module is given by $K[G] \otimes_{K[U]} V_{\psi}$,
from which one can easily deduce the well-known formula (see \cite[(10.2)]{MR632548})
\begin{equation} \label{eqn:induced-character-formula}
	(\ind_U^G \psi)(g) = \frac{1}{|U|} \sum_{h \in G: h^{-1}gh \in U} \psi(h^{-1}gh).
\end{equation}
We will use \eqref{eqn:induced-character-formula} frequently.
Finally, if $N$ is a normal subgroup of $G$ and $\chi \in R^+(G/N)$,
we write $\infl^G_{G/N} \chi \in R^+(G)$ for the inflated character. 
These notions extend to virtual characters by linearity.
For $\chi, \psi \in R(G)$ we write $\langle \chi, \psi \rangle_G$ for the usual
inner product of virtual characters. We recall
the following well-known fact \cite[Theorem 10.9]{MR632548}.

\begin{prop}[Frobenius reciprocity] \label{prop:Frobenius-reciprocity}
	Let $U$ be a subgroup of $G$. Then for each $\chi \in R(G)$ and $\psi \in R(U)$ 
	one has
	\[
		\langle \chi, \ind^G_U \psi \rangle_G= \langle \res^G_U \chi, \psi \rangle_U.
	\]
\end{prop}

Now let $N$ be a normal subgroup of $G$. Then $G$ acts upon $\eta \in R(N)$ via
$^g \eta(x) = \eta(g^{-1}xg)$, $g \in G$, $x \in N$. We denote
its stabilizer $\{g \in G \mid {}^g \eta = \eta\}$ by $G_{\eta}$ and note that
$N$ is always contained in $G_{\eta}$.
The following result is due to Clifford \cite{MR1503352}
(see also \cite[Proposition 11.4]{MR632548})
and will be used extensively in the following.

\begin{theorem}[Clifford] \label{thm:Clifford}
	Let $N$ be a normal subgroup of a finite group $G$. Let $\chi \in \Irr(G)$
	and chose an irreducible constituent $\eta$ of $\res^G_N \chi$. 
	Then there is a positive integer $m$ such that
	\[
		\res^G_N \chi = m \sum_{g \in G / G_{\eta}} {}^g \eta.
	\]
\end{theorem}

In the situation of Theorem \ref{thm:Clifford} we write
\begin{equation} \label{eqn:ind-eta}
	\ind_N^{G_{\eta}} \eta= \sum_{i=1}^{s} m_i \psi_i,
\end{equation}
where $s$ and $m_1, \dots, m_s$ are positive integers and $\psi_i \in \Irr(G_{\eta})$
are distinct irreducible characters. 
By \cite[Kapitel V, Satz 17.11]{MR0224703} the characters 
$\chi_i := \ind_{G_{\eta}}^G \psi_i$ are irreducible 
and each irreducible character of $G$
whose restriction to $N$ is divisible by $\eta$ is a $\chi_i$.
So we may and do assume that $\chi = \chi_1$ and set $\psi := \psi_1$.
It follows that
\begin{equation} \label{eqn:ind-eta-II}
	\ind_N^G \eta = \sum_{i=1}^s m_i \chi_i
\end{equation}
and we have $m_1 = m$ by Frobenius reciprocity. Likewise we have that
\[
	m_i = \langle \psi_i, \ind_N^{G_{\eta}} \eta \rangle_{G_{\eta}} = 
	\langle \res^{G_{\eta}}_N \psi_i, \eta \rangle_N.
\]
Since $G_{\eta}$ acts trivially upon $\eta$, Clifford's Theorem \ref{thm:Clifford}
implies that we have
\begin{equation} \label{eqn:res-psi}
	\res^{G_{\eta}}_N \psi_i= m_i \eta
\end{equation}
for $1 \leq i \leq s$. See also \cite[Proposition 11.4]{MR632548}.

\subsection{Clifford theory for subgroups containing the commutator subgroup} \label{subsec:Clifford-for-commutator-subgroup}
We keep the notation of \S \ref{subsec:Clifford-general} and assume in addition
that $N$ contains the commutator subgroup $G'$ of $G$.
Then $N$ is normal in $G_{\eta}$ and the quotient $G_{\eta}/N$ is abelian.

\begin{prop} \label{prop:psi-when-N-contains-commutator}
	Let $G$ be a finite group and let $N$ be a (normal) subgroup containing
	the commutator subgroup $G'$ of $G$. Let $\chi \in \Irr(G)$ be an irreducible
	character and let $\eta$ be an irreducible constituent of $\res^G_N \chi$.
	Write $\ind_N^{G_{\eta}} \eta= \sum_{i=1}^{s} m_i \psi_i$ as in \eqref{eqn:ind-eta}
	and recall that $\psi = \psi_1$ and $m = m_1$.
	Then the following holds.
	\begin{enumerate}
		\item 
		Each $\psi_i$ is of the from $\psi_i = \psi \otimes \omega_i$
		for some $\omega_i \in \Irr(G_{\eta}/N)$.
		\item
		Each $\psi \otimes \omega$, $\omega \in \Irr(G_{\eta}/N)$ is a $\psi_i$.
		\item 
		We have that $m_i = m$ for all $1 \leq i \leq s$ and $[G_{\eta}:N] = sm^2$.
	\end{enumerate}
\end{prop}

\begin{proof}
	We first note that $\psi \otimes \omega$ is irreducible for every 
	$\omega \in \Irr(G_{\eta}/N)$, since the quotient $G_{\eta}/N$ is abelian.
	We claim that we have an equality
	\[
		m \, \ind_N^{G_{\eta}} \eta = 
			\sum_{\omega \in \Irr(G_{\eta}/N)} \psi \otimes \omega.
	\]
	Then (i) and (ii) follow from \eqref{eqn:ind-eta}.
	For the claim we evaluate both sides of the equation at $g \in G_{\eta}$.
	Since $N$ is normal in $G_{\eta}$, the left-hand side vanishes
	for every $g \not\in N$;
	likewise the right-hand side vanishes, as $\sum_{\omega} \omega$
	is the regular character of $G_{\eta}/N$. Now let $g \in N$. Then the
	formula \eqref{eqn:induced-character-formula} for induced characters implies
	that the left-hand side equals $m [G_{\eta}:N] \eta(g)$. Likewise we have that
	$(\sum_{\omega} \psi \otimes \omega)(g) = [G_{\eta}:N] \psi(g) = m [G_{\eta}:N]
	\eta(g)$, where the last equality is \eqref{eqn:res-psi}. This shows the claim.
	Again by \eqref{eqn:res-psi} we see that for each $1 \leq i \leq s$ one has
	\[
		m_i \eta(1) = \psi_i(1) =  \psi(1) \omega_i(1) = \psi(1) = m \eta(1).
	\]
	As $\eta(1)$ is non-zero, we must have $m_i = m$ for all $i$. Finally, we compute
	\[
		[G_{\eta}:N] \eta(1) = \ind_N^{G_{\eta}} \eta(1) = m \sum_{i=1}^s \psi_i(1)
		= s m^2 \eta(1).
	\]
	This finishes the proof of (iii).
\end{proof}

\begin{corollary} \label{cor:idempotent-equalities}
	Keep the notation of Proposition \ref{prop:psi-when-N-contains-commutator}
	and recall that $\chi_i = \ind_{G_{\eta}}^G \psi_i$ is an irreducible character
	of $G$. Then the following holds.
	\begin{enumerate}
		\item 
		The sum $\sum_{i=1}^s e_{\chi_i}$ is an element of $K[N]$ and indeed
		\[
			\sum_{i=1}^s e_{\chi_i} = \sum_{x \in G/G_{\eta}} e_{^x \eta}.
		\]
		\item 
		The sum $\sum_{i=1}^s e_{\psi_i}$ is an element of $K[N]$ and indeed
		\[
		\sum_{i=1}^s e_{\psi_i} = e_{\eta}.
		\]
		\item 
		For each $1 \leq i \leq s$ the idempotent $e_{\chi_i}$ is an element of
		$K[G_{\eta}]$ and indeed
		\[
			e_{\chi_i} = \sum_{x \in G/G_{\eta}} e_{^x \psi_i}.
		\]
	\end{enumerate}
\end{corollary}

\begin{proof}
	We first note that for each $1 \leq i \leq s$ one has
	\[
		\chi_i(1) = [G:G_{\eta}] \psi_i(1) = [G:G_{\eta}]m \eta(1),
	\]
	where we have used Proposition \ref{prop:psi-when-N-contains-commutator} (iii)
	for the last equality.
	Therefore we have that
	\[
		\sum_{i=1}^s e_{\chi_i} = 
		\frac{m \eta(1)}{|G_{\eta}|} \sum_{i=1}^s \sum_{g \in G} \chi_i(g^{-1}) g =
		\frac{\eta(1)}{|G_{\eta}|} \sum_{g \in G} (\ind_N^G \eta)(g^{-1}) g,
	\]
	where the second equality is \eqref{eqn:ind-eta-II}. Since $(\ind_N^G \eta)(g^{-1})$
	vanishes whenever $g \in G \setminus N$, and equals
	$[G_{\eta}:N] \sum_{x \in G/G_{\eta}} {}^x \eta(g^{-1})$ otherwise,
	claim (i) follows. (ii) and (iii) are special cases of (i).
	For (iii) note that by \eqref{eqn:res-psi} we have an inclusion 
	$G_{\psi_i} \subseteq G_{\eta}$ which is indeed an equality. To see this,
	we first observe that 
	$\res^G_{G_{\eta}} \chi_i = \sum_{x \in G/G_{\psi_i}} {}^x \psi_i$
	by Clifford's Theorem and Frobenius reciprocity. This yields
	\[
		[G:G_{\psi_i}] \psi_i(1) = \chi_i(1) = \ind_{G_{\eta}}^G \psi_i(1) 
		= [G:G_{\eta}] \psi_i(1)
	\]
	as desired.
\end{proof}

The characters $\psi \otimes \omega$ for varying $\omega \in \Irr(G_{\eta}/N)$
may not be distinct. To remedy this, we introduce the following notion.
For each $1 \leq i \leq s$ we let $U_{\psi_i}$ be the smallest subgroup of $G_{\eta}$
that contains $N$ and all $g \in G_{\eta}$ such that $\psi_i(g) \not= 0$.
Since each $\psi_i$ is of the form $\psi \otimes \omega_i$ for some linear
character $\omega_i$ by Proposition \ref{prop:psi-when-N-contains-commutator} (i),
this definition actually does not depend on $i$. So we simply write
$U_{\psi}$ in the following.

\begin{lemma} \label{lem:psi-twists}
	Let $\omega, \tilde{\omega} \in \Irr(G_{\eta}/N)$. Then $\psi \otimes \omega
	= \psi \otimes \tilde{\omega}$ if and only if $\omega$ and $\tilde{\omega}$
	coincide on $U_{\psi}$.
\end{lemma}

\begin{proof}
	The set of all $g \in G_{\eta}$ such that $\omega(g) = \tilde{\omega}(g)$
	is equal to the kernel of $\omega^{-1}\tilde{\omega}$ and thus a normal
	subgroup of $G_{\eta}$.
	It is clear that $\psi(g)\omega(g) = \psi(g)\tilde{\omega}(g)$ for all
	$g \in G_{\eta}$ if and only if $\omega(g) = \tilde{\omega}(g)$ whenever
	$\psi(g) \not=0$. Since both $\omega$ and $\tilde{\omega}$ are trivial
	on $N$, the result follows.
\end{proof}

By definition the characters $\psi_i$ vanish outside $U_{\psi}$. In particular,
the associated idempotents $e_{\psi_i}$ belong to the group ring $K[U_{\psi}]$.
We actually need to consider a slightly more general situation.
Let $U$ be a subgroup of $G_{\eta}$ containing $U_{\psi}$. Then $U$ in particular
contains $N$ and thus the commutator subgroups of $G$ 
and a fortiori that of $G_{\eta}$.
Hence we can play the same game as above with $(G,N,\chi)$ replaced by
$(G_{\eta}, U, \psi)$. Choose an irreducible constituent $\rho$ of
$\res_U^{G_{\eta}} \psi$ and denote the stabiliser of $\rho$ in $G_{\eta}$
by $G_{\eta,\rho}$. Then by Clifford's Theorem there is a positive integer $f$
such that
\[
	\res_U^{G_{\eta}} \psi = f \sum_{x \in G_{\eta}/G_{\eta,\rho}} {}^x \rho.
\]
Since ${}^x\omega = \omega$ for every $\omega \in \Irr(G_{\eta}/N)$,
$x \in G_{\eta}$, it is
straightforward to show that likewise
\begin{equation} \label{eqn:res-U-of-psi_i}
	\res_U^{G_{\eta}} \psi_i = f \sum_{x \in G_{\eta}/G_{\eta,\rho}} {}^x \rho_i
\end{equation}
with the same $f$ and where we set
\begin{equation} \label{eqn:rho_i}
	\rho_i := \rho \otimes (\res^{G_{\eta}}_U \omega_i)
\end{equation}
for all $1 \leq i \leq s$.

\begin{lemma} \label{lem:ind-rho_i}
	With the above notation we have $\ind_U^{G_{\eta}} \rho_i = f \psi_i$
	and $[G_{\eta, \rho_i} : U] = f^2$ for all $1 \leq i \leq s$.
\end{lemma}

\begin{proof}
	The characters $\psi_i$ and $\ind_U^{G_{\eta}} \rho_i$ both vanish outside $U$.
	On $U$ the formula for the character $\ind_U^{G_{\eta}} \rho_i$ almost coincides
	with that of $\res_U^{G_{\eta}} \psi_i$ given in \eqref{eqn:res-U-of-psi_i},
	only $f$ is replaced by $[G_{\eta,\rho_i}:U]$. 
	Hence we have that
	\[
		\ind_U^{G_{\eta}} \rho_i = \frac{[G_{\eta,\rho_i}:U]}{f} \psi_i.
	\]
	However, by Frobenius reciprocity the multiplicity of $\psi_i$ in
	$\ind_U^{G_{\eta}} \rho_i$ is given by
	\[
		\langle \psi_i, \ind_U^{G_{\eta}} \rho_i\rangle_{G_{\eta}} = 
		\langle \res^{G_{\eta}}_U \psi_i, \rho_i\rangle_U = f.
	\]
	Both claims follow.
\end{proof}

We record the following consequence of Lemma \ref{lem:ind-rho_i}
and Corollary \ref{cor:idempotent-equalities}.

\begin{corollary} \label{cor:idempotents-for-U}
	For each $1 \leq i \leq s$ the idempotent $e_{\psi_i}$ is an element of $K[U]$
	and indeed
	\[
		e_{\psi_i} = \sum_{x \in G_{\eta}/G_{\eta, \rho}} e_{^x \rho_i}.
	\]
\end{corollary}

\subsection{More Clifford theory} \label{subsec:more-Clifford}
In this final subsection we provide a few technical lemmas which will be used
in \S \ref{sec:int-and-ann}.

\begin{lemma} \label{lem:psi-induced}
	Let $G$ be a finite group and let $\psi \in \Irr(G)$ be an irreducible character.
	Let $H$ be a normal subgroup of $G$ such that 
	$\psi= \ind_H^G \lambda$ for some (necessarily irreducible) character $\lambda$
	of $H$. Then $G_{\lambda} = H$ and $\res^G_H \psi = \sum_{c \in G/H} {}^c\lambda$.
\end{lemma}

\begin{proof}
	By Frobenius reciprocity we have
	\[
		\langle \res^G_H \psi, \lambda \rangle_H 
		= \langle \psi, \ind_H^G \lambda \rangle_G
		= \langle \psi, \psi \rangle_G = 1
	\]
	so that Clifford's Theorem implies
	$\res^G_H \psi = \sum_{c \in G/G_{\lambda}} {}^c\lambda$. In particular,
	we have $\psi(1) = [G : G_{\lambda}] \lambda(1)$. As $\psi = \ind_H^G \lambda$,
	we also have $\psi(1) = [G:H] \lambda(1)$. Hence the inclusion 
	$H \subseteq G_{\lambda}$ must be an equality.
\end{proof}

\begin{lemma} \label{lem:psi-with-H-and-N}
	We keep the notation and assumptions of Lemma \ref{lem:psi-induced}.
	Let $N$ be a normal subgroup of $G$ which contains the commutator subgroup $G'$.
	Assume that $\eta := \res^G_N \psi$ is irreducible (so that in particular
	$G_{\eta} = G$) and that $U_{\psi} = G$. Then the inclusion $H \hookrightarrow G$
	induces an isomorphism $H / (H \cap N) \simeq G/N$ by which we identify
	$\Irr(H / (H\cap N))$ and $\Irr(G/N)$. Then for all $\omega \in \Irr(G/N)$
	\begin{enumerate}
		\item 
		we have $\psi \otimes \omega = \ind_H^G(\lambda \otimes \omega)$;
		\item
		the idempotent $e_{\psi \otimes \omega}$ is an element of $K[H]$ and indeed
		\[
			e_{\psi \otimes \omega} = \sum_{c \in G/H} e_{{}^c \lambda \otimes \omega}.
		\]
	\end{enumerate}
\end{lemma}

\begin{proof}
	We first note that indeed $G_{\eta} = G$ by Clifford's Theorem. 
	Since $\psi$ is induced from a character of the normal subgroup $H$, it vanishes
	outside $H$ and hence $U_{\psi}$ is contained in $NH$. By assumption we have
	$U_{\psi} = G$ so that also $NH = G$ which shows that the natural inclusion
	$H / (H \cap N) \rightarrow G/N$ is an isomorphism. We now prove (i).
	It is clear that both characters vanish outside $H$. For $h \in H$ we compute
	\begin{align*}
		\ind_H^G (\lambda \otimes \omega)(h) & = 
		\frac{1}{|H|} \sum_{g \in G} (\lambda \otimes \omega)(g^{-1}hg)
		= \frac{1}{|H|} \sum_{g \in G} \lambda (g^{-1}hg) \omega(h)\\
		& = \ind_H^G(\lambda)(h) \omega(h) = (\psi \otimes \omega)(h)
	\end{align*}
	as desired. For (ii) we have
	\begin{align*}
		e_{\psi \otimes \omega} & = 
		\frac{\psi(1)}{|G|} \sum_{g \in G} (\psi \otimes \omega)(g^{-1})g\\
		& = \frac{[G:H] \lambda(1)}{|G|} \sum_{h \in H} \sum_{c \in G/H} ({}^c \lambda \otimes \omega)(h^{-1}) h\\
		& = \sum_{c \in G/H} \frac{\lambda(1)}{|H|} \sum_{h\in H} ({}^c \lambda \otimes \omega)(h^{-1}) h\\
		& = \sum_{c \in G/H} e_{{}^c \lambda \otimes \omega},
	\end{align*}
	where the second equality follows from Lemma \ref{lem:psi-induced}.
\end{proof}

\begin{lemma} \label{lem:final-lemma-on-idempotents}
	We keep the notation and assumptions of the two preceding lemmas.
	Then the following holds.
	\begin{enumerate}
		\item 
		For each $c \in G/H$ the character
		\[
		\eta_c := \res^H_{H \cap N} {}^c \lambda = 
		\res^H_{H \cap N} {}^c \lambda \otimes \omega
		\]
		is irreducible and we have an equality
		\[
			\ind^H_{H \cap N} \eta_c = \sum_{\omega \in \Irr(G/N)} {}^c \lambda \otimes \omega.
		\]
		\item
		For $c \in G/H$ we set $e(\eta_c) := \sum_{\omega \in \Irr(G/N)} e_{{}^c \lambda \otimes \omega} \in K[H\cap N]$. Then we have an equality
		\[
			e_{\eta} = \sum_{c \in G/H} e(\eta_c).
		\]
	\end{enumerate}
	
\end{lemma}

\begin{proof}
	Let $\tilde \eta_c \in \Irr(H \cap N)$ be an irreducible constituent of
	$\eta_c$. By Frobenius reciprocity we have
	\begin{equation} \label{eqn:Frobenius-rec}
		\langle {}^c \lambda \otimes \omega, \ind^H_{H \cap N} \tilde \eta_c \rangle_H
		= \langle \eta_c, \tilde \eta_c \rangle_{H \cap N} > 0.
	\end{equation}
	Since the characters 
	${}^c \lambda \otimes \omega$, $\omega \in \Irr(G/N)$ 
	are pairwise distinct by Lemmas
	\ref{lem:psi-with-H-and-N} (i) and \ref{lem:psi-twists}, this shows the inequality
	\[
		[G:N] \lambda(1) \leq [H:H\cap N] \tilde \eta_c(1) = [G:N] \tilde \eta_c(1).
	\]
	However, we clearly have $\tilde \eta_c(1) \leq \lambda(1)$ as $\tilde \eta_c$
	divides $\eta_c$. It follows that $\tilde \eta_c(1) = \lambda(1) = \eta_c(1)$
	and thus $\eta_c = \tilde \eta_c$ is irreducible. By \eqref{eqn:Frobenius-rec}
	it suffices to compare degrees for the final claim in (i). We compute
	\begin{align*}
		\ind^H_{H \cap N} \eta_c(1) & =  [H:H \cap N] \eta_c(1)\\
		& =  [G:N] {}^c\lambda(1)\\
		& =  \sum_{\omega \in \Irr(G/N)} ({}^c \lambda \otimes \omega)(1).
	\end{align*}
	For (ii) we compute
	\[
		e_{\eta} = \sum_{\omega \in \Irr(G/N)} e_{\psi \otimes \omega} = 
		\sum_{\omega \in \Irr(G/N)} \sum_{c \in G/H} e_{{}^c \lambda \otimes \omega}
		= \sum_{c \in G/H} e(\eta_c),
	\]
	where the first and second equality follow from 
	Corollary \ref{cor:idempotent-equalities} (ii) and Lemma 
	\ref{lem:psi-with-H-and-N} (ii), respectively.
\end{proof}

\section{Applications to integrality rings and denominator ideals}

A main aim of this section is to prove that $|G'| \in \mathcal{H}(G)$
for every finite group $G$. This is clear for abelian groups and we will do 
induction on the group order, where the results established in \S
\ref{sec:Clifford-theory} will play a crucial role.

\subsection{Behaviour of the reduced norm under restriction}
Let $G$ be a finite group and let $K$ be a field of characteristic $0$
that is a splitting field for $G$ and all of its subgroups.
Since we want to do induction, we need to know how the reduced norm behaves
when we restrict to a subgroup $U$ of $G$. More precisely, let
$H \in M_n(K[G])$ be a matrix with entries in $K[G]$. On the one hand we can compute
its reduced norm $\nr_G(H) \in \zeta(K[G])$. On the other hand, we may view $H$ as
a $K[G]$-linear map in the usual way. This map is also $K[U]$-linear and
(after a choice of basis) we obtain a new matrix $H|_U \in M_{n[G:U]}(K[U])$.
Its reduced norm $\nr_U(H|_U) \in \zeta(K[U])$ does not depend upon the particular
choice of basis, and we want to determine its value in terms of $\nr_G(H)$.

If $H$ is invertible, results in the spirit of
Lemma \ref{lem:nr-under-restriction} below already appear in the literature 
if $K$ is a number field \cite[(52.23)]{MR892316} 
and if $K = \mathbb{Q}_p^c$ \cite[\S 2.2]{MR2078894}.
Unfortunately, this is not sufficient for our purposes so that we include  a proof of Lemma \ref{lem:nr-under-restriction}
for convenience. We mainly follow the argument given in \cite{MR892316},
but since $H$ might not be invertible, we cannot reduce to the case $n=1$.

\begin{lemma} \label{lem:nr-under-restriction}
	Let $G$ be a finite group and let $U$ a subgroup of $G$. Let $H \in M_n(K[G])$
	and write $\nr_G(H) = \sum_{\chi \in \Irr(G)} \alpha_{\chi} e_{\chi}$ with
	$\alpha_{\chi} \in K$. Then we have that 
	\[
		\nr_U(H|_U) = \sum_{\psi \in \Irr(U)} \beta_{\psi} e_{\psi}, \mbox{ where }
		\beta_{\psi} = \prod_{\chi \in \Irr(G)} \alpha_{\chi}^{\langle \ind_U^G \psi, \chi \rangle_G}.
	\]
\end{lemma}

\begin{proof}
	We consider $K[G]^n$ as a right $K[G]$-module and let $e_1, \dots, e_n$
	be its standard basis. We set $m := [G:U]$ and choose left coset representatives 
	$g_1, \dots, g_m$ of $U$ in $G$. 
	Then $(e_k g_{\ell})_{1 \leq k \leq n, 1 \leq \ell \leq m}$
	constitutes a basis of $K[G]^n$ as a right $K[U]$-module. We define a $K$-linear map
	\begin{eqnarray*}
		\omega_{G,U}: K[G] & \longrightarrow &  K[U]\\
		\sum_{g \in G} x_g g & \mapsto & \sum_{g \in U} x_g g
	\end{eqnarray*}
	that sends each $g \in G \setminus U$ to zero. This map extends to matrix rings
	in the obvious way. With respect to the above basis the $K[U]$-linear map
	$H|_U$ is given by the block matrix
	\[
		B := \left(\omega_{G,U}(g_j^{-1}Hg_i)\right)_{1 \leq i,j \leq m}.
	\]
	To see this let $H = (h_{ij})_{1 \leq i,j \leq n}$ so that
	$H e_kg_{\ell} = \sum_{j=1}^n h_{kj} g_{\ell} e_j$. We likewise have that
	\[
		B e_kg_{\ell} = \sum_{i=1}^m \sum_{j=1}^n (e_j g_i) \omega_{G,U}(g_i^{-1} h_{kj} g_{\ell}).
	\]
	Comparing the coefficients at each $e_j$ we see that it suffices to show that
	\begin{equation} \label{eqn:xg_ell}
		x g_{\ell} = \sum_{i=1}^m g_i \omega_{G,U}(g_i^{-1} x g_{\ell})
	\end{equation}
	for each $x \in K[G]$. By $K$-linearity we may assume that $x = g \in G$.
	Then $\omega_{G,U}(g_i^{-1} g g_{\ell})$ vanishes unless 
	$g_i^{-1} g g_{\ell} \in U$. The latter happens if and only if
	$gg_{\ell} \in g_i U$ and therefore for exactly one choice of $i$.
	This shows \eqref{eqn:xg_ell}.
	
	Now let $\psi$ be an irreducible character of $U$ and let $\pi$ be a
	representation with character $\psi$. Then the induced representation is given by
	(see \cite[(10.1)]{MR632548}, for instance)
	\[
		\ind_U^G(\pi)(x) = \left(\pi( \omega_{G,U} ( g_{j}^{-1} x g_{i}))\right)_{1\leq i,j \leq m}.
	\]
	In particular, we have $\ind_U^G(\pi)(H) = \pi(B)$ and hence
	\[
		\beta_{\psi} = \det(\pi(B)) = \det(\ind_U^G(\pi)(H)) = 
		\prod_{\chi \in \Irr(G)} \alpha_{\chi}^{\langle \ind_U^G \psi, \chi \rangle_G}
	\]
	as desired.
\end{proof}

\subsection{Bounding denominators of reduced norms}
We first consider the case of local fields. The crucial result we have
to show is the following.

\begin{theorem} \label{thm:integrality-mainstep}
	Let $p$ be  prime and let $G$ be a finite group. Let $F$ be a (possibly infinite)
	extension of $\Q_p$ and $R \subset F$ an integrally closed noetherian ring
	with field of fractions $F$. Then for every square matrix $H$ with entries
	in the $R$-order $R[G]$ one has
	\[
		|G'| \nr(H) \in \zeta(R[G]).
	\]
\end{theorem}

\begin{proof}
	We first note that $\nr(H) \in \zeta(F[G])$ so that it suffices to show that
	$|G'| \nr(H)$ has integral coefficients.
	Let $K$ be a finite extension of $\Q_p$ that is a splitting field for $G$ and
	all of its subgroups. By \cite[Proposition 13.14]{MR1322960}
	we may and do assume that $F$ contains $K$.
	
	Choose a normal subgroup $N$ containing the commutator subgroup $G'$ of $G$.
	Let $\eta$ be an irreducible character of $N$ and let $\chi_1, \dots, \chi_s$ 
	be the irreducible characters of $G$ dividing 
	$\ind^G_N \eta$, which are exactly those whose restriction to $N$ is divisible
	by $\eta$. Hence this is compatible with the notation of 
	\S \ref{subsec:Clifford-for-commutator-subgroup}.
	By Corollary \ref{cor:idempotent-equalities} (i) we have
	\begin{equation} \label{eqn:definition-e-of-eta}
		e(\eta) := \sum_{i=1}^s e_{\chi_i} = \sum_{x \in G/G_{\eta}} e_{^x \eta}
		\in K[N].
	\end{equation}
	The ring of integers $\mathcal{O}_K$ in $K$ is contained in $R$ and we clearly
	have 
	\begin{equation} \label{eqn:N-times-idempotent-is-integral}
		|N|e(\eta) \in \mathcal{O}_K[G].
	\end{equation}
	Let $\eta'$ be a second irreducible
	character of $N$. We write $\eta' \sim \eta$ is $\eta'$ is of the form
	${}^x \eta$ for some $x \in G$. Then $e(\eta') = e(\eta)$ if and only
	if $\eta' \sim \eta$, and $e(\eta')e(\eta) = 0$ otherwise. Hence we can write
	$1 = \sum_{\eta \in \Irr(N) / \sim} e(\eta)$ and likewise
	$\nr(H) = \sum_{\eta \in \Irr(N) / \sim} \nr(H)e(\eta)$.
	We now show that
	\begin{equation} \label{eqn:inductive-integrality}
		|N| \nr(H) e(\eta) \in \zeta(R[G])
	\end{equation}
	for each $\eta$. Choosing $N = G'$ then proves the theorem.
	
	We do this by induction on $|G|$. If $G$ is abelian, then 
	$\nr(H) \in R[G] = \zeta(R[G])$ and \eqref{eqn:inductive-integrality}
	follows from \eqref{eqn:N-times-idempotent-is-integral}.
	Now suppose that $G$ is non-abelian.
	By Corollary \ref{cor:idempotent-equalities} (iii) we have
	$e_{\chi_i} = \sum_{x \in G/G_{\eta}} e_{^x \psi_i}$ for all $i$, where as before
	$\chi_i = \ind^G_{G_{\eta}} \psi_i$ for an (irreducible) character
	$\psi_i$ of $G_{\eta}$. Hence the natural diagonal embeddings
	$F e_{\chi_i} \hookrightarrow \oplus_{x \in G / G_{\eta}} F e_{^x \psi_i}$
	induce an embedding
	\begin{equation} \label{eqn:iota-eta}
		\iota_{\eta}: \zeta(F[G]e(\eta)) = \bigoplus_{i=1}^s F e_{\chi_i} 
		\hookrightarrow \bigoplus_{i=1}^s \bigoplus_{x \in G / G_{\eta}} F e_{^x \psi_i}
		= \zeta(F[G_{\eta}] e(\eta)).
	\end{equation}
	Since $\chi_i = \ind^G_{G_{\eta}} \psi_i$ is irreducible, 
	Lemma \ref{lem:nr-under-restriction} implies that
	$\iota_{\eta}(\nr(H) e(\eta)) = \nr(H|_{G_{\eta}}) e(\eta)$.
	Hence by induction we may assume that $G = G_{\eta}$
	(note that for this step it is not sufficient to consider 
	\eqref{eqn:inductive-integrality} with $N = G'$,
	as the commutator subgroup of $G_{\eta}$ might be strictly smaller
	than $G'$). 
	
	Recall the definition of $U_{\psi}$ from \S 
	\ref{subsec:Clifford-for-commutator-subgroup} (where $\psi = \psi_1$
	by convention). We first assume that
	$U_{\psi}$ is a proper subgroup of $G$. As $G/U_{\psi}$ is abelian,
	we may choose a maximal proper subgroup $U$ of $G$ that contains $U_{\psi}$.
	Then $\ell := [G:U]$ is a prime. Let $\rho$ be an irreducible
	constituent of $\res^G_U \psi$ and define $\rho_i$ as in \eqref{eqn:rho_i}.
	Then there is a positive integer $f$ such that 
	$\res^G_U \psi_i = f \sum_{x \in G /G_{\rho}} {}^x \rho_i$
	by \eqref{eqn:res-U-of-psi_i}. By Lemma \ref{lem:ind-rho_i}
	we have that $[G_{\rho}:U] = f^2$ is a square; but $[G_{\rho}:U]$
	divides $[G:U] = \ell$ so that we must have $G_{\rho} = U$ and $f=1$.
	In complete analogy to the definition of $\iota_{\eta}$,
	but using Corollary \ref{cor:idempotents-for-U}, we obtain an embedding
	\begin{equation} \label{eqn:iota-rho}
		\iota_{\rho}: \zeta(F[G]e(\eta)) = \bigoplus_{i=1}^s F e_{\psi_i}
		\hookrightarrow \bigoplus_{i=1}^s \bigoplus_{x \in G / G_{\rho}} F e_{^x \rho_i}
		= \zeta(F[G_{\rho}]e(\eta)).
	\end{equation}
	We have $\psi_i = \ind^G_{G_{\rho}} \rho_i$ for all $i$ by Frobenius
	reciprocity so that again $\iota_{\rho}(\nr(H)e(\eta)) = \nr(H|_{G_{\rho}}e(\eta))$
	by Lemma \ref{lem:nr-under-restriction}. Hence \eqref{eqn:inductive-integrality}
	holds by induction.
	
	We are left with the case $U_{\psi} = G = G_{\eta}$. By Lemma \ref{lem:psi-twists}
	the twists of $\psi$ by irreducible characters $\omega \in \Irr(G/N)$
	are pairwise distinct. Since the set of these twists coincides with the set
	$\{\psi_1, \dots, \psi_s\}$ by Proposition \ref{prop:psi-when-N-contains-commutator}
	(i) and (ii), we have $s = [G:N]$ and thus $m=1$ by part (iii) of the same result.
	Here we recall that $m = m_i$ denotes the multiplicity 
	of $\eta$ in $\res_N^G \psi_i$. By \cite[Kapitel V, Satz 12.5]{MR0224703}
	we can choose an integral representation 
	\[
		\pi: G \rightarrow \GL_{\psi(1)}(\mathcal{O}_K)
	\]
	with character $\psi$. Let
	$\omega \in \Irr(G/N)$. Then likewise $\pi \otimes \omega$ is an
	integral representation with character $\psi \otimes \omega$. 
	Let us denote the image of $g \in G$ in $G/N$ by $\overline{g}$.
	Note that we have an isomorphism
	$F[G/N] \simeq \oplus_{\omega \in \Irr(G/N)} F$. 
	Hence the representations
	$\pi \otimes \omega$ induce an isomorphism of rings
	\begin{equation} \label{eqn:iota_pi}
		\iota_{\pi}: F[G] e(\eta) \simeq \bigoplus _{\omega \in \Irr(G/N)} M_{\psi(1)}(F)
		\simeq M_{\psi(1)}(F[G/N])
	\end{equation}
	which sends $g e(\eta)$ to $\overline{g} \pi(g)$ for all $g \in G$.
	This extends to matrix rings in the obvious way and 
	for each $n \in \N$ we obtain a commutative	diagram
	\[ \xymatrix{
		M_n(F[G]e(\eta)) \ar[rr]^{\iota_{\pi}}_{\simeq} \ar[d]^{\nr} & &
			M_{n \psi(1)}(F[G/N]) \ar[d]^{\det} \\
		\zeta(F[G]e(\eta)) \ar[rr]^{\iota_{\pi}}_{\simeq} & & F[G/N].
	}
	\]
	Now if $H \in M_n(R[G])$ has integral coefficients, then the image of $H e(\eta)$
	under $\iota_{\pi}$ is contained in $M_{n \psi(1)}(R[G/N])$, since $\pi$
	was chosen to be integral. Hence we obtain
	\begin{equation} \label{eqn:integrality-of-nr}
		\iota_{\pi}(\nr(H e(\eta))) = \det(\iota_{\pi}(H e(\eta))) \in R[G/N].
	\end{equation}
	To finish the proof of \eqref{eqn:inductive-integrality} it now suffices to
	show that $\iota_{\pi}^{-1}(\overline{g}) \in |N|^{-1}R[G]$ for all $g \in G$.
	We compute
	\begin{align*}
		\iota_{\pi}^{-1}(\overline{g}) & = 
		\sum_{\omega \in \Irr(G/N)} \omega(\overline{g}) e_{\psi \otimes \omega}\\
		& =  \sum_{\omega \in \Irr(G/N)} \omega(\overline{g}) \frac{\psi(1)}{|G|}
		\sum_{h \in G}(\psi \otimes \omega)(h^{-1}) h \\
		& = \frac{\psi(1)}{|N|} \sum_{h \in G: \overline{h} = \overline{g}} \psi(h^{-1}) h,
	\end{align*}
	where the last equality follows from the well-known orthogonality relation
	\[
		\frac{1}{[G:N]} \sum_{\omega \in \Irr(G/N)} \omega(\overline{g}) 
		\omega(\overline{h}^{-1}) = \delta_{\overline{g}, \overline{h}}.
	\]
	This finishes the proof of the theorem.
\end{proof}

\begin{corollary} \label{cor:bound-integrality-ring}
	Let $G$ be a finite group with commutator subgroup $G'$. Then the following holds.
	\begin{enumerate}
		\item 
		For each prime $p$ we have an inclusion $|G'| \mathcal{I}_p(G) \subseteq 
		\zeta(\Z_p[G])$.
		\item
		We have an inclusion $|G'| \mathcal{I}(G) \subseteq \zeta(\Z[G])$.
	\end{enumerate}
\end{corollary}

\begin{proof}
	(i) follows from Theorem \ref{thm:integrality-mainstep} with $F = \Q_p$.
	(i) for all primes $p$ implies (ii).
\end{proof}

\begin{corollary} \label{cor:hard-inclusion}
	Let $G$ be a finite group with commutator subgroup $G'$. Then the following holds.
	\begin{enumerate}
		\item 
		For each prime $p$ we have $|G'| \in \mathcal{H}_p(G)$.
		\item
		We have that $|G'| \in \mathcal{H}(G)$.
	\end{enumerate}
\end{corollary}

\begin{proof}
	(ii) is again implied by (i) for all $p$ so that it suffices to prove (i).
	So let $p$ be a prime and let $H \in M_n(\Z_p[G])$ be a matrix
	with reduced characteristic polynomial $f_H(X)$.
	Then the generalised adjoint $H^{\ast}$ a priori has entries in $\Q_p[G]$
	and we have to show that $|G'| H^{\ast} \in M_n(\Z_p[G])$.
	As before we let $K$ be a finite extension of $\Q_p$ such that $K$ is a splitting
	field for $G$ and all of its subgroups. The polynomial ring $\mathcal{O}_K[X]$
	is noetherian and integrally closed. So we may apply Theorem 
	\ref{thm:integrality-mainstep} with $R = \mathcal{O}_K[X]$ and 
	the matrix $\mathbf{1}_{n \times n} \otimes X - H \otimes 1 \in M_n(\mathcal{O}_K[X][G])$. Taking \eqref{eqn:redchar-as-nr} into account we
	see that $|G'| f_H(X) \in \zeta(\mathcal{O}_K[X][G])$. Indeed by
	\eqref{eqn:inductive-integrality} we know that for each irreducible character
	$\eta$ of $G'$ we have that
	\[
		|G'| f_H(X) e(\eta) \in \zeta(\mathcal{O}_K[X][G]),
	\]
	where $e(\eta)$ was defined in \eqref{eqn:definition-e-of-eta}.
	The characters $\chi_1, \dots, \chi_s$ that occur in this definition are
	all of the same degree $d$.
	In particular, if we write $f_H(X) e(\eta) = \sum_{j=0}^{dn} \alpha_j X^j$
	with $\alpha_j \in \zeta(K[G] e(\eta))$, then $|G'| \alpha_j$ has integral
	coefficients for all $j$, whence
	\[
		|G'| H^{\ast} e(\eta) = |G'| (-1)^{dn+1} \sum_{j=1}^{dn} \alpha_j e(\eta) H^{j-1}
		\in M_n(\mathcal{O}_K[G]).
	\]
	As this holds for all $\eta$, we get 
	\[
		|G'| H^{\ast} \in M_n(\Q_p[G]) \cap M_n(\mathcal{O}_K[G]) = M_n(\Z_p[G])
	\]
	as desired.
\end{proof}

\begin{proof}[Proof of Theorem \ref{thm:denominator-ideal-capZ}]
	This is now immediate from Corollary \ref{cor:idempotent-equalities}
	and Example \ref{ex:0-star}.
\end{proof}

\begin{remark} \label{rem:more-elements-in-H}
	An inspection of the proof of Theorem \ref{thm:integrality-mainstep}
	shows that for each normal subgroup $N$ of $G$ containing
	$G'$ and each character $\eta \in \Irr(N)$ one has
	\[
		\frac{|N|}{\eta(1)} e(\eta) \in \mathcal{H}(R[G]).
	\]
\end{remark}

\subsection{More on integrality rings}

The aim of this subsection is to show that the integrality ring $\mathcal{I}(G)$
is strictly larger than $\zeta(\Z[G])$ for every non-abelian finite group $G$.
This is not needed in later sections and may be skipped by the reader who is
only interested in our applications to number theory. We first state the main
result of this subsection.

\begin{theorem} \label{thm:integrality-vs-abelian}
	Let $G$ be a finite group. Then $\mathcal{I}(G) = \zeta(\Z[G])$ if and only if
	$G$ is abelian.
\end{theorem}

If $G$ is abelian, then clearly $\mathcal{I}(G) = \zeta(\Z[G]) = \Z[G]$.
So the main task of this subsection is to show that the integrality ring
is strictly larger than $\zeta(\Z[G])$ whenever $G$ is non-abelian.
This means that we aim to construct elements $x \in \Z[G]$ such that
their reduced norm $\nr(x)$ is not integral. 

For a finite group $G$ and a positive integer $d$ we set
\[
E_d = E_{G,d} := \sum_{\chi(1)=d} e_{\chi},
\]
where the sum runs over all irreducible characters of $G$ of degree $d$.
In particular, one has $E_1 = |G'|^{-1} \Tr_{G'}$.
We will need the following interesting consequence of 
Corollary \ref{cor:idempotent-equalities}.

\begin{lemma} \label{lem:E_d-integrality}
	Let $G$ be a finite group. Then for each $d \in \N$ one has
	\[
	|G'| E_d \in \Z[G'].
	\]
\end{lemma}

\begin{proof}
	Let $K$ be a finite extension of $\Q$ that is a 
	splitting field for $G$ and all of its subgroups. 
	Let $X_d$ be the set of all irreducible characters $\eta$ of $G'$
	such that there is a $\chi \in \Irr(G)$ of degree $d$ with 
	$\langle \res_{G'}^G \chi, \eta \rangle >0$. Let $\eta \in X_d$. As before
	we set $e(\eta) := \sum_{i=1}^s e_{\chi_i}$, where the sum runs
	over all irreducible characters of $G$ whose restriction to $G'$
	is divisible by $\eta$. The characters 
	$\chi_1, \dots, \chi_s$ 
	then all have degree $d$ and Corollary \ref{cor:idempotent-equalities} (i)
	implies that $|G'| e(\eta) \in \mathcal{O}_K[G']$, 
	as was essentially already observed
	in \eqref{eqn:N-times-idempotent-is-integral}.
	As in the proof of Theorem \ref{thm:integrality-mainstep} we write
	$\eta \sim \eta'$ if and only if $\eta' = {}^x \eta$ for some $x \in G$
	if and only if $e(\eta) = e(\eta')$.
	Then we have $E_d = \sum_{\eta \in X_d / \sim} e(\eta)$ and thus
	$|G'| E_d  \in \mathcal{O}_K[G']$.
	Since $\sigma(E_d) = E_d$ for each Galois automorphism $\sigma \in G_{\Q}$, we
	see that $|G'| E_d \in \Q[G] \cap \mathcal{O}_K[G'] = \Z[G']$ as desired.
\end{proof}

\begin{corollary} \label{cor:squarefree}
	Let $G$ be a finite group. Then the following holds.
	\begin{enumerate}
		\item 
		If $\mathcal{I}(G) = \zeta(\Z[G])$,
		then the order of the commutator subgroup is squarefree.
		\item 
		If $p$ is a prime such that $\mathcal{I}_p(G) = \zeta(\Z_p[G])$, 
		then $|G'|$ is not divisible by $p^2$.
		\item
		If $\mathcal{I}(G) = \zeta(\Z[G])$,
		then $G$ is solvable.
	\end{enumerate} 
\end{corollary}

\begin{proof}
	Assume that $p$ is a prime such that $p^2$ divides $|G'|$. Then $p^{-1}|G'|$
	is an integer and
	\begin{align*}
		\nr(p^{-1}|G'|) & =  \sum_{\chi \in \Irr(G)}(p^{-1}|G'|)^{\chi(1)} e_{\chi}\\
		& =  p^{-1}|G'| E_1 + \sum_{d \geq 2} (p^{-1}|G'|)^d E_d\\
		& =  p^{-1} \Tr_{G'} + \sum_{d \geq 2} (p^{-1}|G'|)^{d-2} \frac{|G'|}{p^2} |G'| E_d.
	\end{align*}
	By Lemma \ref{lem:E_d-integrality} the sum on the right-hand side is integral.
	Since the first term has denominator $p$, we see that 
	$\nr(p^{-1}|G'|) \in \mathcal{I}(G) \subseteq \mathcal{I}_p(G)$	is not integral.
	This shows (i) and (ii). If $\mathcal{I}(G) = \zeta(\Z[G])$,
	then the cardinality of $G'$ is squarefree by (i) and thus all Sylow
	subgroups of $G'$ are cyclic. Then $G'$ is solvable by Hölder's Theorem
	\cite[Kapitel IV, Satz 2.9]{MR0224703}. Hence also $G$ is solvable.
\end{proof}

We need a few more preliminary results. In particular, the induction argument
will lead to two types of rather special groups that we have to treat separately. 

\begin{lemma} \label{lem:first-special-group}
	Let $p$ be a prime and let $G$ be a finite $p$-group. Suppose that $G'$
	is cyclic of order $p$ and contained in the centre $Z(G)$ of $G$.
	Moreover, assume that every non-linear irreducible character of $G$ is
	faithful. Then for every $x \in G \setminus Z(G)$ one has
	$\nr(x) \not\in \Z_p[G]$.
\end{lemma}

\begin{proof}
	Let $\chi$ be an irreducible ($\C_p$-valued) character of $G$ of 
	degree $d>1$. 
	Choose a representation
	$\pi_{\chi}: G \rightarrow \GL_{d}(\C_p)$ with character $\chi$. Since $G'$
	is central in $G$, also $\pi_{\chi}(g)$ is central for every $g \in G'$
	and thus of the form $\lambda(g) \mathbf{1}_{d \times d}$, where
	$\lambda: G' \rightarrow \C_p^{\times}$ is a linear character of $G'$.
	In other words, we have $\res^G_{G'} \chi = \chi(1) \lambda$.
	Since $\chi$ is faithful, the character $\lambda$ cannot be trivial.
	
	Now let $x \in G \setminus Z(G)$. As $x$ is not central, we may choose
	$y \in G$ such that the commutator $[x,y] := xyx^{-1}y^{-1}$ is non-trivial.
	Since $\chi$ is a class function and $[x,y] \in G'$, we see that
	\[
		\chi(x^{-1}) = \chi(yx^{-1}y^{-1}) = \chi(x^{-1}[x,y]) = \chi(x^{-1})\lambda([x,y]).
	\]
	By our choice of $x$ and $y$ we have $\lambda([x,y]) \not=1$ so that
	$\chi(x^{-1})$ vanishes. In particular, the coefficient of the associated 
	idempotent $e_{\chi}$ at $x$ vanishes. 
	Hence the coefficient of $\nr(x)$ at $x$ is given by
	\[
		\sum_{\genfrac{}{}{0pt}{}{\chi \in \Irr(G)}{\chi(1) = 1} } \chi(x) \frac{1}{|G|} \chi(x^{-1})
		= \frac{1}{|G'|} = \frac{1}{p} \not\in \Z_p
	\]
	and thus $\nr(x) \not\in \Z_p[G]$ as desired.
\end{proof}

We recall the definition and some basic facts about Frobenius groups.

\begin{definition}
	A \emph{Frobenius group} is a finite group $G$ with a 
	proper nontrivial subgroup $H$
	such that $H \cap gHg^{-1}=\{ 1 \}$ for all $g \in G \setminus H$,
	in which case $H$ is called a \emph{Frobenius complement}.
\end{definition}

\begin{theorem}\label{thm:frob-kernel}
	Let $G$ be a Frobenius group with Frobenius complement $H$. 
	Then the following holds.
	\begin{enumerate}
		\item 
		There is a unique normal subgroup $N$ in $G$, known as the Frobenius kernel, 
		such that $G$ is a semi-direct product $N \rtimes H$.
		\item 
		If $\chi$ is an irreducible character of $G$ such that  
		$N \not \leq \ker(\chi)$ then $\chi= \ind_{N}^{G}\psi$ for some 
		$1 \neq \psi \in \Irr(N)$.
	\end{enumerate}
\end{theorem}

\begin{proof}
	Part (i) is due to Frobenius, see \cite[Kapitel V, Satz 7.6]{MR0224703}.
	For (ii) see \cite[Proposition 14.4]{MR632548}.
\end{proof}

\begin{lemma} \label{lem:second-special-group}
	Let $G \simeq G' \rtimes A$ be a Frobenius group with Frobenius kernel $G'$
	and (necessarily abelian) Frobenius complement $A$. Then for every
	$x \in G \setminus G'$ one has that $\nr(x) \not\in \Z[G]$.
\end{lemma}

\begin{proof}
	Let $\chi \in \Irr(G)$ be non-linear. Then $G'$ is not contained in the kernel
	of $\chi$ so that $\chi$ is induced from a non-trivial irreducible character
	of $G'$ by Theorem \ref{thm:frob-kernel} (ii).
	Since $G'$ is normal, we have $\chi(x^{-1}) = 0$ for every 
	$x \in G \setminus G'$. It follows as in the proof of 
	Lemma \ref{lem:first-special-group} that the coefficient of $\nr(x)$ at
	$x$ is given by $|G'|^{-1}$. Hence $\nr(x)$ is not integral.
\end{proof}

\begin{example} 
	Let $G = S_3$ be the symmetric group on $3$ letters. Then $G$ is a Frobenius
	group with Frobenius kernel $G' = A_3$, the alternating group
	on $3$ letters, and the proof of Lemma 
	\ref{lem:second-special-group} shows that the coefficient of
	$\nr(\tau)$ at $\tau$ equals $1/3$ for every transposition $\tau$.
\end{example}

\begin{example} \label{ex:Aff(q)}
	More generally, let $q\not= 2$ be a prime power. Then the group of affine 
	transformations $\mathrm{Aff}(q) = \F_q \rtimes \F_q^{\times}$ on the
	finite field $\F_q$ is a Frobenius group with Frobenius kernel
	$\mathrm{Aff}(q)' = \F_q$ and Frobenius complement $\F_q^{\times}$.
	Then the coefficient of $\nr(x)$ at $x$ equals $1/q$ for every 
	$x \in \mathrm{Aff}(q) \setminus \mathrm{Aff}(q)'$.
	Note that $\mathrm{Aff}(3) \simeq S_3$ and $\mathrm{Aff}(4) \simeq A_4$,
	the alternating group on $4$ letters.
\end{example}

\begin{lemma} \label{lem:two-cases}
	Let $G$ be a finite, non-abelian, solvable group. Let $p$ be a prime
	and assume that $G'$ is the unique minimal normal subgroup of $G$
	and is cyclic or order $p$. Then either (i) $G$ is a $p$-group
	with $G' \subseteq Z(G)$ or (ii) $G$ is a Frobenius group with Frobenius
	kernel $G'$ (and abelian Frobenius complement $A$).
\end{lemma}

\begin{proof}
	Let us first assume that $G$ is nilpotent. Then $G$ decomposes into a direct
	product $G = P \times H$, where $P$ is a $p$-group and the cardinality of $H$
	is coprime to $p$. Since $G'$ is contained in every normal subgroup of $G$
	and is cyclic of order $p$, we must have $H=1$. So $G$ is a $p$-group.
	The centre of a $p$-group is a non-trivial normal subgroup  
	so that it contains the unique minimal normal subgroup $G'$.
	Hence we are in case (i). 
	
	Recall that the Frattini subgroup
	$\Phi(G)$ is the intersection of all maximal subgroups of $G$.
	If $G$ is not nilpotent, then the commutator subgroup $G'$ is not contained
	in $\Phi(G)$ by a theorem of Wielandt \cite[Kapitel III, Satz 3.11]{MR0224703}.
	Hence $\Phi(G)$ is trivial and there is a maximal subgroup $A$ that does not
	contain $G'$. Then $G = AG'$ by maximality of $A$ and $A\cap G' = 1$,
	as $G' \subseteq A$ otherwise. Hence $G$ is a semi-direct product $G' \rtimes A$
	so that $A \simeq G/G'$ is abelian. It remains to show that for every
	$g \in G \setminus A$ we have $A \cap g^{-1}Ag = 1$. Suppose to the contrary
	that there is an $a \in A$ such that $g^{-1}ag \in A$. Then
	$a^{-1}g^{-1}ag \in G' \cap A = 1$ so that $a$ commutes with $g$.
	Since $A$ is a maximal subgroup of $G$ and is abelian, the element
	$a$ is indeed central. In particular, it generates a normal subgroup
	of $G$ and thus contains $G'$. This contradicts the fact that
	$G' \cap A = 1$. So if $G$ is not nilpotent, we are in case (ii).
\end{proof}

\begin{proof}[Proof of Theorem \ref{thm:integrality-vs-abelian}]
	We have to show that the integrality ring
	$\mathcal{I}(G)$ is strictly larger
	than the centre $\zeta(\Z[G])$. By Corollary \ref{cor:squarefree} we
	may and do assume that $G$ is solvable and that $|G'|$ is squarefree.
	We will now show by induction that there is an $x \in \Z[G]$ 
	such that $\nr(x) = \nr_G(x) \not\in \Z[G]$.
	Suppose first that $G$ contains a proper normal subgroup $N$ such that $G/N$ is non-abelian. Then we may assume by induction that there is an 
	$\overline{x} \in \Z[G/N]$ such that $\nr_{G/N}(\overline{x}) \not \in \Z[G/N]$.
	Let $x \in \Z[G]$ be any lift of $\overline{x}$. Then $\nr_G(x)$ cannot be integral
	as it maps to $\nr_{G/N}(\overline{x})$ under the canonical projection map
	$\Q[G] \rightarrow \Q[G/N]$. If no such $N$ exists, the commutator subgroup
	is contained in every non-trivial normal subgroup, i.e.~it is the unique
	minimal normal subgroup of $G$. By \cite[Kapitel I, Satz 9.13]{MR0224703}
	the minimal normal subgroups of a solvable group are $p$-groups.
	Since $|G'|$ is squarefree, it must be cyclic of order $p$.
	Moreover, the non-linear irreducible characters of $G$ are faithful,
	as their kernel is a normal subgroup that does not contain $G'$.
	So it remains to treat the two cases given by Lemma \ref{lem:two-cases}.
	These are covered by Lemmas \ref{lem:first-special-group} and
	\ref{lem:second-special-group}, respectively.
\end{proof}

Let us briefly discuss the case of $p$-adic groups rings.
If $p$ does not divide the cardinality of the commutator subgroup,
then $\mathcal{I}_p(G) = \zeta(\Z_p[G])$ 
by Proposition \ref{prop:HpG-equals-centre}. 
The converse implication seems
to be likely, but the above argument fails in several respects.
For instance, we can only conclude that $|G'|$ is not divisible by $p^2$,
but this does not imply that $G'$ is solvable. Moreover, if $N$ is a normal
subgroup of $G$ such that $G/N$ is non-abelian, the commutator subgroup of the
quotient may be of cardinality coprime to $p$ even if this does not hold for $G'$.
What we can offer, is the following variant of Theorem \ref{thm:integrality-vs-abelian}.

\begin{prop} \label{prop:integrality-rings-nilpotent}
	Let $G$ be a finite nilpotent group and let $p$ be a prime. Then
	$\mathcal{I}_p(G) = \zeta(\Z_p[G])$ if and only if $p \nmid |G'|$.
\end{prop}

\begin{proof}
	Let $G$ be a finite nilpotent group such that $|G'|$ 
	is divisible by the prime $p$. It suffices to construct an element
	$x \in \Z_p[G]$ such that $\nr_G(x) \not\in \Z_p[G]$. The group $G$ decomposes
	into a direct product $G \simeq P \times H$, where $H$ is of cardinality prime
	to $p$ and $P$ is a finite non-abelian $p$-group. If there is
	$\overline x \in \Z_p[P]$ such that $\nr_P(\overline x)\not\in \Z_p[P]$,
	then any lift $x \in \Z_p[G]$ of $\overline x$ under the 
	projection $\Z_p[G] \rightarrow \Z_p[P]$ will do. So we may assume that
	$G$ is a $p$-group. By a similar argument we may assume that there is no
	proper normal subgroup $N$ of $G$ such that $G/N$ is non-abelian.
	Hence $G'$ is contained in every non-trivial normal subgroup of $G$ and
	every non-linear irreducible character of $G$ must be faithful.
	By Corollary \ref{cor:squarefree} we may assume that $p^2$ does not divide
	$|G'|$. Thus $G'$ is a cyclic group of order $p$.
	Now Lemma \ref{lem:first-special-group} implies the result.
\end{proof}

We are not aware of any example of a finite group $G$, where the conclusion of
Proposition \ref{prop:integrality-rings-nilpotent} fails. For instance, we have
seen in Example \ref{ex:special-linear-group} that for the special linear group
$\mathrm{SL}_2(\F_3)$ the commutator subgroup is of order $8$ and $\mathcal{I}_2(\mathrm{SL}_2(\F_3))$
is indeed strictly larger than $\zeta(\Z_2[\mathrm{SL}_2(\F_3)])$.
We give two further examples.

\begin{example}
	Let $G = G' \rtimes A$ be a Frobenius group with Frobenius kernel $G'$
	such as the group of affine transformations considered in Example \ref{ex:Aff(q)}.
	Then the proof of Lemma \ref{lem:second-special-group} shows that 
	$\nr(x)$ has a denominator $|G'|$ at $x$ for every $x \in G \setminus G'$.
	In particular, $\mathcal{I}_p(G)$ is strictly larger than $\zeta(\Z_p[G])$
	whenever $|G'|$ is divisible by $p$.
\end{example}

\begin{example}
	Let $G = M$ be the Monster group, which is the largest sporadic simple group.
	We show that $\mathcal{I}_p(M)$ is strictly larger than $\zeta(\Z_p[M])$
	whenever $p$ divides
	\[
		|M'| = |M| = 2^{46} \cdot 3^{20} \cdot 5^{9} \cdot 7^{6} \cdot 11^{2} \cdot 13^{3} \cdot 17 \cdot 19 \cdot 23 \cdot 29 \cdot 31 \cdot 41 \cdot 47 \cdot 59 \cdot 71.
	\]
	By Corollary \ref{cor:squarefree} (ii) we only have to consider primes $p$
	in the set $P$ of those primes that divide $|M|$ exactly once.
	For $n \in \Z$ the coefficient of the reduced norm $nr_{M}(n) $ at $1$ 
	is given by $|M|^{-1} A(n)$ where
	\[
		A(n) := \sum_{\chi \in Irr(M)} n^{\chi(1)} \chi(1)^{2}.
	\]
	The degrees of the characters are too large to compute this expression easily.
	However, it suffices to show that for each $p \in P$ there is an $n$ such
	that $A(n) \mod p \not= 0$. Choosing $n = -1$ we computed the following
	values for $A(-1) \mod p$ using Python:
	\begin{center}
		\begin{tabular}{| l | l | l | l | l | l | l | l | l | l |}
			\hline
			$p$ & 17 & 19 & 23 & 29 & 31 & 41 & 47 & 59 & 71 \\ \hline
			$A(-1) \mod p$ & 1 & 1 & 9&15 &10 & 5& 17& 31& 51\\ \hline
		\end{tabular}
	\end{center}
	In particular, we see that $A(-1)$ is not divisible by any $p \in P$
	as desired.
\end{example}

\section{Conjectures and results on integrality and annihilation} \label{sec:int-and-ann}

\subsection{Equivariant Artin $L$-functions} 
Let $L/K$ be a finite Galois extension of number fields with Galois group $G$.
For each place $v$ of $K$ we fix a place $w$ of $L$ above $v$ and write $G_{w}$ and $I_{w}$
for the decomposition group and the inertia subgroup of $G$ at $w$, respectively.  
When $w$ is a finite place, we choose a lift $\sigma_{w} \in G_{w}$ of the Frobenius automorphism at $w$. We write $\mathfrak{P}_{w}$ for the associated prime ideal 
in $L$ and $L(w)$ for the residue field at $w$.
For a finite place $v$ of $K$ we denote the cardinality of its residue field 
$K(v)$ by $\mathrm{N}v$.

Let $S$ be a finite set of places of $K$ containing the set $S_{\infty}$
of all archimedean places of $K$.
For $\chi \in \Irr_{\C}(G)$ let $V_{\chi}$ be a $\C[G]$-module with character $\chi$.
The $S$-truncated Artin $L$-function $L_{S}(s,\chi)$
is defined as the meromorphic extension to the whole complex plane 
of the holomorphic function given by the Euler product
\[
L_{S}(s,\chi) = \prod_{v \notin S} \det (1 - (\mathrm{N}v)^{-s}\sigma_{w} \mid V_{\chi}^{I_{w}})^{-1}, \quad \mathrm{Re}(s)>1.
\]
Recall that there is a canonical isomorphism $\zeta(\C[G]) \simeq \prod_{\chi \in \Irr_{\C}(G)} \C$.
The equivariant $S$-truncated Artin $L$-function is defined to be the meromorphic $\zeta(\C[G])$-valued function
\[
L_{S}(s) := (L_{S}(s,\chi))_{\chi \in \Irr_{\C}(G)}.
\]

Now suppose that $T$ is a second finite set of places of $K$ such that $S \cap T = \emptyset$.
Then we define 
\begin{equation}\label{eq:def-delta-T}
	\delta_{T}(s,\chi) := \prod_{v \in T} \det(1 - (\mathrm{N}v)^{1-s} \sigma_{w}^{-1} \mid V_{\chi}^{I_{w}}) \quad  \textrm{ and }  \quad
	\delta_{T}(s) := (\delta_{T}(s,\chi))_{\chi \in \Irr_{\C}(G)}. 
\end{equation}
Let $x \mapsto x^{\#}$ denote the anti-involution on $\C[G]$ induced by $g \mapsto g^{-1}$ for $g \in G$.
The $(S,T)$-modified $G$-equivariant Artin $L$-function is defined to be
\[
\Theta_{S,T}(s) := \delta_{T}(s) \cdot L_{S}(s)^{\#}.
\]
Note that $L_{S}(s)^{\#} = (L_{S}(s,\check{\chi}))_{\chi \in \Irr_{\C} (G)}$ 
where $\check \chi$ denotes the character contragredient to $\chi$.
Evaluating $\Theta_{S,T}(s)$ at integers $r \leq 0$ gives $(S,T)$-modified Stickelberger elements
\[
\theta_{S}^{T}(r) := \Theta_{S,T}(r) \in \zeta(\Q[G]).
\]
Note that a priori we only have $\theta_{S}^{T}(r) \in \zeta(\C[G])$, but by a result of Siegel \cite{MR0285488} we know that
$\theta_{S}^{T}(r)$ in fact belongs to $\zeta(\Q[G])$. 
If $T$ is empty, we abbreviate $\theta^{T}_{S}(r)$ to $\theta_{S}(r)$.
If the extension $L/K$ is not clear from context, we will also write
$\theta_{S}^{T}(L/K,r)$, $L_S(L/K,s)$, $\delta_{T}(L/K,s)$ etc.

\subsection{The integrality conjecture}
If $S$ is any finite set of places of $K$, we write $S(L)$ for the set of
places of $L$ lying above those in $S$. 
Assume that $S$ contains $S_{\infty}$
and let $T$ be a second finite set of places 
of $K$ such that $T$ and $S$ are disjoint.
We denote the ring of $S(L)$-integers in $L$ by $\mathcal{O}_{L,S}$ and let
$\mathcal{O}_{L,S,T}^{\times}$ be the subgroup of $\mathcal{O}_{L,S}^{\times}$
comprising those $S(L)$-units which are congruent to $1$ modulo each place
in $T(L)$.

We will usually impose the following hypotheses on the sets $S$ and $T$.
Note that condition (iii) in particular implies that $T$ is non-empty.
\begin{hypothesis*}
	Let $S$ and $T$ be finite sets of places of $K$. 
	We say that $\Hyp(S,T)$ is satisfied if
	(i)
	$S_{\ram} \cup S_{\infty} \subseteq S$,
	(ii)
	$S \cap T = \emptyset$, and
	(iii)
	$\mathcal{O}_{L,S,T}^{\times}$ is torsionfree.
\end{hypothesis*}

If $G$ is abelian and $r \leq 0$ is an integer, 
then independent work of Pi.\ Cassou-Nogu\`es \cite{MR524276} 
and of Deligne and Ribet \cite{MR579702} (see also Barsky \cite{MR525346}) shows that
\[
\theta_S^T(r) \in \Z[G]
\]
whenever $\Hyp(S,T)$ is satisfied. For general $G$ this is no longer true,
but we have the following conjecture due to the second named author.
It originates from the more general \cite[Conjecture 2.1]{MR2976321} (if $r=0$) and 
\cite[Conjecture 2.11]{MR2801311} (if $r<0$). See also 
\cite[\S 3]{MR3552493}.

\begin{conj}[Integrality Conjecture] \label{conj:integrality}
	Let $L/K$ be a Galois extension of number fields with Galois group $G$
	and let $r \leq 0$ be an integer. Then for
	every pair $S,T$ of finite sets of places of $K$ satisfying $\Hyp(S,T)$ we have
	$\theta_{S}^{T}(r) \in \mathcal{I}(G)$.
\end{conj}

\begin{remark}
	Let $d_G$ be the lowest common multiple of $|G'|$
	and the cardinalities of the conjugacy classes of $G$.
	If Conjecture \ref{conj:integrality} holds, then Corollary 
	\ref{cor:bound-integrality-ring} implies that the denominators of
	$\theta_{S}^{T}(r)$ are bounded by $|G'|$, that is $|G'| \theta_{S}^{T}(r) \in 
	\zeta(\Z[G])$. A variant of this conjecture in the case $r=0$ 
	has been formulated by Dejou
	and Roblot \cite{MR3208394}, which suggests that one has to replace $|G'|$
	by the weaker bound $d_G$. This was motivated by the following
	example which is taken from \cite[p. 83]{Dejou_thesis}.
	Let $L$ be the splitting field of the polynomial 
	$\mathtt{x^8 - 2x^7+x^6+x^5-x^4+2x^3+4x^2-16x+16}$.
	Then $L/\Q$ is a Galois extension with Galois group 
	$G \simeq \mathrm{SL}_2(\F_3)$ so that $|G'| = 8$ and $d_G = 24$.
	The only roots of unity
	in $L$ are $\pm 1$ and the only ramified prime is $853$.
	Using the notation of Example \ref{ex:special-linear-group} one has
	\[
		\theta_{S_{\infty} \cup S_{\ram}}(0) = \frac{1}{3} (C_1 - C_2)
		(-42 C_1 + 2 C_3 - 22 C_5 - 21 C_6).
	\]
	Because of the factor $1/3$ it seems that one has to multiply by $d_G$
	rather than $|G'|$. (Note that there is no set $T$ in this example,
	but the number of roots of unity in $L$ is $2$ so that
	there is a totally decomposed prime $p \not= 2$ such that
	$\delta_{\{p\}}(0) \in \zeta(\Z_3[G])^{\times}$ by 
	\cite[Lemma 2.2]{MR2976321}; hence the $3$-adic denominators
	do not change if we multiply the equation
	by $\delta_{\{p\}}(0)$.)
	However, using equalities of the form $C_3 = C_2 C_5$ 
	this can be rewritten as
	\[
	\theta_{S_{\infty} \cup S_{\ram}}(0) = 
	-14 C_1 + 14 C_2 + 8 C_3 + 7 C_4 - 8 C_5 - 7 C_6
	\]
	which is indeed integral.
\end{remark}

\subsection{Brumer's conjecture and generalisations}
Let $L/K$ be a finite Galois extension of number fields with Galois group $G$.
Assume that $S$ contains $S_{\infty}$
and let $T$ be a second finite set of places 
of $K$ such that $T$ and $S$ are disjoint. Write $S_f := S\setminus S_{\infty}$
for the subset of $S$ comprising all finite places in $S$.
We let $\cl_L^T$ be the ray class group of $L$ associated to the modulus
$\mathfrak{M}_{L}^{T} := \prod_{w \in T(L)} \mathfrak{P}_{w}$ and write
$\cl_{L,S}^T$ for the cokernel of the natural map
$\Z[S_f(L)] \rightarrow \cl_L^T$ which sends each place $w \in S_f(L)$ to
the class of the associated prime ideal $\mathfrak{P}_w$.
If $T$ is empty we abbreviate $\cl_{L,S}^T$ and $\cl_{L}^{T}$ to $\cl_{L,S}$ and
$\cl_{L}$, respectively.
All these modules are equipped with a natural $G$-action
and we have the following exact sequence of finitely generated $\Z[G]$-modules
\begin{equation}\label{eqn:ray-class-sequence}
	0 \longrightarrow \mathcal{O}_{L,S,T}^{\times} \longrightarrow \mathcal{O}_{L,S}^{\times} \longrightarrow \prod_{w \in T(L)} L(w)^{\times}
	\stackrel{\nu}{\longrightarrow} \cl_{L,S}^{T} \longrightarrow \cl_{L,S} \longrightarrow 0,
\end{equation}
where the map $\nu$ lifts an element $\overline x \in \prod_{w \in T(L)} L(w)^{\times} \simeq (\mathcal{O}_{L} / \mathfrak{M}_{L}^{T})^{\times}$ to
$x \in \mathcal{O}_{L}$ and
sends it to the ideal class $[(x)] \in \cl_{L,S}^{T}$ of the principal ideal $(x)$.

If $G$ is abelian, then Brumer's conjecture
simply asserts that $\theta_S^T(0)$ annihilates the class group $\cl_L$.
However, we will mainly deal with the following strengthening 
that was stated by Tate \cite{MR782485} and is known as the Brumer--Stark conjecture.

\begin{conj}\label{conj:Brumer--Stark}
	Let $L/K$ be an abelian extension of number fields with Galois group $G$. Then for
	every pair $S,T$ of finite sets of places of $K$ satisfying $\Hyp(S,T)$ we have
	$\theta_{S}^{T}(0) \in \Ann_{\Z[G]}(\cl_{L}^{T})$.
\end{conj}

By groundbreaking work of Dasgupta and Kakde \cite[Corollary 3.8]{dasgupta-kakde}
Conjecture \ref{conj:Brumer--Stark} holds away from its $2$-primary part.

\begin{theorem}[Dasgupta--Kakde] \label{thm:Dasgupta-Kakde}
	Let $L/K$ be an abelian extension of number fields with Galois group $G$. Then for
	every pair $S,T$ of finite sets of places of $K$ satisfying $\Hyp(S,T)$ we have
	$\theta_{S}^{T}(0) \in \Ann_{\Z[1/2][G]}(\Z[\half] \otimes_{\Z} \cl_{L}^{T})$.
\end{theorem}

Non-abelian generalisations of Brumer's conjecture have been formulated by
the second named author \cite{MR2976321}, by Dejou and Roblot \cite{MR3208394}
and in even greater generality by Burns \cite{MR2845620}. We consider 
the following generalisation of Conjecture \ref{conj:Brumer--Stark}.

\begin{conj}\label{conj:Brumer--Stark-nonab}
	Let $L/K$ be a Galois extension of number fields with Galois group $G$. Then for
	every pair $S,T$ of finite sets of places of $K$ satisfying $\Hyp(S,T)$ we have
	\[
		\mathcal{H}(G) \cdot \theta_{S}^{T}(0) \subseteq \Ann_{\zeta(\Z[G])}(\cl_{L}^{T}).
	\]
\end{conj}

Note that the Integrality Conjecture implies that the set on the left is indeed
contained in the centre of $\Z[G]$. By Corollary \ref{cor:bound-integrality-ring}
we thus expect that in particular $|G'| \theta_{S}^{T}(0)$ has integral
coefficients and annihilates the ray class group $\cl_L^T$.

\subsection{Derived categories and Galois cohomology} \label{sec:derived}
Let $\Lambda$ be a noetherian ring and $\PMod(\Lambda)$ be the category of all finitely
generated projective $\Lambda$-modules. We write $\mathcal D(\Lambda)$ for the derived category
of $\Lambda$-modules and $\mathcal C^b(\PMod(\Lambda))$ for the category of bounded complexes
of finitely generated projective $\Lambda$-modules.
Recall that a complex of $\Lambda$-modules is called perfect if it is isomorphic in $\mathcal D(\Lambda)$
to an element of $\mathcal C^b(\PMod(\Lambda))$.
We denote the full triangulated subcategory of $\mathcal D(\Lambda)$
comprising perfect complexes by $\mathcal D^{\perf}(\Lambda)$.

If $M$ is a $\Lambda$-module and $n$ is an integer, we write $M[n]$ for the complex
\[
\cdots \rightarrow 0 \rightarrow 0 \rightarrow M \rightarrow 0 \rightarrow 0 \rightarrow \cdots,
\]
where $M$ is placed in degree $-n$. 

Let $L/K$ be a Galois extension
of number fields with Galois group $G$.
For a finite set $S$ of places of $K$ containing $S_{\infty}$
we let $G_{L,S}$ be the Galois group over $L$ of the maximal extension of $L$
that is unramified outside $S(L)$.
For any topological $G_{L,S}$-module $M$
we write $R\Gamma(\mathcal O_{L,S}, M)$ for the complex
of continuous cochains of $G_{L,S}$ with coefficients in $M$.
If $F$ is a field and $M$ is a topological $G_{F}$-module, we likewise define
$R\Gamma(F,M)$ to be the complex
of continuous cochains of $G_{F}$ with coefficients in $M$.

If $F$ is a global or a local field of characteristic zero, and $M$ is a discrete or
a compact $G_F$-module, then for $n \in \Z$ we denote the $n$-th Tate twist of $M$ by $M(n)$.
Now let $p$ be a prime and suppose that $S$ also contains all $p$-adic places of $K$.
Then we will consider the complexes
$R\Gamma(\mathcal O_{L,S}, \Z_p(n))$ in $\mathcal D(\Z_p[G])$.
Note that for an integer $i$ the cohomology group in degree $i$ of
$R\Gamma(\mathcal O_{L,S}, \Z_p(n))$ naturally identifies with
$H^i_{\et} (\mathcal{O}_{L,S}, \Z_p(n))$, the $i$-th \'etale cohomology group of the affine scheme
$\Spec(\mathcal{O}_{L,S})$ with coefficients in the \'etale $p$-adic sheaf $\Z_p(n)$.

\subsection{The Coates--Sinnott conjecture and generalisations}
For an integer $n \geq 0$
we let $K_n(\mathcal{O}_{L,S})$ denote the Quillen $K$-groups
of $\mathcal{O}_{L,S}$.
Coates and Sinnott \cite{MR0369322} 
formulated the following analogue of Brumer's conjecture for higher $K$-groups.

\begin{conj}[Coates--Sinnott] \label{conj:Coates--Sinnott}
	Let $L/K$ be a finite abelian extension of number fields with Galois group $G$.
	Let $r$ be a negative integer. Then for
	every pair $S,T$ of finite sets of places of $K$ satisfying $\Hyp(S,T)$ we have
	$\theta_{S}^T(r) \in \Ann_{\Z[G]}(K_{-2r}(\mathcal{O}_{L,S}))$.
\end{conj}

Building upon work of Dasgupta and Kakde \cite{dasgupta-kakde}
it has recently been shown by Johnston and the second named author that
the Coates--Sinnott conjecture holds away from its $2$-primary part.

\begin{theorem}[{\cite[Corollary 12.6]{abelian-MC}}] \label{thm:Coates-Sinnott}
	Let $L/K$ be a finite abelian extension of number fields with Galois group $G$.
	Let $r$ be a negative integer. Then for
	every pair $S,T$ of finite sets of places of $K$ satisfying $\Hyp(S,T)$ we have
	$\theta_{S}^T(r) \in \Ann_{\Z[\frac{1}{2}][G]}(\Z[\half] \otimes_{\Z} K_{-2r}(\mathcal{O}_{L,S}))$.
\end{theorem}

The following is a non-abelian generalisation of Conjecture \ref{conj:Coates--Sinnott}
and is a consequence of \cite[Conjecture 2.11]{MR2801311} 
away from its $2$-primary part.

\begin{conj} \label{conj:Coates--Sinnott-nonab}
	Let $L/K$ be a finite Galois extension of number fields with Galois group $G$.
	Let $r$ be a negative integer. Then for
	every pair $S,T$ of finite sets of places of $K$ satisfying $\Hyp(S,T)$ we have
	\[
		\mathcal{H}(G)	\cdot \theta_{S}^T(r) \subseteq \Ann_{\zeta(\Z[G])}(K_{-2r}(\mathcal{O}_{L,S})).
	\]
\end{conj}

Let $p$ be an odd prime and suppose in addition that $S$ contains the set $S_p$ of all
$p$-adic places of $K$.
For any negative integer $r$ and $i=0,1$ Soul\'e \cite{MR553999}
has constructed canonical $G$-equivariant $p$-adic Chern class maps
\begin{equation} \label{eqn:Chern-class-maps}
	\Z_{p} \otimes_{\Z} K_{i-2r}(\mathcal{O}_{L,S}) \longrightarrow 
	H^{2-i}_{\et} (\mathcal{O}_{L,S}, \Z_{p}(1-r)).
\end{equation}
Soul\'e  proved surjectivity
and by the norm residue isomorphism theorem \cite{MR2529300} (formerly
known as the Quillen--Lichtenbaum Conjecture) these maps are
actually isomorphisms. 
Hence the $p$-parts of Conjectures \ref{conj:Coates--Sinnott} and
\ref{conj:Coates--Sinnott-nonab} can be formulated
in terms of \'etale cohomology. As in the case of Brumer's conjecture we will
consider a strengthening as follows. For each $n \in \Z$ define a complex 
\[
	R\Gamma_T(\mathcal O_{L,S}, \Z_p(n)) := \cone (R\Gamma(\mathcal O_{L,S}, \Z_p(n)) 
	\rightarrow \bigoplus_{w \in T(K)} R\Gamma(L(w), \Z_p(n)))[-1]
\]
in $\mathcal{D}(\Z_p[G])$. Note that the associated long exact sequence in 
cohomology gives rise to the $p$-completion of sequence \eqref{eqn:ray-class-sequence}
if $n=1$. We denote the cohomology group of this complex in degree $i$ 
by $H^i_T(\mathcal O_{L,S}, \Z_p(n))$. Let $n>1$. Then
these groups vanish unless $i=1$ or $i=2$
by \cite[Lemma 5.2]{MR3072281}; moreover, the second cohomology group
$H^2_T(\mathcal O_{L,S}, \Z_p(n))$ is finite.

The following result is the ($p$-adic) analogue of Theorem \ref{thm:Dasgupta-Kakde}
and has essentially been shown by Johnston and the second named author
in the course of the proof of Theorem \ref{thm:Coates-Sinnott}.

\begin{theorem} \label{thm:stong-CS-abelian}
	Let $L/K$ be a finite abelian extension of number fields with Galois group $G$.
	Let $r$ be a negative integer and let $p$ be an odd prime. Let $S$ and $T$
	be two finite non-empty sets of places of $K$ such that $S$ contains 
	$S_{\infty} \cup S_{\ram} \cup S_p$ and $S \cap T = \emptyset$. Then we have
	$\theta_{S}^T(r) \in \Ann_{\Z_p[G]}(H_T^2(\mathcal{O}_{L,S}, \Z_p(1-r)))$.
\end{theorem}
\begin{proof}
	We actually show the stronger statement
	$\theta_{S}^T(r) \in \Fitt_{\Z_p[G]}(H_T^2(\mathcal{O}_{L,S}, \Z_p(1-r)))$,
	where $\Fitt_R(M)$ denotes the initial Fitting ideal of a
	finitely presented $R$-module $M$.
	By \cite[Lemma 5.8]{MR3072281} we may assume that $L/K$ is a CM extension,
	in which case the result follows from 
	\cite[Theorem 12.5]{abelian-MC}. 
\end{proof}

The following conjecture generalises Theorem \ref{thm:stong-CS-abelian}
to the non-abelian case
and is the ($p$-adic) analogue of Conjecture \ref{conj:Brumer--Stark}.
It is apparently implied by \cite[Conjecture 5.1]{MR3072281}.

\begin{conj} \label{conj:strong-Coates--Sinnott}
	Let $L/K$ be a finite Galois extension of number fields with Galois group $G$.
	Let $r$ be a negative integer and let $p$ be an odd prime.  Let $S$ and $T$
	be two finite non-empty sets of places of $K$ such that $S$ contains 
	$S_{\infty} \cup S_{\ram} \cup S_p$ and $S \cap T = \emptyset$. Then we have
	\[
		\mathcal{H}_p(G) \cdot \theta_{S}^T(r) \subseteq \Ann_{\zeta(\Z_p[G])}(H_T^2(\mathcal{O}_{L,S}, \Z_p(1-r))).
	\]
\end{conj}

\subsection{Results}
In this subsection we prove the annihilation results from the introduction. 
More generally, we show the following.

\begin{theorem} \label{thm:main-result}
	Let $L/K$ be a Galois extension of number fields such that $G = \Gal(L/K)$
	is nilpotent. Let $S,T$ be a pair of finite sets of places of $K$ such that
	$\Hyp(S,T)$ is satisfied. Let $r \leq 0$ be an integer. Then we have
	\[
		|G'| \theta_S^T(r) \in \zeta(\Z[G])
	\]
	and moreover the following holds.
	\begin{enumerate}
		\item 
		If $r=0$, then $|G'| \theta_S^T(0)$ annihilates $\Z[\half] \otimes_{\Z} \cl_{L}^{T}$.
		\item
		If $r<0$, then $|G'| \theta_S^T(r)$ annihilates $\Z[\half] \otimes_{\Z} 
		K_{-2r}(\mathcal{O}_{L,S})$.
		\item
		Let $p$ be an odd prime and assume in addition that $S_p \subseteq S$.
		If $r<0$, then 
		$|G'| \theta_S^T(r)$ annihilates $H_T^2(\mathcal{O}_{L,S}, \Z_p(1-r))$.
	\end{enumerate}
\end{theorem}

\begin{proof}
	Let $N$ be a normal subgroup of $G$ 
	containing the commutator subgroup $G'$ of $G$.
	Let $\chi$ be an irreducible complex-valued character of $G$ and let $\eta$ 
	be an irreducible constituent of $\res^G_N \chi$ so that 
	$\res^G_N \chi = m \sum_{x \in G/G_{\eta}} {}^x \eta$ 
	for some positive integer $m$. 
	Let $\Q(\res^G_N \chi) := \Q(\chi(n) \mid n \in N)$
	which is a finite abelian extension of $\Q$. 
	We denote its Galois group by $\mathcal{A}_{\chi,N}$.
	Recall from 
	\eqref{eqn:definition-e-of-eta} that we have defined an idempotent
	\[
		e(\eta) = \sum_{x \in G/G_{\eta}} e_{{}^x \eta} = 
		\frac{\eta(1)}{m|N|} \sum_{g \in N} \chi(g^{-1}) g \in \Q(\res^G_N \chi)[N]
	\]
	which is actually central in $\Q(\res^G_N \chi)[G]$. Now let $\chi'$ be a second
	irreducible character of $G$. We say that $\chi'$ is equivalent to $\chi$
	if there is a $\sigma \in \mathcal{A}_{\chi,N}$ such that 
	$\res^G_N \chi' = \sigma(\res^G_N \chi)$. 
	Then $\mathcal{A}_{\chi,N} = \mathcal{A}_{\chi',N}$ and $G_{\eta} = G_{\eta'}$
	for every irreducible constituent $\eta'$ of $\res^G_N \chi'$.
	We set
	\[
		\epsilon_{\chi} := \sum_{\chi'} e_{\chi'} =
		\sum_{\sigma \in \mathcal{A}_{\chi,N}} \sigma(e(\eta))\in \Q[N]
	\]
	where the first sum runs over all $\chi'$ which are equivalent to $\chi$,
	and the equality follows from Corollary \ref{cor:idempotent-equalities} (i).
	
	We now show by induction on the group order that
	$|N| \theta_S^T(r) \epsilon_{\chi}$ is contained in $\zeta(\Z[G])$ and annihilates
	the Galois modules occurring in claims (i), (ii) and (iii). 
	Taking $N = G'$ then finishes the proof.
	If $G$ is abelian,
	then $|N| \epsilon_{\chi} \in \Z[G]$ and the claims follow from
	Theorems \ref{thm:Dasgupta-Kakde}, \ref{thm:Coates-Sinnott} and
	\ref{thm:stong-CS-abelian}, respectively.
	
	Now let $G$ be arbitrary. We first follow the strategy of the proof of
	Theorem \ref{thm:integrality-mainstep}.
	Let $F$ be a splitting field for $G$ and all
	of its subgroups. Recall the embedding $\iota_{\eta}$ from
	\eqref{eqn:iota-eta}. For each $\sigma \in \mathcal{A}_{\chi,N}$ we likewise
	have an embedding $\iota_{\eta}^{\sigma}: \zeta(F[G] \sigma(e(\eta)))
	\hookrightarrow \zeta(F[G_{\eta}] \sigma(e(\eta)))$. Their product then
	induces an embedding
	\[
		\iota_{\chi}: \zeta(F[G] \epsilon_{\chi})
		\hookrightarrow \zeta(F[G_{\eta}] \epsilon_{\chi})
	\]
	which is actually the canonical inclusion: The coefficient of an element
	on the left hand side at $g \in G$ equals $0$ whenever $g \not\in G_{\eta}$.
	Let $\psi'$ be an irreducible character of $G_{\eta}$ such that
	$e_{\psi'}$ divides $\epsilon_{\chi}$. Then Corollary 
	\ref{cor:idempotent-equalities} implies that $\chi' := \ind^G_{G_{\eta}} \psi'$
	is an irreducible character of $G$ which is equivalent to $\chi$
	(recall that in the notation of the Corollary we have that
	$\chi = \ind_{G_{\eta}}^G \psi = \ind_{G_{\eta}}^G {}^x\psi$
	for each $x \in G$). 
	We set $K_{\eta} := L^{G_{\eta}}$ and let $S_{\eta}$ and $T_{\eta}$
	be the sets of places of $K_{\eta}$ lying above those in 
	$S$ and $T$, respectively.
	Since $ L_S(r, \check{\chi}') = 
	L_{S_{\eta}}(r, \check{\psi}')$
	and likewise  $\delta_T(r, \chi') = \delta_{T_{\eta}}(r, \psi')$, we have
	\[
		\iota_{\chi}(\theta_S^T(L/K, r) \epsilon_{\chi}) = 
		\theta_{S_{\eta}}^{T_{\eta}}(L/K_{\eta}, r) \epsilon_{\chi}.
	\]
	As we may consider each $G$-module as $G_{\eta}$-module by restriction,
	we may and do assume that $G = G_{\eta}$. Hence $\chi = \psi$ is an
	irreducible character of $G = G_{\eta}$ and we consider the subgroup
	$U_{\psi}$, which has been defined before Lemma \ref{lem:psi-twists}.
	We may generalize the second step of the proof of 
	Theorem \ref{thm:integrality-mainstep} just as above and assume that
	$G = U_{\psi}$.
	
	If $\chi = \psi$ is linear, then so is $\eta = \res^G_N \psi$. We consider
	the factor groups $\overline{G} := G/G'$ and $\overline{N} := N/G'$.
	Then $\chi = \infl^G_{\overline{G}} \overline{\chi}$ and
	$\eta = \infl^N_{\overline{N}} \overline{\eta}$ for (linear) characters
	$\overline{\chi}$ and $\overline{\eta}$ of $\overline{G}$ and 
	$\overline{N}$, respectively. The natural projection
	$\Q[G] \rightarrow \Q[\overline{G}]$ is split by the map
	$\phi: 1 \mapsto |G'|^{-1} \Tr_{G'}$. In particular, the image of
	$\Z[\overline{G}]$ under $\phi$ is contained in $|G'|^{-1} \Z[G]$.
	Since $\overline{G}$ is abelian, we know that
	$|\overline{N}| \theta_{S}^T(L^{G'}/K,r)\epsilon_{\overline{\chi}}
	\in \Z[\overline{G}]$, where $L^{G'}$ is the maximal abelian
	subextension of $K$ in $L$. Thus we obtain that
	\[
		|N| \theta_{S}^T(L/K,r) \epsilon_{\chi} = 
		|G'|\phi(|\overline{N}|\theta_S^T(L^{G'}/K,r)\epsilon_{\overline{\chi}})
		\in \Z[G].
	\]
	Multiplication by $|N| \theta_{S}^T(L/K,0) \epsilon_{\chi} = 
	|\overline{N}| \theta_{S}^T(L/K,0) \epsilon_{\chi} \Tr_{G'}$ on 
	$\cl_L^T$ factors as
	\[
		\cl_L^T \xrightarrow{\Tr_{G'}} \cl_{L^{G'}}^T
		\xrightarrow{|\overline{N}|\theta_S^T(L^{G'}/K,0)\epsilon_{\overline{\chi}}}
		\cl_{L^{G'}}^T \longrightarrow \cl_L^T,
	\]
	where the first and last map are induced by the usual trace and inclusion
	maps on ideals. Since the middle arrow annihilates $\Z[\half] \otimes_{\Z} \cl_{L^{G'}}^T$ by Theorem \ref{thm:Dasgupta-Kakde}, we are done in case (i). 
	We now consider case (iii) and for this purpose we temporarily assume
	that $S$ contains $S_p$. By \cite[(22)]{MR3072281} there is a natural isomorphism
	of $\Z_p[\overline{G}]$-modules
	\[
		H_T^2(\mathcal{O}_{L,S}, \Z_p(1-r))_{G'} \simeq
		H_T^2(\mathcal{O}_{L^{G'},S}, \Z_p(1-r)).
	\]
	Since multiplication by $\Tr_{G'}$ factors through the quotient
	$H_T^2(\mathcal{O}_{L,S}, \Z_p(1-r))_{G'}$ and 
	$|\overline{N}|\theta_S^T(L^{G'}/K,r)\epsilon_{\overline{\chi}}$ annihilates
	$H_T^2(\mathcal{O}_{L^{G'},S}, \Z_p(1-r))$ by Theorem
	\ref{thm:stong-CS-abelian}, this shows the claim
	in case (iii).
	By \eqref{eqn:Chern-class-maps}
	and the fact that  $(1-(Nv)^r \sigma_{w}) \in \Z_p[G]^{\times}$
	for $v \in S_p$ and $r<0$ (see the proof of
	\cite[Lemma 6.13]{MR3383600}, for instance), this also implies (ii).
	
	We are left with the case where $\chi = \psi$ is non-linear.
	By \cite[Theorem 11.3]{MR632548} there is a necessarily proper subgroup
	$\tilde H$ of $G$ and a linear character $\tilde \lambda$ of $\tilde H$
	such that $\psi= \ind_{\tilde H}^G \tilde \lambda$.
	By \cite[5.2.4]{MR1357169} there is a proper \emph{normal}
	subgroup $H$ containing $\tilde H$. We set 
	$\lambda := \ind_{\tilde H}^H \tilde \lambda$ so that $\psi = \ind_H^G \lambda$
	and we are in the situation of \S \ref{subsec:more-Clifford}.
	
	Since $G = G_{\eta}$ we have $e(\eta) = e_{\eta}$ which is a 
	primitive central idempotent in $F[N]$ and which is still central in $F[G]$.
	As $G = U_{\psi}$ we have a decomposition into primitive central 
	idempotents $e_{\eta} = \sum_{\omega \in \Irr(G/N)} e_{\psi \otimes \omega}$ in $F[G]$ by Proposition 
	\ref{prop:psi-when-N-contains-commutator} and Corollary
	\ref{cor:idempotent-equalities} (ii). Moreover, the idempotent $e_{\eta}$
	is indeed in $F[H]$, where we have a decomposition into primitive central
	idempotents $e_{\eta} = \sum_{c \in G/H} \sum_{\omega \in \Irr(G/N)}
	e_{^c\lambda \otimes \omega}$ by Lemma \ref{lem:final-lemma-on-idempotents} (ii).
	We obtain the diagonal embedding
	\[
		\iota_{\eta}^H: \zeta(F[G]e_{\eta}) = 
		\bigoplus_{\omega \in \Irr(G/N)} F e_{\psi \otimes \omega} \hookrightarrow
		\bigoplus_{\omega \in \Irr(G/N)} \bigoplus_{c \in G/H} 
		F e_{^c\lambda \otimes \omega} = \zeta(F[H] e_{\eta}).
	\]
	As above we have a similar embedding for $\sigma(e_{\eta})$ for each 
	$\sigma \in \mathcal{A}_{\chi,N}$ so that the product over all $\sigma$
	induces an embedding
	\[
		\iota_{\chi}^H: \zeta(F[G] \epsilon_{\chi}) \hookrightarrow
		\zeta(F[H] \epsilon_{\chi}).
	\]
	By Lemma \ref{lem:psi-with-H-and-N} (i) we have that 
	$\ind_H^G ({}^c \lambda \otimes \omega)
	= \ind_H^G (\lambda \otimes \omega) = \psi \otimes \omega$
	for each $c \in G/H$, $\omega \in \Irr(G/N)$. Since the $L$-functions
	behave well under induction of characters, we see that
	\[
		\iota_{\chi}^H(\theta_{S}^{T}(L/K, r) \epsilon_{\chi}) = 
		\theta_{S(L^H)}^{T(L^H)}(L/L^H, r) \epsilon_{\chi}.
	\]
	Note that the inclusion $\iota_{\chi}^H$ actually leaves any element $z$ in
	$\zeta(F[G] \epsilon_{\chi})$ unchanged; the coefficients of $z$ at
	any $g \in G \setminus H$ must vanish. We may therefore view our $G$-modules
	as $H$-modules by restriction and replace the triple $(G, N, \chi)$
	by the triples $(H, N \cap H, {}^c \lambda)$, $c \in G/H$. 
	Since $H$ is a proper subgroup of $G$ and $\epsilon_{\chi}$
	is a sum of analogously defined idempotents $\epsilon_{{}^c \lambda}$,
	where $c$ runs over a possibly proper subset of $G/H$, this completes the proof.
\end{proof}

Recall that $E_d$ was defined as the sum over all primitive central idempotents
$e_{\chi}$ associated to irreducible characters of $G$ of degree $d$.
Then $E_d$ may be written as a sum of idempotents of the form $\epsilon_{\chi}$
and the proof of Theorem \ref{thm:main-result} indeed shows the following.

\begin{corollary}
	Let $L/K$ be a Galois extension of number fields such that $G = \Gal(L/K)$
	is nilpotent. Let $S,T$ be a pair of finite sets of places of $K$ such that
	$\Hyp(S,T)$ is satisfied. Let $r \leq 0$ and $d>0$ be integers. Then we have
	\[
	|G'| \theta_S^T(r) E_d \in \zeta(\Z[G])
	\]
	and moreover the following holds.
	\begin{enumerate}
		\item 
		If $r=0$, then $|G'| \theta_S^T(0) E_d$ annihilates $\Z[\half] \otimes_{\Z} \cl_{L}^{T}$.
		\item
		If $r<0$, then $|G'| \theta_S^T(r) E_d$ annihilates $\Z[\half] \otimes_{\Z} 
		K_{-2r}(\mathcal{O}_{L,S})$.
		\item
		Let $p$ be an odd prime and assume in addition that $S_p \subseteq S$.
		If $r<0$, then 
		$|G'| \theta_S^T(r) E_d$ annihilates $H_T^2(\mathcal{O}_{L,S}, \Z_p(1-r))$.
	\end{enumerate}
\end{corollary}

\begin{remark}
	Let $N$ be a normal subgroup of $G$ containing $G'$.
	Similar to Remark \ref{rem:more-elements-in-H}, one can show that
	indeed 
	\[
		\left(\frac{|N|}{\eta(1)} \varepsilon_{\chi}\right) \theta_{S}^T(r) 
		\in \zeta(\Z[G])
	\]
	and annihilates the corresponding Galois module. This follows along
	the lines of the proof of Theorem \ref{thm:main-result}. One only
	has to observe in the last step of the proof that 
	$|N| / \eta(1) = |N \cap H| / \eta_c(1)$ for all $c \in G/H$,
	as $\eta(1) = [G : H] \eta_c(1)$ and $H / (H\cap N) \simeq G/N$
	by Lemma \ref{lem:psi-with-H-and-N}.
\end{remark}

\section{The relation to the equivariant Tamagawa number conjecture}

\subsection{Algebraic $K$-theory}
We briefly recall some basic notions from algebraic $K$-theory.
For more details we refer the reader to
\cite{MR892316} and \cite{MR0245634}.

Let $R$ be any ring.
Recall that $\PMod(R)$ denotes the category of finitely generated 
projective left $R$-modules. We write $K_{0}(R)$ for 
the Grothendieck group of $\PMod(R)$
and $K_{1}(R)$ for the Whitehead group 
(see \cite[\S 38 and \S 40]{MR892316}).

Let $\mathcal{O}$ be a noetherian integral domain of characteristic $0$ with field of fractions $F$.
Let $A$ be a finite-dimensional semisimple $F$-algebra and let $\Lambda$ be an $\mathcal{O}$-order in $A$.
For any field extension $E$ of $F$ we set $A_{E} := E \otimes_{F} A$.
Let $K_{0}(\Lambda, E)$ denote the relative
algebraic $K$-group associated to the ring homomorphism $\Lambda \hookrightarrow A_{E}$.
We recall that $K_{0}(\Lambda, E)$ is an abelian group with generators $[X,g,Y]$ where
$X$ and $Y$ are finitely generated projective $\Lambda$-modules
and $g:E \otimes_{R} X \rightarrow E \otimes_{R} Y$ is an isomorphism of $A_{E}$-modules;
for a full description in terms of generators and relations, we refer the reader to \cite[p.\ 215]{MR0245634}.
Furthermore, there is a long exact sequence of relative $K$-theory
(see  \cite[Chapter 15]{MR0245634})
\begin{equation}\label{eqn:long-exact-seq}
	K_{1}(\Lambda) \longrightarrow K_{1}(A_{E}) \stackrel{\partial_{\Lambda,E}}{\longrightarrow}
	K_{0}(\Lambda,E) \longrightarrow K_{0}(\Lambda) \longrightarrow K_{0}(A_{E}).
\end{equation}

Let $p$ be a prime and let $G$ be a finite group.
Let $F$ be a finite extension of $\Q_{p}$
with ring of integers $\mathcal{O} = \mathcal{O}_F$ and let $e \in
\mathcal{O}[G]$ be a central idempotent. It follows
from a theorem of Swan \cite[Theorem 32.1]{MR632548} that the map
$K_0(e\mathcal{O}[G]) \rightarrow K_0(eF[G])$ is injective. Hence from
\eqref{eqn:long-exact-seq} we obtain a right-exact sequence 
\begin{equation} \label{eqn:longterm-for-grouprings}
	K_1(e\mathcal{O}[G]) \rightarrow
	K_1(eF[G]) \rightarrow K_0(e\mathcal{O}[G],F) \rightarrow 0.
\end{equation}
Moreover, by \cite[Theorems 45.3 and 40.31]{MR892316} the reduced
norm induces an isomorphism $K_1(eF[G]) \simeq \zeta(eF[G])^{\times}$
and $\nr(K_1(e\mathcal{O}[G])) = \nr(e\mathcal{O}[G]^{\times})$.
So from \eqref{eqn:longterm-for-grouprings} we obtain an isomorphism
\begin{equation} \label{eqn:relative-K0-iso}
	K_0(e\mathcal{O}[G], F) \simeq \zeta(eF[G])^{\times} / 
	\nr(e\mathcal{O}[G]^{\times}).
\end{equation}

Choose a maximal $\mathcal{O}$-order $\mathcal{M}(G)$ in $F[G]$ containing
$\mathcal{O}[G]$. Then a similar reasoning using 
\cite[Proposition 45.8]{MR892316} yields an isomorphism
\begin{equation} \label{eqn:relative-K0-iso-max}
	K_0(e\mathcal{M}(G), F) \simeq \zeta(eF[G])^{\times} / 
	\zeta(e\mathcal{M}(G))^{\times}.
\end{equation}

Now let $E$ be a finite Galois extension of $F$. We let
$\sigma \in \Gal(E/F)$ act on $x = \sum_{g \in G} x_g g \in E[G]$
by $\sigma(x) = \sum_{g \in G} \sigma(x_g) g$.

\begin{lemma} \label{lem:norm-surjective}
	Let $p$ be a prime and let $G$ by a finite $p$-group.
	Let $F$ be a finite extension of $\Q_p$ and let $E$ be a finite
	unramified extension of $F$ such that $p$ and $[E:F]$ are coprime.
	Then the norm map
	\[
		N_{E/F}: \mathcal{O}_E[G]^{\times} \rightarrow
			\mathcal{O}_F[G]^{\times}, \quad
			x \mapsto \prod_{\sigma \in \Gal(E/F)} \sigma(x)
	\]
	is surjective.
\end{lemma}

\begin{proof}
	We denote the radical of a ring $\Lambda$ by $\rad(\Lambda)$.
	Let $\mathbb{E}$ and $\F$ be the residue fields of
	$E$ and $F$, respectively. By \cite[Corollary 5.25]{MR632548}
	the radical of $\mathcal{O}_F[G]$ is generated by
	a uniformizer of $\mathcal{O}_F$ and the augmentation ideal,
	which is the kernel of the natural $\mathcal{O}_F$-linear map
	$\mathcal{O}_F[G] \rightarrow \mathcal{O}_F$ that sends each
	$g \in G$ to $1$. In particular, we have an isomorphism
	$\mathcal{O}_F[G] / \rad(\mathcal{O}_F[G]) \simeq \F$.
	Similar observations hold with $F$ replaced by $E$. 
	
	Since an element is invertible if and only if it is invertible modulo
	the radical, we obtain a commutative diagram
	\[ \xymatrix{
		1 \ar[r] & 1 + \rad(\mathcal{O}_E[G]) \ar[r] \ar[d]^{N_{E/K}} &
			\mathcal{O}_E[G]^{\times} \ar[r] \ar[d]^{N_{E/K}} &
			\mathbb{E}^{\times} \ar[r] \ar[d]^{N_{E/K}} & 1\\
		1 \ar[r] & 1 + \rad(\mathcal{O}_F[G]) \ar[r] &
		\mathcal{O}_F[G]^{\times} \ar[r] &
		\mathbb{F}^{\times} \ar[r] & 1.
	}\]
	As $E/F$ is unramified, the rightmost vertical map is surjective.
	Since $N_{E/K}(1+x) = (1+x)^{[E:F]}$ for $x \in \rad(\mathcal{O}_F[G])$
	and $1 + \rad(\mathcal{O}_F[G])$ is a pro-$p$-group (see
	\cite[Example 45.29]{MR892316}), the leftmost vertical map
	is also surjective. The result follows.
\end{proof}

\begin{corollary} \label{cor:nr-image}
	Let $p$ be a prime and let $G$ be a finite abelian $p$-group.
	Let $F$ be a finite extension of $\Q_p$ and let $D$ be a skewfield
	with centre $F$ such that $[D:F]$ is finite and coprime to $p$.
	Let $\mathcal{O}_D$ be the unique maximal $\mathcal{O}_F$-order in $D$.
	Then we have an equality
	\[
		\nr(K_1(\mathcal{O}_D[G])) = \mathcal{O}_F[G]^{\times}.
	\]
\end{corollary}

\begin{proof}
	Let $s$ be the Schur index of $D$ so that $s^2 = [D:F]$.
	Then $s$ is prime to $p$ by assumption. By \cite[\S 14]{MR1972204}
	there is an unramified extension $E$ of $F$ of degree $s$, which is
	contained in $D$. Then $E$ is a maximal subfield of $D$ and splits
	$D$. Moreover, the ring of integers $\mathcal{O}_E$ is contained
	in $\mathcal{O}_D$ and the matrix representation that occurs in the
	proof of \cite[Theorem 14.6]{MR1972204} (see \cite[(14.7)]{MR1972204})
	shows that $\nr(x) = N_{E/F}(x)$ for $x \in E$. More precisely,
	the matrix corresponding to $x$ is diagonal with the Galois
	conjugates of $x$ as entries. We have isomorphisms of rings
	\begin{equation} \label{eqn:splitting-of-DG}
		E \otimes_F D[G] \simeq (E \otimes_F D)[G] \simeq
		M_{s}(E)[G] \simeq M_{s}(E[G]),
	\end{equation}
	where the second isomorphism is induced by just that matrix
	representation. 
	Let $\sigma \in \Gal(E/F)$ be the unique lift of the Frobenius 
	automorphism. Choose $1 \leq r \leq s$ such that $r/s$ is the
	Hasse invariant of $D$. Then $r$ is prime to $s$ and 
	$\theta := \sigma^r$ is also a generator of $\Gal(E/F)$.
	Then $\theta$ acts upon
	$a \otimes x$, $a \in E$, $x \in D[G]$ by $\theta(a \otimes x) =
	\theta(a) \otimes x$. If $A = (\alpha_{i,j})_{1 \leq i,j \leq s}$
	is a matrix in $M_{s}(E[G])$, then we let $\theta$ act on $A$ by
	$\theta(A) = (\beta_{i,j})_{i,j}$ with 
	$\beta_{i,j} = \theta(\alpha_{i-1,j-1})$, where the indices are
	understood to be considered modulo $s$.
	By inspection of \cite[(14.7)]{MR1972204}
	the composite map \eqref{eqn:splitting-of-DG}
	is then $\Gal(E/F)$-equivariant.
	Let $x \in \mathcal{O}_D[G]^{\times}$. 
	Then the image $A_x$ of $1 \otimes x$
	under \eqref{eqn:splitting-of-DG} is invariant under the
	action of $\Gal(E/F)$ and a matrix with entries in
	$\mathcal{O}_E[G]$. Hence 
	\[
		\nr(x) = \det(A_x) \in (\mathcal{O}_E[G]^{\times})^{\Gal(E/F)} = 
		\mathcal{O}_F[G]^{\times}
	\]
	so that $\nr(K_1(\mathcal{O}_D[G])) = \nr(\mathcal{O}_D[G]^{\times})$ 
	is contained in 
	$\mathcal{O}_F[G]^{\times}$. For $x \in \mathcal{O}_E[G]$
	its image $A_x$ is a diagonal matrix whose
	entries are again the Galois conjugates of $x$.
	Hence $\nr(x) = N_{E/K}(x)$ and the result follows from
	Lemma \ref{lem:norm-surjective}.
\end{proof}

\subsection{The relation to the equivariant Tamagawa number conjecture}
Let $L/K$ be a Galois extension of number fields with Galois group 
$G = \Gal(L/K)$ and let $r$ be an integer.
The equivariant Tamagawa number conjecture (ETNC for short)
for the pair $(h^0(\Spec(L))(r), \Z[G])$ has been formulated by
Burns and Flach \cite{MR1884523} and
asserts that a certain canonical element $T\Omega(L/K, \Z[G], r)$
in the relative algebraic $K$-group $K_0(\Z[G], \R)$ vanishes.
We refrain from giving a more precise statement, as we will only exploit known
cases and functorial properties of the conjecture.

We henceforth assume that $L/K$ is a CM extension.
So $L$ is a totally complex extension of the totally real field $K$ such that
complex conjugation induces a unique automorphism $\tau$ of $L$ 
which is central in $G$. Note that this is the essential case
for the integrality and annihilation conjectures considered in this
article (see \cite[\S 4]{MR3552493}, for instance). 

We define a central idempotent 
$e_r := \frac{1- (-1)^{r} \tau}{2}$ in $\Z[\half][G]$, where
$\tau \in G$ denotes complex conjugation. 
The ETNC for the pair $(h^0(\Spec(L))(r), e_{r}\Z[\half][G])$ then likewise
asserts that a canonical element 
$T\Omega(L/K,r)$ in $K_{0}(e_{r} \Z[\half][G], \R)$ vanishes. 
(It would be more accurate to denote $T\Omega(L/K, r)$ by
$T\Omega(L/K, e_{r} \Z[\half][G], r)$, but in order to avoid too heavy notation
we will drop the reference to the order $e_{r} \Z[\half][G]$.)
This corresponds to the plus
or minus part of the ETNC (away from $2$) if $r$ is odd or even, respectively.

Now assume that $r \leq 0$. Then a result of Siegel
\cite{MR0285488} implies that $T\Omega(L/K, r)$ actually
belongs to the subgroup
\begin{equation} \label{eqn:K0-iso}
	K_{0}(e_{r} \Z[\half][G], \Q) \simeq \bigoplus_{p\, \mathrm{ odd}} K_{0}(e_{r} \Z_{p}[G], \Q_{p})
\end{equation}
and we say that the $p$-part of the ETNC for the pair 
$(h^0(\Spec(L))(r), e_r\Z[\half][G])$ holds if its image $T\Omega_p(L/K,r)$ in
$K_{0}(e_{r} \Z_{p}[G], \Q_{p})$ vanishes.
More generally, if $\mathcal{M}$ is an order containing $e_r\Z[\half][G]$,
we say that the $p$-part of the ETNC for the pair 
$(h^0(\Spec(L))(r), \mathcal{M})$ holds if $T\Omega_p(L/K,r)$
is contained in the kernel of the natural map
$K_{0}(e_{r} \Z_{p}[G], \Q_{p}) \rightarrow K_0(\Z_p \otimes_{\Z} \mathcal{M}, \Q_p)$.

\subsection{Functorialities} \label{subsec:functorialities}
By \cite[Theorem 4.1]{MR1884523} the ETNC behaves well under base change.
In particular, the following holds in our situation.

Let $H$ be a subgroup of $G$ which contains $\tau$. Then there is a
natural restriction map
\[
	\res^G_H: K_{0}(e_{r} \Z[\half][G], \R) \rightarrow
	K_{0}(e_{r} \Z[\half][H], \R).
\]
such that for every $r \in \Z$ one has
\[
	\res^G_H(T\Omega(L/K,r)) = T\Omega(L/L^H,r).
\]
Likewise, if $N$ is a normal subgroup of $G$ 
which does not contain $\tau$, then
there is a natural quotient map
\[
\quot^G_{G/N}: K_{0}(e_{r} \Z[\half][G], \R) \rightarrow
K_{0}(e_{r} \Z[\half][G/N], \R)
\]
such that
\[
\quot^G_{G/N}(T\Omega(L/K,r)) = T\Omega(L^N/K,r).
\]
If $r \leq 0$ and $p$ is an odd prime, then similar observations hold
with $T\Omega(L/K,r)$ replaced by $T\Omega_p(L/K,r)$.

\subsection{Results on the ETNC}
The following result and variants thereof will be important for our purposes.
Note that parts (i) and (ii) are essentially known, but part (iii) is new
and generalizes \cite[Theorem 1.2]{MR3552493}.

\begin{theorem} \label{thm:ETNC-cases}
	Let $L/K$ be a Galois CM extension of number fields with Galois group $G$.
	Let $p$ be an odd prime and let $r \leq 0$ be an integer. 
	Let $S$ and $T$ be two finite sets of places of $K$ such that
	$\Hyp(S,T)$ is satisfied.
	Then the following holds.
	\begin{enumerate}
		\item 
		If the $p$-part of the ETNC for the pair 
		$(h^0(\Spec(L))(r), e_r\Z[\half][G])$ holds, then 
		$\theta_{S}^{T}(L/K,r) \in \mathcal{I}_p(G)$.
		In particular, we then have that
		$|G'| \theta_{S}^{T}(L/K,r) \in \zeta(\Z_p[G])$.
		\item
		Let $\mathcal{M}(G)$ be a maximal order containing $\Z[\half][G]$.
		Then the $p$-part of the ETNC for the pair 
		$(h^0(\Spec(L))(r), e_r\mathcal{M}(G))$ holds. In particular, we have 
		$\theta_{S}^{T}(L/K,r) \in \zeta(\Z_p \otimes_{\Z} \mathcal{M}(G))$
		and hence
		$|G| \theta_{S}^{T}(L/K,r) \in \zeta(\Z_p[G])$.
		\item 
		Assume that $G$ decomposes into a direct product $G = N \times \Delta$
		of a finite group $N$ and a finite abelian $p$-group $\Delta$.
		Let $\mathcal{M}(N)$ be a maximal order containing $\Z[\half][N]$.
		Then the $p$-part of the ETNC for the pair 
		$(h^0(\Spec(L))(r), e_r\mathcal{M}(N)[\Delta])$ holds. In particular, we have 
		$\theta_{S}^{T}(L/K,r) \in \zeta(\Z_p \otimes_{\Z} \mathcal{M}(N))[\Delta]$ 
		and hence
		$|N| \theta_{S}^{T}(L/K,r) \in \zeta(\Z_p[G])$.
	\end{enumerate}
\end{theorem}

\begin{proof}
	If the $p$-power roots of unity $\mu_L(p)$ in $L$ form 
	a cohomologically trivial $G$-module, 
	then the first claim  in (i) is	due to the second named author
	\cite[Theorem 5.1]{MR2976321}; this includes the cases $\mu_L(p) = 1$
	and $p \nmid |G|$. The general case follows from more recent
	work of Burns \cite[Proof of Corollary 3.11 (iii)]{MR4195656}.
	The final claim is a consequence of Theorem \ref{thm:denominator-ideal-capZ}.
	
	For (ii) we first observe that it suffices to prove 
	the $p$-part of the ETNC for the pair 
	$(h^0(\Spec(L))(r), e_r\mathcal{M}(G))$.
	This follows from (the proofs of) \cite[Theorem 4.1]{MR2976321} if $r=0$ and
	\cite[Theorem 5.7]{MR2801311} if $r<0$. The relevant special cases of the
	ETNC indeed hold as follows from work of the second named author
	\cite[Theorem 1]{StrongStark} if $r=0$
	and of Burns \cite[Corollary 2.10]{MR3294653} if $r<0$ (see also 
	\cite[Theorem 12.3 (i)]{abelian-MC} for details).
	
	For part (iii) it again suffices to show the claim on the ETNC.
	The final assertion can then be deduced as in (i) and (ii).
	If $G$ is abelian, the ETNC for the pair
	$(h^0(\Spec(L))(r), e_{r}\Z[\half][G])$ holds unconditionally by recent work of
	Bullach, Burns, Daoud and Seo \cite{abelian-ETNC} if $r=0$ and of
	Johnston and the second named author \cite[Theorem 1.3]{abelian-MC} if $r<0$.
	From this we now deduce the $p$-part of the ETNC over the order 
	$e_r \mathcal{M}(N)[\Delta]$ by using Brauer induction as follows.
	Since $\Delta$ is a $p$-group for an odd prime $p$, 
	we have $\tau \in N$.
	We then have a decomposition
	\[
		\Z_p \otimes_{\Z} e_r \mathcal{M}(N) \simeq \prod_{\chi} \mathcal{M}_p(\chi)
	\]
	where the product runs over all irreducible odd (resp.\ even)
	characters of $N$ up to Galois action if $r$ is even (resp.\ odd).
	Each $\mathcal{M}_p(\chi)$ is a maximal $\Z_p$-order with centre 
	$\Z_p(\chi)$, the ring of integers in $\Q_p(\chi) := \Q_p(\chi(n) \mid n \in N)$,
	in a central simple $\Q_p(\chi)$-algebra. 
	The latter is a matrix ring over a skewfield whose Schur index
	is divisible by $(p-1)$ by \cite[Theorem 1]{MR279210}
	and thus in particular is prime to $p$. By Corollary \ref{cor:nr-image}
	the above decomposition and
	the reduced norm induce isomorphisms
	\[
		K_0(\Z_p \otimes_{\Z} e_r \mathcal{M}(N)[\Delta], \Q_p) \simeq
		\prod_{\chi} K_0(\mathcal{M}_p(\chi)[\Delta], \Q_p) \simeq
		\prod_{\chi} \Q_p(\chi)[\Delta]^{\times} / \Z_p(\chi)[\Delta]^{\times}.
	\]
	Now let $U$ be a subgroup of $N$ containing $\tau$. 
	Then there is a natural restriction map
	\[
		K_0(\Z_p \otimes_{\Z} e_r \mathcal{M}(N)[\Delta], \Q_p)  \rightarrow
		K_0(\Z_p \otimes_{\Z} e_r \mathcal{M}(U)[\Delta], \Q_p) 
	\]
	that maps a tuple $(f_{\chi})_{\chi}$
	to the tuple $(g_{\psi})_{\psi} \in \prod_{\psi}
	\Q_p(\psi)[\Delta]^{\times} / \Z_p(\psi)[\Delta]^{\times}$, where the product
	is over all irreducible characters $\psi$ of $U$ of the same parity as $\chi$
	modulo Galois action, and
	\[
		g_{\psi} = \prod_{\chi} f_{\chi}^{\langle \chi, \ind_U^N \psi \rangle_N}
	\]
	by a (variant of) Lemma \ref{lem:nr-under-restriction}.
	Let us fix one of the characters $\chi$. By Brauer's induction theorem 
	\cite[Theorem 15.9]{MR632548} there are integers $z_i$ and linear characters
	$\lambda_i$ of subgroups $U_i$ of $N$ such that
	\begin{equation} \label{eqn:Brauer-for-chi}
		\chi = \sum_i z_i \ind_{U_i}^N \lambda_i.
	\end{equation}
	We may assume that $\tau \in U_i$ for all $i$. To see this assume that
	$\tau \not\in U_i$ and set $\tilde U_i := U_i \times \langle \tau \rangle$. 
	We write $\ind_{U_i}^N \lambda_i = \ind_{\tilde U_i}^N (\ind_{U_i}^{\tilde U_i}
	\lambda_i)$ and observe that $\ind_{U_i}^{\tilde U_i} \lambda_i$ is the sum
	of two linear characters. Since the parity of a character is preserved under induction, we may assume in addition 
	that all $\lambda_i$ are of the same parity as $\chi$.
	We compute
	\[
		f_{\chi} = \prod_{\chi'} f_{\chi'}^{\langle \chi',\chi \rangle_N}
		= \prod_i \prod_{\chi'} f_{\chi'}^{z_i \langle\chi', \ind_{U_i}^N \lambda_i \rangle_N}
		= \prod_i g_{\lambda_i}^{z_i}.
	\]
	If all $g_{\lambda_i}$ are in $\Z_p(\lambda_i)[\Delta]^{\times}$,
	then $f_{\chi}$ must belong to $\Z_p(\chi)[\Delta]^{\times}$.
	Since the relevant special case of the ETNC holds for abelian CM extensions,
	the functorial behaviour of the ETNC, as discussed in
	\S \ref{subsec:functorialities}, implies the result.
\end{proof}

This theorem turns out to be sufficient in many, but not all cases of interest to us.
In order to handle the general case, we need a further variant in the case
where $G$ decomposes into a semi-direct product $G = N \rtimes \Delta$,
where $\Delta$ is a cyclic $p$-group. We start with some technical results
which are inspired by work of Ritter and Weiss \cite{MR2114937}.

\begin{lemma} \label{lem:center-of-semidirect}
	Let $G$ be a finite group. Let $N$ be a normal subgroup of $G$ such that
	$\Delta := G/N$ is a cyclic group. Choose a generator $\overline{\delta}$
	of $\Delta$.
	Let $F$ be a field of characteristic $0$
	such that $F$ is a splitting
	field for $G$ and all of its subgroups.
	Let $V_{\chi}$ be an irreducible $F[G]$-module with character $\chi$
	and let $\eta$ be an irreducible constituent of $\res^G_N \chi$.
	Then the following holds.
	\begin{enumerate}
		\item 
		The idempotent $e(\eta) := \sum_{x \in G / G_{\eta}} e_{^x \eta} \in F[N]$
		is central in $F[G]$ and there exists a unique 
		$\delta_{\chi} \in\zeta(F[G]e(\eta))$
		with all of the following properties.
		\begin{enumerate}
			\item 
			We have $\delta_{\chi} = \delta^{w_{\chi}} \cdot c 
			= c \cdot \delta^{w_{\chi}}$, 
			where $\delta \in G$ maps to $\overline{\delta}$
			under the canonical map $G \rightarrow G/N = \Delta$, 
			$c \in (F[N]e(\eta))^{\times}$ and $w_{\chi} = [G:G_{\eta}]$.
			\item
			The action of $\delta_{\chi}$ on $V_{\chi}$ is trivial.
			\item
			One has $\zeta(F[G]e(\eta)) = F[\Delta_{\chi}]$, where $\Delta_{\chi}$
			is a cyclic group of order $[G_{\eta}:N]$, generated by $\delta_{\chi}$.
		\end{enumerate}
		\item
		Let $\chi'$ be a second irreducible character of $G$ and $\eta'$
		be an irreducible constituent of $\res^G_N \chi'$. Then
		$e(\eta) = e(\eta')$ if and only if $\chi' = \chi \otimes \rho$
		for a linear character $\rho$ of $G$ such that $N \subset \ker(\rho)$.
		If this is the case, then 
		$\delta_{\chi \otimes \rho} = \delta_{\chi} \cdot \rho(\delta)^{-w_{\chi}}$.
		\item
		If $F$ is a finite extension of either $\Q$ or $\Q_p$ for some
		prime $p$ and $H$ is a matrix with entries in $\mathcal{O}_F[G]e(\eta)$,
		then $\nr(H) \in \mathcal{O}_F[\Delta_{\chi}]$.
	\end{enumerate}

\end{lemma}

\begin{proof}
	It is clear that $e(\eta)$ is central in $G$ (see also Corollary 
	\ref{cor:idempotent-equalities} (i)). The existence of $\delta_{\chi}$
	with properties (a) and (b) can be shown by the argument
	of Ritter and Weiss \cite[Proposition 5 (2)]{MR2114937}.
	Note that this actually shows that (a) and (b) determine $\delta_{\chi}$
	uniquely.
	We warn the reader that our $e_{\eta}$ and $e(\eta)$ are denoted
	by $e(\eta)$ and $e_{\chi}$ in \cite{MR2114937}, respectively.
	The main point is that in the notation of \S \ref{sec:Clifford-theory}
	we have $m=1$ and $U_{\psi} = G_{\eta}$ by \cite[Satz 17.12]{MR0224703} because
	$\Delta$ is a cyclic group. For (c) one first has to check
	that $\delta_{\chi}^{[G_{\eta}:N]} = e(\eta)$ once
	(a) and (b) are established. Here the argument slightly differs from that
	given in \cite{MR2114937} and runs as follows. 
	
	Since $c$ and $\delta^{w_{\chi}}$ commute and
	$w_{\chi} \cdot	[G_{\eta}:N] = [G:N] = |\Delta|$, we certainly
	have that $\delta_{\chi}^{[G_{\eta}:N]} = n \cdot c^{[G_{\eta}:N]}$ for 
	some $n \in N$. In particular, $\delta_{\chi}^{[G_{\eta}:N]}$ is in
	$(F[N]e(\eta))^{\times}$ and acts trivially upon $V_{\chi}$.
	However, the module $F[N]e(\eta)$ is a submodule of some multiple
	of $\res^G_N V_{\chi}$ so that the only such element is $e(\eta)$.
	
	The index $[G_{\eta}:N]$ is actually the minimal positive integer $z$ such that
	$\delta_{\chi}^z = e(\eta)$ as the latter equality in particular implies
	that $\delta^{w_{\chi} z} \in N$. Hence $\delta_{\chi}$ indeed generates a
	cyclic group of order $[G_{\eta}:N]$. Now the inclusion
	$F[\Delta_{\chi}] \subseteq \zeta(F[G]e(\eta))$ must be an equality as we
	already know that the right hand side is an $F$-vector space of dimension
	$[G_{\eta}:N]$ by Proposition \ref{prop:psi-when-N-contains-commutator} (iii).
	This finishes the proof of (i).
	
	Part (ii) can be shown along the lines of 
	\cite[Corollary on p.\ 556 and Proposition 5 (3)]{MR2114937}.
	
	To show the final claim, we give a second argument 
	which reproves a large part
	of (c) and clarifies the relation to our results of 
	\S \ref{sec:Clifford-theory}. 
	We have seen in the proof of Theorem \ref{thm:integrality-mainstep}
	that the embedding \eqref{eqn:iota-eta} allows us to assume 
	$G = G_{\eta}$ by Lemma \ref{lem:nr-under-restriction}.
	Since $G_{\eta} = U_{\psi}$, we then have developed the (abstract)
	isomorphism \eqref{eqn:iota_pi}. Putting both maps together this gives an
	abstract isomorphism $\zeta(F[G]e(\eta)) \simeq F[G_{\eta}/N]$.
	Note that the quotient $G_{\eta}/N$ identifies with the subgroup of 
	$\Delta$ generated by $\overline\delta^{w_{\chi}}$,
	and an isomorphism is
	explicitly given by $\delta_{\chi} \mapsto \overline\delta^{w_{\chi}}$.
	Finally, equation \eqref{eqn:integrality-of-nr} shows that
	$\nr(H) \in \mathcal{O}_F[G_{\eta}/N]$ under this identification.
\end{proof}

Let $p$ be a prime and let $G$ be a finite group with a normal subgroup $N$
such that $\Delta := G/N$ is cyclic.
If $F/\Q_p$ is a finite extension that is a splitting field for $G$
and all of its subgroups, then we have canonical maps
\begin{equation} \label{eqn:scaler-extension}
	K_0(\Z_p[G], \Q_p) \rightarrow K_0(\mathcal{O}_F[G], F) \simeq
	\zeta(F[G])^{\times} / \nr(\mathcal{O}_F[G]^{\times}) \rightarrow
	\prod_{\chi / \sim}F[\Delta_{\chi}]^{\times} / \mathcal{O}_F[\Delta_{\chi}]^{\times}.
\end{equation}
Here, the first map and the isomorphism are induced from extension of scalars
and the reduced norm, respectively. The product on the right is over all
irreducible characters $\chi$ of $G$ modulo the relation $\chi \sim \chi'$
if and only if $\chi' = \chi \otimes \rho$ for a linear character $\rho$
of $G$ whose kernel contains $N$. Lemma
\ref{lem:center-of-semidirect} (i)(c) implies that
\[
	\zeta(F[G]) = \prod_{\chi/ \sim} F[\Delta_{\chi}]
\]
and the final claim of the same lemma implies that the rightmost arrow
is a well-defined epimorphism.
Note that if $(f_{\chi})_{\chi} \in \zeta(F[G])$, then replacing a
character $\chi$ by an equivalent character $\chi \otimes \rho$,
changes $f_{\chi}$ to $f_{\chi \otimes \rho} = \rho^{\sharp}(f_{\chi})$,
where $\rho^{\sharp}$ is the $F$-linear map on $F[\Delta_{\chi}]$ sending
$\delta_{\chi}^{i}$ to $(\rho(\delta)^{-w_{\chi}} \delta_{\chi})^{i} = 
\delta_{\chi \otimes \rho}^{i}$, $i \in \Z$.

We now prove the following variant of Theorem \ref{thm:ETNC-cases} (iii).

\begin{prop} \label{prop:bounding-denominators}
	Let $L/K$ be a Galois CM extension of number fields with Galois group 
	$G$. Let $N$ be a normal subgroup of $G$ such that $\Delta := G/N$ 
	is a cyclic group.
	Let $p$ be an odd prime and let $r \leq 0$ be an integer. 
	Let $S$ and $T$ be two finite sets of places of $K$ such that
	$\Hyp(S,T)$ is satisfied. 
	Let $F/\Q_p$ be a finite extension that is a splitting field for $G$
	and all of its subgroups. Then the following holds.
	\begin{enumerate}
		\item 
		The image of $T\Omega(L/K, r)$ in
		$\prod_{\chi / \sim}F[\Delta_{\chi}]^{\times} / \mathcal{O}_F[\Delta_{\chi}]^{\times}$ under \eqref{eqn:K0-iso}
		and \eqref{eqn:scaler-extension} vanishes.
		\item 
		We have that $\theta_{S}^T(L/K,r) \in 
		(\prod_{\chi / \sim} \mathcal{O}_F[\Delta_{\chi}]) \cap \zeta(\Q_p[G])$
		so that in particular $|N|\theta_{S}^{T}(L/K,r) \in \zeta(\Z_p[G])$.
	\end{enumerate}
\end{prop}

\begin{proof}
	Let $U$ be a subgroup of $G$ containing $\tau$. 
	We set $V := U \cap N$ which is a normal subgroup of $U$.
	Then $\Sigma := U/V$ naturally identifies with a subgroup of $\Delta$
	and is therefore cyclic.
	Consider the canonical restriction map 
	$K_0(\mathcal{O}_F[G],F) \rightarrow K_0(\mathcal{O}_F[U],F)$.
	We view $K_0(\mathcal{O}_F[G],F)$ as a quotient of $\zeta(F[G])^{\times}
	\simeq \prod_{\chi/ \sim} F[\Delta_{\chi}]^{\times}$ and likewise
	$K_0(\mathcal{O}_F[U],F)$ as a quotient of $\zeta(F[U])^{\times}
	\simeq \prod_{\psi/ \sim} F[\Sigma_{\psi}]^{\times}$.
	We now describe the image $(g_{\psi})_{\psi}$ 
	of a tuple $(f_{\chi})_{\chi}$ under restriction.
	For this let $j_{\chi}: F[\Delta_{\chi}] \rightarrow
	F[\Delta]$ be the injective map which sends $\delta_{\chi}$ to
	$\overline \delta^{w_{\chi}}$. Define $j_{\psi}: F[\Sigma_{\psi}]
	\rightarrow F[\Sigma] \subseteq F[\Delta]$ similarly. 
	Then one has the following equality in $F[\Delta]$ which determines $g_{\psi}$
	uniquely:
	\[
		j_{\psi}(g_{\psi}) = \prod_{\chi} j_{\chi}(f_{\chi})^{\langle\chi, \ind^G_U \psi \rangle_G}
	\]
	This can be shown as \cite[Lemma 9]{MR2114937}.
	Now fix $\chi \in \Irr(G)$. By Brauer's induction theorem 
	\cite[Theorem 15.9]{MR632548} there are integers $z_i$ and linear characters
	$\lambda_i$ of subgroups $U_i$ of $G$ such that
	\[
		\chi = \sum_i z_i \ind_{U_i}^G \lambda_i.
	\]
	If $\chi' = \chi \otimes \rho$ is a character equivalent to $\chi$, then likewise
	\[
		\chi \otimes \rho = \sum_i z_i (\ind_{U_i}^G \lambda_i) \otimes \rho
		= \sum_i z_i \ind_{U_i}^G (\lambda_i \otimes \res^G_{U_i} \rho)
	\]
	and each $\lambda_i \otimes \res^G_{U_i} \rho$ is a character equivalent to
	$\lambda_i$. For (i) we now may conclude as in the proof of Theorem
	\ref{thm:ETNC-cases}.
	
	By Theorem \ref{thm:ETNC-cases} (i) the full $p$-part of the ETNC implies
	that the Stickelberger elements belong to
	\[
		\mathcal{I}_p(G) \subseteq \prod_{\chi / \sim} \mathcal{O}_F[\Delta_{\chi}],
	\]
	where the inclusion follows from Lemma \ref{lem:center-of-semidirect} (iii).
	Since we know the ETNC up to a factor in the right hand side,
	we get (ii).
\end{proof}

\section{Iwasawa theory}

\subsection{Stickelberger elements in cyclotomic $\Z_{p}$-extensions}
Fix an odd prime $p$. We write $F_{\infty}$
for the cyclotomic $\Z_p$-extension of a number field $F$ and $F_n$
for its $n$-th layer which is the unique intermediate field of degree
$[F_n:F] = p^n$. Let $\Gamma_F := \Gal(F_{\infty}/F)$ which is 
(non-canonically) isomorphic to $\Z_p$. We choose a topological
generator $\gamma_F$ of $\Gamma_F$.
For a profinite group $\mathcal{G}$ we let
$\Lambda(\mathcal{G}) :=
\Z_p\llbracket \mathcal{G} \rrbracket = \varprojlim \Z_p[\mathcal{G}/\mathcal{N}]$,
where the inverse limit runs over all open normal subgroups $\mathcal{N}$ of
$\mathcal{G}$.

Let $L/K$ be a Galois CM extension of number fields with Galois group $G$.
 We set $G_n := \Gal(L_n/K)$ and 
$\mathcal{G} := \Gal(L_{\infty}/K) = \varprojlim_n G_n$.
Then $\mathcal{G}$ may be written as a semi-direct product
$\mathcal{G} = H \rtimes \Gamma$, where $H$ identifies with a subgroup
of $G$ and $\Gamma \simeq \Z_p$. The quotient $\mathcal{G}/H$ naturally
identifies with $\Gamma_K$.
If $n$ is large enough, one likewise has $G_n = H \rtimes \Gamma_n$,
where $\Gamma_n \simeq G_n/H = \Gal(L_n \cap K_{\infty}/K)$ is a finite
cyclic $p$-group. 

Now let $S$ and $T$ be two finite sets of places of $K$ such that
$\Hyp(S,T)$ is satisfied for all extensions $L_n/K$, $n \in \N$.
In particular, all $p$-adic places are contained in $S$.
Let $r \leq 0$ be an integer. Then the Stickelberger element
$\theta_{S}^{T}(L_{n+1}/K,r)$ maps to $\theta_{S}^{T}(L_n/K,r)$
under the canonical projection $\Q[G_{n+1}] \rightarrow \Q[G_n]$.
Since $\mathcal{G}/H$ is abelian, the commutator subgroup
$\mathcal{G}'$ of $\mathcal{G}$ is indeed a subgroup of $H$ and the canonical
surjection $\mathcal{G} \rightarrow G_n$ induces an epimorphism
$\mathcal{G}' \rightarrow G_n'$ which is an isomorphism for sufficiently large $n$.
If the integrality conjectures hold along the cyclotomic tower, then the
denominators of $\theta_{S}^{T}(L_n/K,r)$ are bounded by $|\mathcal{G}'|$
as a consequence of Theorem \ref{thm:denominator-ideal-capZ}.
We do not know this in general, 
but as a consequence of Proposition 
\ref{prop:bounding-denominators} (ii) the following holds.
	
\begin{theorem} \label{thm:bounded-denominators}
	Let $p$ be an odd prime and let $r \leq 0$ be an integer.
	Let $L/K$ be a Galois CM extension of number fields
	and suppose that $\Hyp(S,T)$ is satisfied for every
	extension $L_n/K$.
	Then the denominators of the Stickelberger elements 
	$\theta_{S}^T(L_n/K,r) \in \Q_p[G_n]$
	are bounded by a constant that does not depend on either
	$n$ or $r$.
\end{theorem}

This suffices to define
\begin{equation} \label{eqn:infinite-Stickelberger}
	\theta_{S}^{T}(L_{\infty}/K,r) := \varprojlim_n \theta_{S}^{T}(L_n/K,r)
	\in \Q_p \otimes_{\Z_p} \Lambda(\mathcal{G}).
\end{equation}
Note that we can use Theorem \ref{thm:ETNC-cases} (iii) if 
$\mathcal{G} = H \times \Gamma$ is a direct product.

\subsection{The Iwasawa algebra as an order}
Recall the $\mathcal{G} = H \rtimes \Gamma$ is a one-dimensional $p$-adic Lie
group for an odd prime $p$.
If $F$ is a finite field extension of $\Q_{p}$  with ring of integers $\mathcal{O}=\mathcal{O}_{F}$,
we put $\Lambda^{\mathcal{O}}(\mathcal{G}) := \mathcal{O} \otimes_{\Z_{p}} \Lambda(\mathcal{G}) = \mathcal{O} \llbracket \mathcal{G} \rrbracket $.
We choose a sufficiently large integer $n$ such that $\Gamma^{p^n}$ is central
in $\mathcal{G}$ and set $\Gamma_{0} := \Gamma^{p^n}$ which is an open central
subgroup of $\mathcal{G}$. Choose a topological generator $\gamma$ of $\Gamma$.
There is a ring isomorphism
$R:=\mathcal{O} \llbracket \Gamma_{0} \rrbracket  \simeq \mathcal{O} \llbracket t \rrbracket $ induced by $\gamma^{p^n} \mapsto 1+t$
where $\mathcal{O} \llbracket t \rrbracket $ denotes the power series ring in one variable over $\mathcal{O}$.
If we view $\Lambda^{\mathcal{O}}(\mathcal{G})$ as an $R$-module (or indeed as a left $R[H]$-module), there is a decomposition
\[
	\Lambda^{\mathcal{O}}(\mathcal{G}) = \bigoplus_{i=0}^{p^n-1} R[H] \gamma^{i}.
\]
Hence $\Lambda^{\mathcal{O}}(\mathcal{G})$ is finitely generated as an $R$-module and is an $R$-order in the separable $E:=Quot(R)$-algebra
$\mathcal{Q}^{F} (\mathcal{G})$, the total ring of fractions of $\Lambda^{\mathcal{O}}(\mathcal{G})$, obtained
from $\Lambda^{\mathcal{O}}(\mathcal{G})$ by adjoining inverses of all central regular elements.
Note that $\mathcal{Q}^{F} (\mathcal{G}) =  E \otimes_{R} \Lambda^{\mathcal{O}}(\mathcal{G})$ and that by
\cite[Lemma 1]{MR2114937} we have $\mathcal{Q}^{F} (\mathcal{G}) = F \otimes_{\Q_{p}} \mathcal{Q}(\mathcal{G})$,
where $\mathcal{Q}(\mathcal{G}) := \mathcal{Q}^{\Q_{p}}(\mathcal{G})$.

\subsection{Characters and central primitive idempotents} \label{subsec:idempotents}
Fix a character $\chi \in \Irr_{\Q_{p}^{c}}(\mathcal{G})$ 
(i.e.\ an irreducible $\Q_{p}^{c}$-valued character of $\mathcal{G}$ with open kernel)
and let $\eta$ be an irreducible constituent of
$\res^{\mathcal{G}}_{H} \chi$.
As in the case of finite groups, $\mathcal{G}$ acts on $\eta$ as 
$^{g}\eta(h) = \eta(g^{-1}hg)$
for $g \in \mathcal{G}$, $h \in H$, and we set
\[
\mathcal{G}_{\eta} := \{g \in \mathcal{G} \mid {}^{g}\eta = \eta \}, \quad e_{\eta} := \frac{\eta(1)}{|H|} \sum_{h \in H} \eta(h^{-1}) h,
\quad \varepsilon_{\chi} := \sum_{\eta \mid \res^{\mathcal{G}}_{H} \chi} e_{\eta}.
\]
Note that $\varepsilon_{\chi}$ was denoted $e(\eta)$ in previous sections,
but we now focus on its dependence on $\chi$ rather than $\eta$.
By \cite[Corollary to Proposition 6]{MR2114937} $\varepsilon_{\chi}$ 
is a primitive central idempotent of
$\mathcal{Q}^{c}(\mathcal{G}) := \Q_{p}^{c} \otimes_{\Q_{p}} \mathcal{Q}(\mathcal{G})$.
In fact, every primitive central idempotent of $\mathcal{Q}^{c}(\mathcal{G})$ is of this form
and $\varepsilon_{\chi} = \varepsilon_{\chi'}$ if and only if $\chi = \chi' \otimes \rho$ for some character $\rho$ of $\mathcal{G}$ of type $W$
(i.e.~$\res^{\mathcal{G}}_{H} \rho = 1$).
Let $w_{\chi} = [\mathcal{G} : \mathcal{G}_{\eta}]$ and note that this is a power of $p$ since $H$ is a subgroup of $\mathcal{G}_{\eta}$.

Let $F/\Q_{p}$ be a finite extension over which both characters $\chi$ and $\eta$ have realisations.
Let $V_{\chi}$ denote a realisation of $\chi$ over $F$.	
By \cite[Propositions 5 and 6]{MR2114937}, there exists a unique element $\gamma_{\chi} \in \zeta(\mathcal{Q}^{F}(\mathcal{G})\varepsilon_{\chi})$ 
such that $\gamma_{\chi}$ acts trivially on $V_{\chi}$ and $\gamma_{\chi} = g_{\chi}c_{\chi}$ with $g_{\chi} \in \mathcal{G}$ mapping to
$\gamma_{K}^{w_{\chi}} \bmod H$ and with $c_{\chi} \in (F[H]\varepsilon_{\chi})^{\times}$.
Moreover, $\gamma_{\chi}$ generates a pro-cyclic $p$-subgroup $\Gamma_{\chi}$ of $\mathcal{Q}^{F}(\mathcal{G})\varepsilon_{\chi}$ and induces an isomorphism $\mathcal{Q}^{F}(\Gamma_{\chi}) \stackrel{\simeq}{\longrightarrow} \zeta(\mathcal{Q}^{F} (\mathcal{G})\varepsilon_{\chi})$.
One should compare this to Lemma \ref{lem:center-of-semidirect}.

\subsection{Determinants and reduced norms}\label{subsec:dets-and-nr}
Following \cite[Proposition 6]{MR2114937}, we define a map
\[
j_{\chi}: \zeta(\mathcal{Q}^{F} (\mathcal{G})) \twoheadrightarrow \zeta(\mathcal{Q}^{F} (\mathcal{G})\varepsilon_{\chi}) \simeq \mathcal{Q}^{F}(\Gamma_{\chi}) \rightarrow  \mathcal{Q}^{F}(\Gamma_{K}),
\]
where the last arrow is induced by mapping $\gamma_{\chi}$ to $\gamma_{K}^{w_{\chi}}$.
It follows from op.\ cit.\ that $j_{\chi}$ is independent of the choice of $\gamma_{K}$ and that 
for every matrix $\Theta \in M_{n \times n} (\mathcal{Q}(\mathcal{G}))$ we have
\begin{equation} \label{eqn:jchi-det}
	j_{\chi} (\nr(\Theta)) = \mathrm{det}_{\mathcal{Q}^{F}(\Gamma_{K})} (\Theta \mid \Hom_{F[H]}(V_{\chi},  \mathcal{Q}^{F}(\mathcal{G})^n)).
\end{equation}
Here, $\Theta$ acts on $f \in \Hom_{F[H]}(V_{\chi},  \mathcal{Q}^{F}(\mathcal{G})^{n})$ via right multiplication,
and $\gamma_{K}$ acts on the left via $(\gamma_{K} f)(v) = \gamma \cdot f(\gamma^{-1} v)$ for all $v \in V_{\chi}$
which is easily seen to be independent of the choice of $\gamma$.
If $\rho$ is a character of $\mathcal{G}$ of type $W$,
then we denote by
$\rho^{\sharp}$ the automorphism of the field $\mathcal{Q}^{c}(\Gamma_{K})$ induced by
$\rho^{\sharp}(\gamma_{K}) = \rho(\gamma_{K}) \gamma_{K}$. 
Moreover, we denote the additive group generated by all $\Q_{p}^{c}$-valued
characters of $\mathcal{G}$ with open kernel by $R_p(\mathcal{G})$. 
If $\mathcal{Q}$ is a subring of $\mathcal{Q}^{c}(\Gamma_{K})$
which is closed under the actions of $G_{\Q_{p}}$ and $\rho^{\sharp}$
for all $\rho$, we let $\Maps^{\ast}(\Irr_{\Q_{p}^{c}}(\mathcal{G}), \mathcal{Q})$ (resp.\
$\Hom^{\ast}(R_{p}( \mathcal{G}), \mathcal{Q}^{\times})$)
be the set of all maps $f: \Irr_{\Q_{p}^{c}}(\mathcal{G}) \rightarrow 
\mathcal{Q}$ (resp.\
the group of homomorphisms 
$f: R_p(\mathcal{G}) \rightarrow \mathcal{Q}^{\times}$) satisfying
\[
\begin{array}{ll}
	f(\chi \otimes \rho) = \rho^{\sharp}(f(\chi)) & \mbox{ for all characters } \rho \mbox{ of type } W \mbox{ and}\\
	f({}^{\sigma}\chi) = \sigma(f(\chi)) & \mbox{ for all Galois automorphisms } \sigma \in G_{\Q_{p}}.
\end{array}
\]
By \cite[Proof of Theorem 8]{MR2114937} we have isomorphisms
\begin{eqnarray*}
	\zeta(\mathcal{Q}(\mathcal{G})) & \simeq & 
	\Maps^{\ast}(\Irr_{\Q_{p}^{c}}(\mathcal{G}), \mathcal{Q}^{c}(\Gamma_{K}))\\
	\zeta(\mathcal{Q}(\mathcal{G}))^{\times} & \simeq & 
	\Hom^{\ast}(R_{p}(\mathcal{G}), \mathcal{Q}^{c}(\Gamma_{K})^{\times})\\
	x & \mapsto & [\chi \mapsto j_{\chi}(x)].
\end{eqnarray*}
Let $\Z_p^c$ be the integral closure of $\Z_p$ in $\Q_p^c$.
Let $\mathcal{M}(\mathcal{G})$ be a maximal $R$-order in 
$\mathcal{Q}(\mathcal{G})$ containing $\Lambda(\mathcal{G})$.
Then the first isomorphism restricts to
\begin{equation} \label{eqn:center-of-max}
	\zeta(\mathcal{M}(\mathcal{G})) \simeq
	\Maps^{\ast}(\Irr_{\Q_{p}^{c}}(\mathcal{G}), \Z_p^c \otimes_{\Z_p}\Lambda(\Gamma_{K}))
\end{equation}
by \cite[Remark H]{MR2114937}. Since $\Q_p \otimes_{\Z_p} \Lambda(\mathcal{G})
= \Q_p \otimes_{\Z_p} \mathcal{M}(\mathcal{G})$, we finally obtain
an isomorphism
\[
	\zeta(\Q_p \otimes_{\Z_p} \Lambda(\mathcal{G})) \simeq
	\Maps^{\ast}(\Irr_{\Q_{p}^{c}}(\mathcal{G}), \Q_p^c \otimes_{\Z_p}\Lambda(\Gamma_{K})).
\]

For each totally odd irreducible character $\chi$ we set
\[
	g_{\chi, S}^{T} := j_{\chi}(\theta_{S}^T(L_{\infty}/K,0)) \in
	\Q_{p}^c \otimes_{\Z_p} \Lambda(\Gamma_K),
\]
which is well-defined by \eqref{eqn:infinite-Stickelberger}.
Let $\rho$ be a character of type $W$.
We may view $g_{\chi, S}^{T}$ as a power series in $t$ via $\gamma_K \mapsto 1+t$.
Evaluating at $\rho(\gamma_K)-1$ gives
\begin{equation} \label{eqn:evaluation-S-T}
	g_{\chi, S}^{T}(\rho(\gamma_K)-1) = g_{\chi \otimes \rho, S}^{T}(0) =
	\delta_T(0, \chi \otimes \rho)L_S(0, \check\chi \otimes \check \rho).
\end{equation}
If $v$ is a place not in $S \cup T$ and $T' = T \cup \{v\}$, then
\[
	\theta_{S}^{T'}(L_{\infty}/K,0) = \delta_v \cdot
	\theta_{S}^T(L_{\infty}/K,0), \quad \delta_v := \nr(1- (Nv) \sigma^{-1}_{w_{\infty}}),
\]
where $\sigma_{w_{\infty}}$ denotes the Frobenius automorphism
at a place $w_{\infty}$ in $L_{\infty}$ above $v$. 
Therefore we may define
\[
	\theta_{S}(L_{\infty}/K,0) :=  \left(\prod_{v \in T} 
	\delta_v^{-1}\right) \theta_{S}^T(L_{\infty}/K,0)
\]
and likewise $g_{\chi, S} := j_{\chi}(\theta_{S}(L_{\infty}/K,0))$.
Then \eqref{eqn:evaluation-S-T} implies that
\begin{equation} \label{eqn:evaluation_S}
	g_{\chi, S}(\rho(\gamma_K)-1) = g_{\chi \otimes \rho, S}(0) =
	L_S(0, \check\chi \otimes \check \rho).
\end{equation}
If $\chi$ is a not necessarily irreducible character of $\mathcal{G}$
with open kernel, we write $\chi$ as a finite sum 
$\chi = \sum_i n_i \chi_i$ with positive integers $n_i$ and
$\chi_i \in \Irr_{\Q_{p}^{c}}(\mathcal{G})$, and define
$g_{\chi,S}^T := \prod_i (g_{\chi_i,S}^T)^{n_i}$ and similarly if $T = \emptyset$. Note that \eqref{eqn:evaluation-S-T} and \eqref{eqn:evaluation_S}
remain valid.

Let $N$ be a finite normal subgroup of $\mathcal{G}$ such that
$\tau \not\in N$. Put
$\overline{\mathcal{G}} := \mathcal{G}/N$.
Suppose that $\chi = \infl_{\overline{\mathcal{G}}}^{\mathcal{G}}
\overline{\chi}$ for some character $\overline{\chi}$
of $\overline{\mathcal{G}}$. As Artin $L$-series behave well under
inflation of characters, we have that
\begin{equation} \label{eqn:behaviour-inflation}
	g_{\chi,S}^T = g_{\overline{\chi},S}^T \mbox{ and }
	g_{\chi,S} = g_{\overline{\chi},S} \mbox{ in } \mathcal{Q}^c(\Gamma_K).
\end{equation}

Let $\mathcal{H}$ be an open subgroup of $\mathcal{G}$,
which contains $\tau$. Set
$K' := L_{\infty}^{\mathcal{H}}$. Restricting $\gamma' \in \Gamma_{K'}$
to $K_{\infty}$ induces a natural
embedding $\Gamma_{K'} \hookrightarrow \Gamma_K$. We obtain an inclusion  
of $\mathcal{Q}^c(\Gamma_{K'})$ into $\mathcal{Q}^c(\Gamma_K)$. Now let 
$\chi'$ be a character of $\mathcal{H}$ and set
$\chi := \ind_{\mathcal{H}}^{\mathcal{G}} \chi'$. 
Let $S' := S(K')$ and $T' := T(K')$.
As Artin $L$-series behave well
under induction of characters, we have an equality
\begin{equation} \label{eqn:behaviour-induction}
	g_{\chi,S}^T = g_{\chi',S'}^{T'} \mbox{ and }
	g_{\chi,S} = g_{\chi',S'} \mbox{ in } \mathcal{Q}^c(\Gamma_K).
\end{equation}
We refer the reader to \cite[Lemma 7.8 and Proposition 7.9]{Nils_thesis}
for more details.

\subsection{The $p$-adic cyclotomic character and its projections}\label{subsec:cyclotomic-char}
Let $\chi_{\mathrm{cyc}}$ be the $p$-adic cyclotomic character
\[
\chi_{\mathrm{cyc}}: \Gal(L_{\infty}(\zeta_{p})/K) \longrightarrow \Z_{p}^{\times},
\]
defined by $\sigma(\zeta) = \zeta^{\chi_{\mathrm{cyc}}(\sigma)}$ for any $\sigma \in \Gal(L_{\infty}(\zeta_{p})/K)$ and any $p$-power root of unity $\zeta$.
Let $\omega$ and $\kappa$ denote the composition of $\chi_{\mathrm{cyc}}$ with the projections onto the first and second factors of the canonical decomposition $\Z_{p}^{\times} = \mu_{p-1} \times (1+p\Z_{p})$, respectively;
thus $\omega$ is the Teichm\"{u}ller character.
We note that $\kappa$ factors through $\Gamma_{K}$ 
(and thus also through $\mathcal{G}$) and by abuse of notation we also 
use $\kappa$ to denote the associated maps with these domains.
We put $u := \kappa(\gamma_{K})$.

\subsection{Power series and $p$-adic Artin $L$-functions}\label{subsec:power-series-p-adic-L-functions}
Fix a totally even character $\psi \in \Irr_{\Q_{p}^{c}}(\mathcal{G})$. 
We set $\mathcal{G}^+ = \mathcal{G} / \langle \tau \rangle$ and view $\psi$
as a character of $\mathcal{G}^+$.
Each topological generator $\gamma_{K}$ of  $\Gamma_{K}$ permits the definition of a 
power series $G_{\psi,S}(t) \in \Q_{p}^{c} \otimes_{\Q_{p}} Quot(\Z_{p} \llbracket t \rrbracket )$ 
by starting out from the Deligne-Ribet power series for one-dimensional characters of open subgroups 
of $\mathcal{G}$ (see \cite{MR579702}; also see \cite{ MR525346, MR524276}) 
and then extending to the general case by using Brauer induction (see \cite{MR692344}).
One then has an equality
\[
L_{p,S}(1-s,\psi) = \frac{G_{\psi,S}(u^s-1)}{H_{\psi}(u^s-1)},
\]
where $L_{p,S}(s,\psi)$ denotes the `$S$-truncated $p$-adic Artin $L$-function' attached to $\psi$ constructed by Greenberg \cite{MR692344},
and where, for irreducible $\psi$, one has
\[
H_{\psi}(t) = \left\{\begin{array}{ll} \psi(\gamma_{K})(1+t)-1 & \mbox{ if }  H \subseteq \ker \psi\\
	1 & \mbox{ otherwise.}  \end{array}\right.
\]
Recall that $L_{p,S}(s, \psi) : \Z_{p} \rightarrow \C_{p}$ 
is the unique $p$-adic meromorphic
function with the property that for each strictly negative 
integer $r$ and each field isomorphism $\iota : \C \simeq \C_{p}$, we have
\begin{equation}\label{eqn:interpolation-property}
	L_{p,S}(r, \psi) = \iota \left( L_{S}(r, \iota^{-1} \circ (\psi \otimes \omega^{r-1})) \right).
\end{equation}
By Siegel's result \cite{MR0285488}, the right-hand side of  \eqref{eqn:interpolation-property}
does not depend on the choice of $\iota$.
By a slight abuse of notation, we simply write
$L_{S}(r,\psi \otimes \omega^{r-1})$ for the right hand side in
\eqref{eqn:interpolation-property} in the following.
If $\psi$ is linear, then \eqref{eqn:interpolation-property} is also valid when $r=0$.

Now \cite[Proposition 11]{MR2114937} implies that
\[
L_{K,S} : \psi \mapsto \frac{G_{\psi,S}(\gamma_{K}-1)}{H_{\psi}(\gamma_{K}-1)}
\]
is independent of the topological generator $\gamma_{K}$ and lies in 
$\Hom^{\ast}(R_{p}( \mathcal{G}^+), \mathcal{Q}^{c}(\Gamma_{K})^{\times})$.

\subsection{Comparing the power series}
We now relate the power series $g_{\chi,S}$ of a totally odd character $\chi$
to the power series of the previous subsection associated to the even character
$\check\chi \omega$.

\begin{theorem} \label{thm:comparing-series}
	Let $p$ be an odd prime.
	Let $L/K$ be a Galois CM extension and let $L_{\infty}$ be the cyclotomic
	$\Z_p$-extension of $L$. Let $\chi$ be a totally odd irreducible character
	of $\mathcal{G} = \Gal(L_{\infty}/K)$ with open kernel. 
	Then for each finite set of places $S$ of $K$
	containing all archimedean places and all places that ramify in $L_{\infty}/K$
	we have an equality
	\begin{equation} \label{eqn:power-series-equality}
		g_{\chi,S} = 
		\frac{G_{\check\chi \omega,S}(u(1+t)^{-1}-1)}{H_{\check\chi \omega}(u(1+t)^{-1}-1)}.
	\end{equation}
\end{theorem}

\begin{proof}
	Let us denote the right hand side of equation \eqref{eqn:power-series-equality}
	by $f_{\chi,S}$. Let $\rho$ be a character of type $W$. 
	Since $L_{K,S}$ lies in 
	$\Hom^{\ast}(R_{p}( \mathcal{G}^+), \mathcal{Q}^{c}(\Gamma_{K})^{\times})$,
	an easy computation shows that also $\chi \mapsto f_{\chi,S}$
	is compatible with $\rho$-twists. See \cite[Lemma 7.13]{Nils_thesis}
	for details.
	
	Now assume that $\chi$ is a linear character. Then we compute
	\begin{align*}
		f_{\chi,S}(\rho(\gamma_K)-1) & =  f_{\chi \otimes \rho, S}(0)
		= \frac{G_{(\check\chi \otimes \check\rho)\omega, S}(u-1)}{H_{(\check\chi \otimes \check\rho) \omega}(u-1)}\\
		& = L_{p,S}(0,(\check\chi \otimes \check\rho) \omega) = L_S(0, \check\chi \otimes \check\rho)\\
		& = g_{\chi,S}(\rho(\gamma_K)-1),
	\end{align*}
	where the last equality is \eqref{eqn:evaluation_S}. Then an application
	of \cite[Corollary 7.4]{MR1421575} shows that $f_{\chi,S} = g_{\chi,S}$.
	The equation \eqref{eqn:power-series-equality} is actually well-known
	for linear characters and we refer the reader to 
	\cite[\S 5]{MR3383600} and \cite[\S 7]{Nils_thesis} for 
	more details. Note that both $f_{\chi,S}$ and $g_{\chi,S}$ behave well
	under induction and inflation of characters. For $f_{\chi,S}$ this 
	actually holds by Greenberg's construction \cite[\S 2]{MR692344}
	(also see \cite[Proposition 12]{MR2114937}). 
	For $g_{\chi,S}$ this holds by \eqref{eqn:behaviour-induction}
	and \eqref{eqn:behaviour-inflation}, respectively.
	Hence the general case follows from Brauer's
	induction theorem.
\end{proof}

We now draw a few important consequences. Our first observation is
that we obtain a direct construction of $p$-adic $L$-series
of totally even characters.

\begin{corollary}
	Let $\psi \in \Irr(\mathcal{G})$ be a totally even character.
	Then one has 
	\[
		L_{p,S}(s, \psi) = g_{\check{\psi}\omega, S}(u^s-1).
	\]
\end{corollary}

As a second application we give an affirmative answer to a question
raised by Greenberg \cite[\S 4]{MR692344}, which is also a 
special case of a conjecture of Gross \cite{MR656068}.

\begin{corollary} \label{cor:Gross-conjecture}
	Let $\psi \in \Irr(\mathcal{G})$ be a totally even character.
	Then one has an equality
	\[
		L_{p,S}(0,\psi) = L_S(0, \psi \omega^{-1}).
	\]
\end{corollary}

\begin{proof}
	We apply Theorem \ref{thm:comparing-series} to the character
	$\chi= \check\psi\omega$. Evaluating $g_{\chi,S}$ at $t=0$ yields
	$L_S(0, \check\chi) = L_S(0, \psi \omega^{-1})$ 
	by \eqref{eqn:evaluation_S} with $\rho$ equal to
	the trivial character. Evaluating the right hand side of 
	\eqref{eqn:power-series-equality} at $t=0$ gives
	$L_{p,S}(0, \psi)$ as desired.
\end{proof}

\begin{remark}
	In \cite{MR3980291} Johnston and the second named author deduced cases
	of the non-abelian Brumer--Stark conjecture from the equivariant Iwasawa
	main conjecture. In many of these results it had been assumed that
	the underlying group is monomial. However, this was only required to guarantee
	the conclusion of Corollary \ref{cor:Gross-conjecture}.
	Since this now holds unconditionally for an arbitrary totally
	even character, one can drop the hypothesis of being monomial
	in the respective results.
\end{remark}

Finally, we obtain a new proof of the $p$-adic Artin conjecture
\cite[p.\ 82]{MR692344}, which does not rely on the main conjecture
proved by Wiles \cite{MR1053488}.

\begin{corollary}[The $p$-adic Artin conjecture] \label{cor:p-adic-Artin}
	Let $p$ be an odd prime and let $K$ be a totally real number field.
	Let $\psi$	be an irreducible totally even character of $G_K$. Then
	for each finite set $S$ of places of $K$ containing $S_{\infty}
	\cup S_p$ and all places where $\psi$ is ramified, we have
	\[
		G_{\psi,S}(t) \in \Q_p^c \otimes_{\Z_p} \Z_p\llbracket t \rrbracket.
	\]
\end{corollary}

\begin{proof}
	By Theorem \ref{thm:comparing-series} we have
	$G_{\psi,S} (t) = H_{\psi}(t) \cdot g_{\chi,S}(u(1+t)^{-1}-1)$,
	where we set $\chi := \check\psi \omega$. The main obstacle
	now is that we only know that $g_{\chi,S}^T(t) \in 
	\Q_p^c \otimes_{\Z_p} \Z_p\llbracket t \rrbracket$ for any
	set $T$ of places of $K$ such that $\Hyp(S,T)$ is satisfied for
	all layers $L_n/K$. We may choose $T = T_v := \{v\}$, 
	where $v$ is a place 
	of $K$ of sufficiently large norm. 
	For each such $v$ we then have an equality
	\[
		g_{\chi,S}^{T_v} = j_{\chi}(\delta_v) g_{\chi,S}.
	\]
	Moreover, there is a short exact sequence of 
	$\Q_p^c \otimes_{\Z_p} \Lambda(\mathcal{G})$-modules
	\begin{equation} \label{eqn:resolution-of-TateTwist}
		0 \rightarrow \Q_p^c \otimes_{\Z_p} \Lambda(\mathcal{G})
		\rightarrow \Q_p^c \otimes_{\Z_p} \Lambda(\mathcal{G})
		\rightarrow \ind_{\mathcal{G}_{w_{\infty}}}^{\mathcal{G}} \Q_p^c(1)
		\rightarrow 0,
	\end{equation}
	where the first non-trivial arrow is multiplication by
	$(Nv)\sigma_{w_{\infty}}^{-1} - 1$. Here $\sigma_{w_{\infty}}$ denotes
	again the Frobenius automorphism at a place $w_{\infty}$ in 
	$L_{\infty}$ above $v$, which is also a topological generator
	of the decomposition group $\mathcal{G}_{w_{\infty}}$.
	
	We now consider the case where $H$ is not contained in the
	kernel of $\psi$. Then the Teichmüller character does not occur
	as an irreducible constituent of $\res^{\mathcal{G}}_H \chi$, and
	there is an $h \in H$
	such that $\psi(h) \not= \psi(1)$. By  Chebotarev's
	density theorem we may choose $v$ such that
	$\sigma_{w_{\infty}} = h \gamma \in H \rtimes \Gamma$ 
	for some $\gamma \in \Gamma$. 
	The primitive central idempotents $\varepsilon_{\chi}$ lie in
	$\Q_p^c \otimes_{\Z_p} \Lambda(\mathcal{G})$ and taking
	`$\chi$-parts' of sequence \eqref{eqn:resolution-of-TateTwist}
	shows that $((Nv)\sigma_{w_{\infty}}^{-1} - 1)\varepsilon_{\chi}$
	is invertible in 
	$\Q_p^c \otimes_{\Z_p} \Lambda(\mathcal{G})\varepsilon_{\chi}$.
	Hence $j_{\chi}(\delta_v) \in
	(\Q_{p}^c \otimes_{\Z_p} \Lambda(\Gamma_K))^{\times}$ and we are done
	in this case.
	
	Now assume that $H$ is contained in the kernel of $\psi$ so that
	$\psi$ is a character of type $W$.
	As $G_{\psi\otimes\rho,S} = \rho^{\sharp}(G_{\psi,S})$
	for any totally even $\psi$ and any $\rho$ of type $W$, we may assume that
	$\psi$ is the trivial character. Now choose $v$ such that
	$\sigma_{w_{\infty}}$ is a topological generator of $\Gamma$.
	By replacing $\gamma_K$ if necessary,
	we may assume in addition that $\sigma_{w_{\infty}}$ maps to $\gamma_K$
	under $\mathcal{G} \twoheadrightarrow \Gamma_{K}$.
	Then we get $j_{\chi}(\delta_v)(t) = u(1+t)^{-1} - 1$ and thus
	$j_{\chi}(\delta_v)(u(1+t)^{-1}-1) = t = H_{\psi}(t)$
	as desired.
\end{proof}

The main conjecture actually implies the stronger statement that
$G_{\psi,S} \in \Z_p^c \otimes_{\Z_p}\Z_p \llbracket t \rrbracket$,
see \cite[Remark G]{MR2114937}. The proof of Corollary
\ref{cor:p-adic-Artin} and \eqref{eqn:center-of-max}
then imply the following result in the case $r=0$.
The case $r<0$ can be shown similarly (indeed the proof simplifies
as \eqref{eqn:interpolation-property} was already known to hold for
non-linear $\chi$).

\begin{corollary}
	Let $p$ be an odd prime.
	Let $L/K$ be a Galois CM extension and let $L_{\infty}$ be the cyclotomic
	$\Z_p$-extension of $L$. Let $\mathcal{G} = \Gal(L_{\infty}/K)$
	and choose a maximal $R$-order $\mathcal{M}(\mathcal{G})$ containing 
	$\Lambda(\mathcal{G})$. Let $S$ and $T$ be finite sets of places
	of $K$ such that $\Hyp(S,T)$ is satisfied for all layers $L_n/K$.
	Then for each integer $r \leq 0$ one has
	\[
		\theta_S^T(L_{\infty}/K,r) \in \zeta(\mathcal{M}(\mathcal{G})).
	\]
\end{corollary}

\bibliography{integrality_Bib}{}
\bibliographystyle{amsalpha}

\end{document}